\documentclass[11pt]{amsart}%{article}
\usepackage[dvipdfm]{graphicx, color}
\usepackage{amssymb, amsmath,amsthm}
\newtheorem{thm}{{\bf Theorem}}[section]
\newtheorem{lem}[thm]{{\bf Lemma}}
\newtheorem{cor}[thm]{{\bf Corollary}}
\newtheorem{prop}[thm]{{\bf Proposition}}

\newtheorem{claim}[thm]{Claim} 
\newtheorem{rem}[thm]{Remark}

\newtheorem{ques}[thm]{Question}
\newtheorem{conj}[thm]{Conjecture}

\numberwithin{equation}{section}

\begin{document} 

\title[Minimal dilatations and their asymptotic behavior]{
	Minimal dilatations of pseudo-Anosovs generated by 
	the magic $3$-manifold and their asymptotic behavior
}
	
\author[E. Kin]{%
    Eiko Kin
}
\address{%
       Department of Mathematics, Graduate School of Science, Osaka University Toyonaka, Osaka 560-0043, JAPAN
%       Department of Mathematical and Computing Sciences \\
%       Tokyo Institute of Technology \\
%       Ohokayama, Meguro \\
%       Tokyo 152-8552 Japan
}
\email{%
        kin@math.sci.osaka-u.ac.jp
}

\author[S. Kojima]{%
    Sadayoshi Kojima
}
\address{%
        Department of Mathematical and Computing Sciences \\
        Tokyo Institute of Technology \\
        Ohokayama, Meguro \\
        Tokyo 152-8552 Japan
}
\email{%
        sadayosi@is.titech.ac.jp
}

\author[M. Takasawa ]{%
    Mitsuhiko Takasawa
}
\address{%
        Department of Mathematical and Computing Sciences \\
        Tokyo Institute of Technology \\
        Ohokayama, Meguro \\
        Tokyo 152-8552 Japan
}
\email{%
        takasawa@is.titech.ac.jp
}



\date{%
	\today
	}

\thanks{%
	The first author is partially supported by 
	Grant-in-Aid for Young Scientists (B) (No. 20740031), MEXT, Japan, 
	and 
	the second author is partially supported by 
	Grant-in-Aid (A) (No. 22244004), JSPS, Japan.  
	} 

\begin{abstract} 
This paper concerns the set $\widehat{\mathcal{M}}$ of pseudo-Anosovs 
which occur as monodromies of fibrations on manifolds
obtained from the magic $3$-manifold $N$ by Dehn filling three cusps  with a mild restriction. 
Let  $N(r)$ be the manifold obtained from $N$ by Dehn filling 
one cusp along the slope  $r \in {\Bbb Q} $. 
We prove that for each $g$ (resp. $g \not\equiv 0 \pmod{6}$), 
the minimum among dilatations of elements (resp. elements with orientable invariant foliations) 
of $\widehat{\mathcal{M}}$ defined on   a closed surface $\varSigma_g$ of genus $g$ 
is achieved by the monodromy of some $\varSigma_g$-bundle over the circle 
obtained from $N(\tfrac{3}{-2})$ or $N(\tfrac{1}{-2})$ by Dehn filling both cusps. 
These minimizers are the same ones identified  by Hironaka, Aaber and Dunfield, Kin and Takasawa independently. 
In the case $g \equiv 6 \pmod{12}$ we find a new family of pseudo-Anosovs defined on $\varSigma_g$ 
with orientable invariant foliations obtained from $N(-6)$ or $N(4)$ by Dehn filling both cusps.  
We prove that 
if $\delta_g^+$ is  the minimal dilatation of pseudo-Anosovs with orientable invariant foliations 
defined on $\varSigma_g$, then 
$$ \limsup_{\substack{g \equiv 6 \pmod{12} \\ g \to \infty}}  g  \log \delta^+_g  \le 2 \log  \delta(D_5) \approx 1.0870,$$
where $\delta(D_n)$ is the minimal dilatation of pseudo-Anosovs on an $n$-punctured disk. 
We  also study monodromies of fibrations on $N(1)$. 
We prove that if $\delta_{1,n}$ is the minimal dilatation of pseudo-Anosovs on a genus $1$ surface with $n$ punctures, 
then $$ \limsup_{n \to \infty} n \log \delta_{1,n} \le 2 \log \delta(D_4) \approx 1.6628.  $$
\end{abstract} 
\maketitle

\section{Introduction}\label{section_Introduction}

\subsection{Minimal dilatations of pseudo-Anosovs}

Let  $\mathrm{Mod}(\varSigma)$  be the mapping class group of 
a connected oriented surface  $\varSigma$, 
and let $\phi \in \mathrm{Mod}(\varSigma)$ be a pseudo-Anosov class. 
Then  $\phi \in \mathrm{Mod}(\varSigma)$  contains as a representative 
a pseudo-Anosov homeomorphism  $\Phi: \varSigma \rightarrow \varSigma$ 
equipped with a constant  $\lambda= \lambda(\Phi)>1$  called the dilatation of  $\Phi$. 
%equipped with a constant  $\lambda= \lambda(\Phi)>1$  called the dilatation of  $\Phi$, see \cite{CB,Thurston2}. 
The dilatation  $\lambda(\phi)$ of $\phi$  is defined to be  $\lambda(\Phi)$. 
The  topological entropy  $\mathrm{ent}(\Phi)$ of $\Phi$ is equal to  $ \log \lambda(\Phi)$, 
and  $\mathrm{ent}(\Phi)$  attains the minimal entropy among all  homeomorphisms  
which are isotopic to  $\Phi$  \cite[Expos\'{e} 10]{FLP}.  
We denote  this characteristic number by  $\mathrm{ent}(\phi)$, 
and call it the {\it entropy} of  $\phi$. 
We call   $\mathrm{Ent}(\phi)= |\chi(\varSigma)| \, \mathrm{ent}(\phi)$ the {\it  normalized entropy} of  $\phi$. 

If we fix  $\varSigma$, 
\begin{equation*} 
	\{\mathrm{ent}(\phi)\ |\ \phi \in \mathrm{Mod}(\varSigma)\  \mbox{is pseudo-Anosov}\} 
\end{equation*}  
is a closed discrete subset of  ${\Bbb R}$, 
see \cite{Ivanov}. 
In particular there exists the minimum  $\delta(\varSigma)$ among dilatations of  
pseudo-Anosov elements in  $\mathrm{Mod}(\varSigma)$. 
The explicit values of  $\delta(\varSigma)$  were 
computed in several cases where   $|\chi(\varSigma)|$ is 
small  \cite{CH,HS,KLS,LT1}. 
For example, if $D_n$ is an $n$-punctured disk, then 
$\delta(D_3)=  \tfrac{3+ \sqrt{5}}{2} \approx 2.6180$, 
$\delta(D_4) \approx 2.2966$ is equal to  the largest real root of $t^4 -2t^3 - 2t+1$, and 
$\delta(D_5) \approx 1.7220$ is equal to the largest real root of $t^5 - 2t^3 -2t^2+1$. 
However, 
it is widely open to determine $\delta(\varSigma)$ for  most surfaces $\varSigma$. 

Let  $\varSigma_g$  be a closed surface of genus $g$, and let 
$\varSigma_{g,n}$ be  a compact surface of genus  $g$  with  $n$  boundary components. 
We  set  $\delta_g= \delta(\varSigma_g)$ and $\delta_{g,n}= \delta(\varSigma_{g,n})$.  
Penner proved that $\log \delta_g \asymp \tfrac{1}{g}$ \cite{Penner}. 
It is an open problem to compute  $\delta_g$  for $g >2$, but 
some partial results are known. 
Let $\delta_g^+$ be the minimal dilatation of pseudo-Anosov homeomorphisms on $\varSigma_g$ 
with orientable invariant foliations. 
The explicit values $\delta_g^+$ are known for all $2 \le g \le 8$ 
except for $g = 6$, see \cite{Zhirov,LT,AD,Hironaka,KT1}.

We are motivated by the following question, posed by McMullen, 
which asks about the asymptotic behavior of the sequence $\{\delta_g\}_{g \ge 2}$. 

\begin{ques}[\cite{McMullen}]
\label{ques_M}
Does $\displaystyle \lim_{g \to \infty}g  \log \delta_g$ exist? What is its value?
\end{ques} 

\noindent
It was proved 
by Minakawa \cite{Minakawa} and independently by Hironaka-Kin \cite{HK} that 
$\log \delta_g^+ \asymp \tfrac{1}{g}$, 
and by Tsai  \cite{Tsai,Tsai1} that  $\log \delta_{1,n} \asymp \tfrac{1}{n}$.   
Thus we can also ask:

\begin{ques}
\label{ques_M2}
Does $\displaystyle \lim_{g \to \infty}g  \log \delta_g^+$ or $\displaystyle \lim_{n \to \infty}n  \log \delta_{1,n}$ 
 exist? What is its  value?
\end{ques}

\noindent 
Penner's lower bound on  $\delta_{g,n}$ in  \cite{Penner} gives  
a uniform lower bound 
$\tfrac{\log 2}{12} < g \log \delta_g \le g \log \delta^+_g$ 
and 
$\tfrac{\log 2}{4} \le n \log \delta_{1,n}$.

The purpose of this paper is to provide 
not a complete but a considerably sharp answer to 
Questions~\ref{ques_M} and  \ref{ques_M2}.   
To explain what we prove and 
why we believe it is very close to the sharp answer more precisely,  
we would like to give a rather long introduction.

\subsection{Thurston norm and fibered $3$-manifolds}
\label{subsection_Tnorm}

Let $M$ be an  oriented  $3$-manifold 
with boundary $\partial M$ (possibly $\partial M = \emptyset$). 
Thurston discovered a pseudo-norm 
\begin{equation*} 
	\| \cdot \|: H_2(M, \partial M; {\Bbb R}) \rightarrow {\Bbb R}. 
\end{equation*}   
When  $M$ is a hyperbolic  $3$-manifold,  $\| \cdot \|$  becomes a norm.  
Moreover when  $M$  fibers over the circle,   
he described a relation between  $\| \cdot \| $ and fibrations on $M$  
as we recall below. 
(For more details, see \cite{Thurston1}.) 

The Thurston norm 
$\| \cdot \|$  is defined for 
an integral class  $a \in H_2(M, \partial M; {\Bbb Z})$  by 
\begin{equation*} 
	\| a\|= \min_F \{- \chi(F)\}, 
\end{equation*} 
where the minimum is taken over 
all oriented surfaces  $F$  embedded in $M$, satisfying  $a= [F]$,  
with no components  of non-negative Euler characteristic. 
The surface  $F$  which realizes this minimum is called 
the {\it minimal representative} of $a $, denoted by  $F_a$. 
The norm  $\| \cdot \|$  defined on integral  classes 
admits a unique continuous extension 
$\| \cdot \|: H_2(M, \partial M; {\Bbb R})  \rightarrow {\Bbb R}$  which is 
linear on rays through the origin. 
The unit ball  $U_M= \{a \in H_2(M, \partial M; {\Bbb R})\ |\ \|a\| \le 1\}$  is 
a compact, convex polyhedron \cite{Thurston1}. 

Suppose that $M$ is a surface bundle over the circle and 
let  $F$  be its fiber. 
The fibration determines a cohomology 
class  $a^* \in H^1(M; \mathbb{Z}) \cong [M, S^1]$,   
and hence a homology class  $a \in H_2(M, \partial M; \mathbb{Z})$  
by Poincar\'e duality.  
Thurston proved in \cite{Thurston1} that 
there exists a top dimensional face  $\Omega$ on $\partial U_M$  such that 
$[F]$  is an integral class of  $int(C_{\Omega})$, 
where  $C_{\Omega}$  is the cone over  $\Omega$  with the origin and 
$int(C_{\Omega})$  is its interior.   
Moreover he proved that 
the minimal representative  $F_a$  
for {\it any} integral class  $a$  in  $int(C_{\Omega})$   
becomes a fiber of the fibration associated to  $a$.  
Because of this result, 
$\Omega$  is called a {\it fibered face} of  $M$,  
and an integral class  $a \in int (C_{\Omega})$  is called a {\it fibered class}.  
This property  shows that 
if a hyperbolic 3-manifold with the second Betti number more 
than  $1$  admits a fibration over the circle, 
then it admits an infinite family of fibrations over the ciecle.   
If $a \in int(C_{\Omega})$  is a primitive integral class, then  
the associated fibration on  $M$  has a connected fiber represented by $F_a$.  
Since  $M$  is hyperbolic,  
the mapping class $\phi_a = [\Phi_a]$ of the monodromy $\Phi_a: F_a \rightarrow F_a$ 
is pseudo-Anosov due to 
the hyperbolization theorem by Thurston~\cite{Thurston3}. 
In particular, 
a single fibered $3$-manifold could offer 
infinitely many pseudo-Anosovs defined on surfaces with variable topology.

Let us fix a fibered face  $\Omega$  of  $M$. 
The set of integral classes (hence fibered classes) and rational classes 
of  $int(C_{\Omega})$  are denoted by 
$int(C_{\Omega}({\Bbb Z}))$  and  $int(C_{\Omega}({\Bbb Q}))$  respectively. 
Let   $a \in int(C_{\Omega}(\mathbb{Z}))$  be a  primitive class.  
The {\it dilatation}  $\lambda(a)$  and  {\it entropy}  
$\mathrm{ent}(a) = \log \lambda(a)$  are defined as the dilatation 
and entropy of the pseudo-Anosov mapping class $\phi_a$ respectively.
The entropy defined on primitive fibered classes is extended to 
rational classes as follows:
for a rational number $r $ and a primitive fibered class $a$, 
the entropy  $\mathrm{ent}(ra)$ is defined 
by  $\frac{1} {|r|}  \mathrm{ent}(a)$.  
Fried proved that  
$\tfrac{1}{\mathrm{ent}}: int(C_{\Omega}({\Bbb Q})) \rightarrow {\Bbb R}$ 
is concave \cite{Fried}, 
and in particular 
$\mathrm{ent}: int(C_{\Omega}({\Bbb Q})) \rightarrow {\Bbb R}$  
admits a unique continuous extension 
\begin{equation*} 
	\mathrm{ent}: int(C_{\Omega}) \rightarrow {\Bbb R}.  
\end{equation*}  
Moreover, 
Fried proved that the restriction of  $\mathrm{ent}$  to 
the open fibered face  $int(\Omega)$  is proper, 
namely, 
$\mathrm{ent}(a)$  goes to  $ \infty$  as  $a$  goes to 
a point on the boundary  $\partial \Omega$.  
Note that  
$\frac{1}{\mathrm{ent}} : int(C_{\Omega}) \rightarrow {\Bbb R}$  is linear along each ray through the origin 
and cannot be strictly concave for this direction, 
but it is actually strictly concave for other directions.  
This refinement of concavity was proved originally 
by Matsumoto \cite{Matsumoto} and 
later  by McMullen \cite{McMullen}.    
The strict concavity of  $\frac{1}{\mathrm{ent}}$  on   $int({\Omega})$  implies 
that  $\mathrm{ent}$  is 
strictly convex on  $int({\Omega})$  because  $\mathrm{ent}$  is 
positive valued.  
Now, 
by the definition of  $\mathrm{ent}$, 
we see that  
\begin{equation*} 
	\mathrm{Ent}= \|\cdot \| \, \mathrm{ent} : int(C_{\Omega}) \to \mathbb{R} 
\end{equation*} 
becomes constant on each ray in  $int(C_{\Omega})$  through the origin.  
We call $\mathrm{Ent}(a)$  the  {\it normalized entropy} 
of $a \in int(C_{\Omega})$. 
Since  $\| \cdot \|$  is constant on a fibered face  $\Omega$, 
the normalized entropy  $\mathrm{Ent}$  is still strictly convex 
on  $int(\Omega)$.  
Thus because of the properness of  $\mathrm{ent}$  by Fried, 
$\mathrm{Ent}$  admits a minimum at a unique ray through the origin.  
In other words,    
if we regard  $\mathrm{Ent}$  as a function defined on $int(\Omega)$, 
then it has a minimum at a unique point in  $int(\Omega)$. 
We denote this minimum by  $\min \mathrm{Ent}(M, \Omega)$.    
We also denote by  $\min \mathrm{Ent}(M)$,  the minimum of  
$\{\min \mathrm{Ent}(M,\Omega) \, | \, \Omega \; \; \text{is a fibered face of  $M$} \}$.

\begin{figure}
\begin{center}
\includegraphics[width=4in]{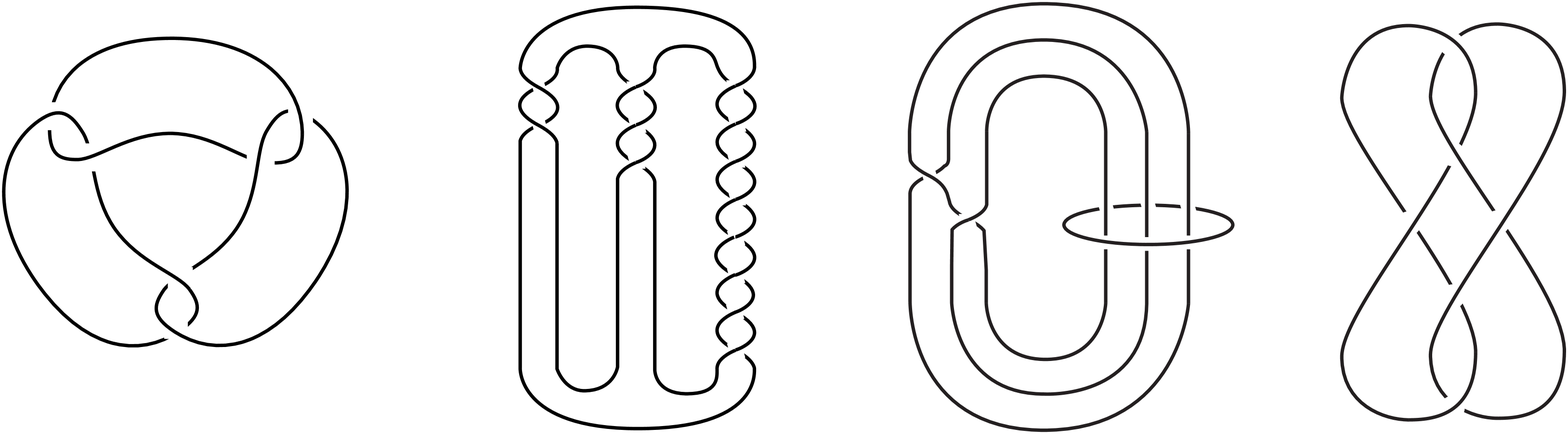}
\caption{
(from left to right) $3$ chain link $\mathcal{C}_3$, $(-2,3,8)$-pretzel link,  link $6_2^2$, Whitehead link}
\label{fig_3chain_etc}
\end{center}
\end{figure}

\subsection{Finiteness of Farb-Leininger-Margalit and Agol}

We recall a  result that connects pseudo-Anosovs 
having small dilatations with finitely many fibered $3$-manifolds.  
For  $P >1$,  consider the following set of pseudo-Anosov homeomorphisms on any surface as follows: 
\begin{equation*} 
	\Psi_P = \{ \Phi \, | \,  
	\Phi: \varSigma \rightarrow \varSigma \; \;  \mbox{pseudo-Anosov},  \; 
	\chi(\varSigma)<0, \; 
	\mathrm{Ent}(\Phi) = |\chi(\varSigma)| \mathrm{ent}(\Phi) \le \log P\}. 
\end{equation*}
Farb, Leininger and Margalit  called elements of  $\Psi_P$  {\it small dilatation pseudo-Anosov homeomorphisms}.  
Let  $\varSigma^{\circ} \subset \varSigma$  be the surface obtained by removing all 
the singularities of the stable foliation for  $\Phi$,  and  
$\Phi|_{\varSigma^{\circ}}: \varSigma^{\circ} \rightarrow \varSigma^{\circ}$  
denotes the restriction of  $\Phi$  to  $\varSigma^{\circ}$.  
Observe that  $\lambda(\Phi)= \lambda(\Phi|_{\varSigma^{\circ}})$. 
Let us put 
\begin{equation*} 
	\Psi_P^{\circ}= \{\Phi|_{\varSigma^{\circ}} 
	\, | \, 
	\Phi \in \Psi_P\}. 
\end{equation*} 
Penner's result implies that  
the set $\Psi_P$ is infinite if  $P$  is large 
($P \ge 11^2= 121$ for instance),  
and hence so is  $\Psi_P^{\circ}$.    
However, 
Farb-Leininger-Margalit \cite{FLM}  and Agol \cite{Agol2} showed  that 
if we let  ${\Bbb T}(\Psi_P^{\circ})$  be the set of  mapping tori of elements of  $ \Psi_P^{\circ}$,  
then  ${\Bbb T}(\Psi_P^{\circ})$  becomes a finite set. 
In other words, for any  $P >1$,  
there is a list of finitely many complete, 
noncompact hyperbolic $3$-manifolds  $M_1, M_2, \cdots, M_r$  fibering 
over the circle with the following property:
for any  $\Phi \in \Psi_P$,  
there exist  $M_i$  in the list and a particular fibration on $M_i$  such that 
$\Phi$   occurs as the monodromy of 
the  fibration on the manifold obtained 
from  $M_i$  by Dehn filling along boundary slopes of the fiber in question. 

Because of this,  it makes sense to say that 
small dilatation pseudo-Anosovs are ``generated" by 
a finite list of fibered  $3$-manifolds. 
This in particular implies that the following sets are finite 
because $\log \delta_g \asymp \tfrac{1}{g}$ and 
$\log \delta (D_n) \asymp \tfrac{1}{n}$ \cite{HK}. 
\begin{align*}
	\mathcal{U} 
	&= \{{\Bbb T}(\Phi|_{\varSigma^{\circ}})\ |\ 
	\Phi\ 
	\mbox{is\ pseudo-Anosov\ on\ } \varSigma= \varSigma_g 
	\mbox{\ such\ that\ }\lambda(\Phi)= \delta_g, 
	\, g \ge 2 \}, \\
	\mathcal{V} 
	&= \{{\Bbb T}(\Phi|_{\varSigma^{\circ}})\ |\ 
	\Phi\ \mbox{is\ pseudo-Anosov\ on\ } \varSigma= D_n 
	\mbox{\ such\ that\ } \lambda(\Phi)= \delta(D_n), 
	\, n \ge 3 \}. 
\end{align*}
A natural question arises: 
how large are these sets? 
By results \cite{AD,Hironaka,KKT,KT,KT1,Venzke}, 
we predict that  $\mathcal{U}$  and  $\mathcal{V}$  are quite small.  

A result in \cite{KT} says that the {\it magic manifold}  $N$,  
which is the  exterior of the $3$ chain 
link  $\mathcal{C}_3$  (see Figure~\ref{fig_3chain_etc}),  
is a member of  $\mathcal{V}$.   
More concretely, in \cite{KT} it was shown that 
for each  $3 \le n \le 8$ (resp. $n \ge 9$), 
there exists a pseudo-Anosov homeomorphism  
$\Phi_n: D_n \rightarrow D_n$ 
with the smallest dilatation (resp. the smallest known dilatation)  
which is generated by  $N$.  
The set  $\mathcal{V}$  might consist of a single element  $N$  indeed.  
Let us turn to the set  $\mathcal{U}$. 
Potential examples of members of  $\mathcal{U}$  are of 
the form  $N(r)$, 
which is the manifold obtained 
from the magic manifold  $N$ by Dehn filling 
one cusp along the slope  $r \in {\Bbb Q} \cup \{ \frac{1}{0} \}$. 
It was proved in \cite{AD,Hironaka,KT1}  that  
there exists a pseudo-Anosov homeomorphism 
on  $\varSigma_g$  for  $g \ge 3$  with  small dilatation 
generated by  $N(\tfrac{3}{-2})$  or by  $N(\tfrac{1}{-2})$. 
The manifolds  $N(\tfrac{3}{-2})$  and  $N(\tfrac{1}{-2})$  are 
the Whitehead sister link 
(i.e, $(-2,3,8)$-pretzel link) exterior and 
the  $6_2^2$  link exterior respectively 
(see Figure~\ref{fig_3chain_etc}).  
What Kin-Takasawa proved in \cite[Theorem~1.5]{KT1} is the following.

\begin{thm}
\label{thm_three}
Let  $r \in \{ \tfrac{3}{-2}, \tfrac{1}{-2}, 2\}$. 
For each  $g \ge 3$,  
there exist  $\varSigma_g$-bundles over the circle 
obtained from  $N(r)$  by Dehn filling both cusps  
along boundary slopes of fibers of  $N(r)$. 
Among them,  
there exist monodromies  
$ \Phi_g(r): \varSigma_g \rightarrow \varSigma_g$ of the fibrations such that  
\begin{equation*} 
	\lim_{g \to \infty} g  \log \lambda (\Phi_g(r)) 
	=\log \delta(D_3) = \log( \tfrac{3+ \sqrt{5}}{2}).
\end{equation*} 
\end{thm}

\noindent
As a corollary, we have the following estimate from above, which was  proved by Hironaka  first \cite{Hironaka}. 
\begin{equation}
\label{equation_asyEneq}
	\limsup_{g \to \infty} \, g \log  \delta_g \le \log(\tfrac{3+ \sqrt{5}}{2}). 
\end{equation}

\noindent
Theorem~\ref{thm_three}  is also established by Hironaka   for  $r = \frac{1}{-2}$ in \cite{Hironaka}  and  
Aabar-Dunfield  for  $r = \frac{3}{-2}$ in \cite{AD}  independently.

\subsection{Thurston norm and Teichm\"{u}ller polynomial 
of the magic manifold}\label{subsection_TTM}

In view of the results in previous two subsections, 
we will focus only on the magic manifold  $N$ 
and present various computational results.  
To do this, 
we discuss some detailed description of  $N$  in this subsection.  

Let $K_{\alpha}$, $K_{\beta}$ and  $K_{\gamma}$  be the components 
of the $3$ chain link  $\mathcal{C}_3$. 
The orientation of each component of $\mathcal{C}_3$ 
is given in Figure~\ref{fig_poly}(right). 
They  bound the oriented 
disks $F_{\alpha}$, $F_{\beta}$ and $F_{\gamma}$ with $2$ holes. 
Let us set  $\alpha = [F_{\alpha}]$, $ \beta = [F_{\beta}]$,  $\gamma= [F_{\gamma}] \in H_2(N, \partial N; {\Bbb Z})$. 
In \cite{Thurston1},  
Thurston computed  the unit ball $U_N$  which  is 
the the parallelepiped with vertices 
$\pm \alpha $, $\pm \beta $, $\pm \gamma$, $\pm(\alpha + \beta + \gamma)$, 
see Figure~\ref{fig_poly}(left). 
%(In particualr $N$ has $6$ top dimensional faces on $\partial U_N$.) 
The set $\{\alpha, \beta, \gamma\}$ is a basis of $H_2(N, \partial N; {\Bbb Z})$, and 
the class $x \alpha + y \beta + z \gamma \in H_2(N, \partial N)$ is denoted by $(x,y,z)$.  
Every top dimensional face on $\partial U_N$ is a fibered face 
because of  the symmetries of $H_2(N, \partial N)$, 
see Section~\ref{subsection_Fibered}. 
McMullen developed a general theory of 
the Teichm\"{u}ller polynomial $P_{\Omega}$  for a 
fibered face   $\Omega$  of fibered hyperbolic 3-manifolds,  
from which one can compute the dilatation  $\lambda(a)$ of each $a \in int(C_{\Omega}) $, 
see \cite{McMullen}. 
Let us pick the fibered face  $\Delta$ on  $\partial U_N$  as in Figure~\ref{fig_poly}(left) with 
vertices  $(1,0,0)$, $ (1,1,1)$, $(0,1,0)$ and $ (0,0,-1)$.  
The Teichm\"{u}ller polynomial $P_{\Delta}$ tells us that 
the dilatation $\lambda(a) = \lambda_{(x,y,z)}$ of 
a primitive fibered class $a=(x,y,z) \in int(C_{\Delta})$ is the largest  real root of  
\begin{equation} 
\label{equation_TpolyMagic}
	f_{(x,y,z)}(t)= t^{x+y-z}-t^x - t^y - t^{x-z}- t^{y-z}+1, 
\end{equation}  
see \cite[Theorem~3.1]{KT}. 
Thus, 
we have a reasonable source to compute dilatations systematically.  

\begin{figure}[htbp]
\begin{center}
\includegraphics[width=4in]{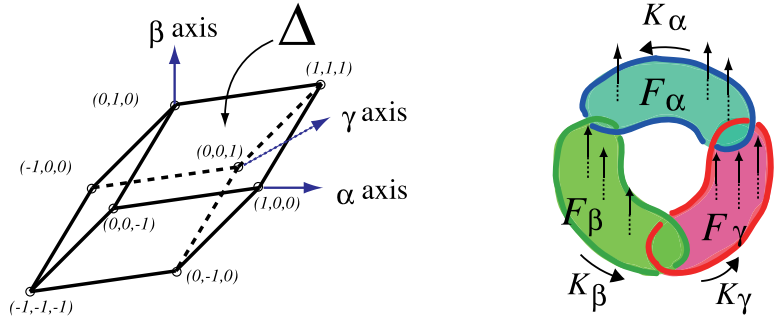}
\caption{(left)  Thurston norm ball for N. (right) $F_{\alpha}$, $F_{\beta}$, $F_{\gamma}$. 
[arrows indicate the normal direction of oriented surfaces.]}
\label{fig_poly}
\end{center}
\end{figure} 

To relate these data to ones on closed manifolds 
obtained from  $N$  by Dehn filling, 
we prepare a few homological properties of  $N$.  
Denote by  $T_{\alpha}$  the torus which is 
the boundary of a regular neighborhood of  $K_{\alpha}$. 
We define the tori  $T_{\beta}$  and  $T_{\gamma}$  in the same manner. 
For a primitive integral class  $a= (x,y,z) \in H_2(N, \partial N)$, 
let us set  $\partial_{\alpha} F_{a} = \partial F_{a} \cap T_{\alpha}$ 
which consists of the parallel simple closed curves on  $T_{\alpha}$. 
We define  $\partial_{\beta} F_{a} $ and $\partial_{\gamma} F_{a}$  in 
the same manner. 
We see that the slope of  $\partial_{\alpha} F_a$  
(resp. $\partial_{\beta} F_a$,  
$\partial_{\gamma} F_a$)  is given by 
$b_{\alpha}(a)= \tfrac{y+z}{-x}$  
(resp.  $b_{\beta}(a)= \tfrac{z+x}{-y}$,  
$b_{\gamma}(a)= \tfrac{x+y}{-z}$). 
We call each of  
$b_{\alpha}(a)$,  $b_{\beta}(a)$,  $b_{\gamma}(a)$  the {\it boundary slope} 
of  $a$.  

For more detailed computation,  
we specify the cusp to be Dehn filled.  
Let  $N(r)$  be the manifold obtained from  $N$  by Dehn filling 
the cusp specified, 
say,  
by $T_{\beta}$  along the slope  $r \in \mathbb{Q}$.  
Then, 
there exists a natural injection  
$\iota_{\beta}: H_2(N(r), \partial N(r)) \rightarrow H_2(N, \partial N)$  
whose image equals $S_{\beta}(r)$, 
where 
\begin{equation*} 
	S_{\beta}(r)= \{(x,y,z) \in H_2(N, \partial N)\ |\ -ry= z+x\},   
\end{equation*} 
see Proposition~\ref{prop_S_r}. 
This implies that every slope  $r  \in {\Bbb Q}$  can 
be realized by a boundary slope of some 
$a \in H_2(N, \partial N)$.   
It is known by  \cite{MP}  that 
$N(r)$  is hyperbolic if and only 
if  $r \in {\mathcal Hyp} = {\Bbb Q}   \setminus \{-3, -2, -1, 0\}$.  
Choose such  $r \in {\mathcal Hyp}$,  
and assume that  $a \in S_{\beta}(r) = \mathrm{Im} \, \iota_{\beta}$  is 
a fibered class in  $H_2(N, \partial N)$.  
Then,  
$\overline{a} = \iota_{\beta}^{-1}(a) \in H_2(N(r), \partial N(r))$  is also 
a fibered class of  $N(r)$.  
This description enables us to compute the Thurston norm of  $N(r)$, 
especially the unit ball and fibered faces, 
and hence to handle closed surface bundles obtained 
from $N$ by Dehn filling all cusps 
systematically.  

What we further need for our purpose is to know  
a systematic method to compute entropies of monodromies on such bundles.

%%%
%%%

\subsection{Main results}
\label{subsection_MainResults}

This paper concerns the small dilatation pseudo-Anosovs
generated by the magic manifold  $N$  with a very mild restriction  
which we describe below. 
Let  $\Phi: F \rightarrow F$ be the monodromy of 
a fibration on  $N$,  
and let  $\phi$  be the mapping class of  $\Phi$. 
Then the fibration extends naturally to a fibration on the closed manifold 
obtained from  $N$  by Dehn filling three cusps along 
boundary slopes of  $F$.  
Also, $\Phi$  extends to 
the monodromy  $\widehat{\Phi}: \widehat{F} \rightarrow \widehat{F}$  of the extended fibration,  
where the extended fiber  $\widehat{F}$  is obtained from  $F$  by filing holes.  
Suppose that the stable foliation  $\mathcal{F}$  of  $\Phi$ has the property such that 
each boundary component of  $F$ has no $1$ prong.  
Then   $\mathcal{F}$  extends canonically 
to the stable foliation  $\widehat{\mathcal{F}}$  of  $\widehat{\Phi}$, 
and 
$\widehat{\phi}= [\widehat{\Phi}]$  becomes pseudo-Anosov 
(including Anosov)  with the same dilatation 
as that of  $\phi$. 
We consider the set  $\mathcal{M}$  of (pseudo-Anosov) mapping classes coming from fibrations of $N$  with this condition, i.e, 
if we let $\mathcal{F}$ be the stable foliation associated to the fibration on $N$, then 
$\mathcal{F}$ has the property such that 
\begin{equation}
\label{equation_1plong} 
	\text{any boundary component of\ } F\ \text{has no\ }1\ \text{prong}.
\end{equation}   
We will see that 
this restriction is fairly mild (Lemmas~\ref{lem_face-prong} and \ref{lem_LargeSet}). 
Let  $\widehat{\mathcal{M}}$ be 
the set of extensions  $\widehat{\phi}$  of  $\phi \in \mathcal{M}$  defined on 
the closed surfaces.  
For example, 
the pseudo-Anosov $\Phi_g(r): \varSigma_g \rightarrow \varSigma_g$  in Theorem~\ref{thm_three} enjoys 
$[\Phi_g(r)] \in \widehat{\mathcal{M}}$ for large $g$, see \cite[Lemma~4.8]{KT1}.

Let $\widehat{\delta}_g$ be the minimum among dilatations  
of elements in $\widehat{\mathcal{M}} \, \cap \, \mathrm{Mod}(\varSigma_g)$. 
Clearly $\delta_g \le \widehat{\delta}_g$. 
The equality is achieved when $g=2$ \cite{Hironaka}. 
We prove that the limit of $ g  \log \widehat{\delta}_g$ exists and it equals the number 
which we encountered  in Theorem~\ref{thm_three}. 

\begin{thm}
\label{thm_main} 
\ 
\begin{enumerate}
\item[(1)] 
We have 
$\displaystyle \lim_{g \to \infty} g  \log \widehat{\delta}_g = \log( \tfrac{3+ \sqrt{5}}{2})$. 

\item[(2)] 
For large $g$, $\widehat{\delta}_g$ is achieved by the monodromy of some $\varSigma_g$-bundle over the circle 
obtained from either  $N(\tfrac{3}{-2})$  or  $N(\tfrac{1}{-2})$    
by Dehn filling both cusps.  
\end{enumerate}
\end{thm}

\noindent
More precisely, the following holds. (See also Remark~\ref{rem_main}.) 
For large $g$ such that $g \equiv 0,1,5,6,7,9 \pmod{10}$ 
(resp.   large $g$ such that $g \equiv 3,8 \pmod{10}$), 
$\widehat{\delta}_g$ is achieved by the monodromy of some $\varSigma_g$-bundle over the circle 
obtained from $N(\tfrac{3}{-2})$ (resp. $N(\tfrac{1}{-2})$)     
by Dehn filling both cusps.  

We know from \cite[Proposition 4.37]{KT1} that 
for $g= 8,13$, $\widehat{\delta}_g$ cannot be achieved by 
the monodromy of any $\varSigma_g$-bundle over the circle 
obtained from either $N(\tfrac{3}{-2})$ or $N(\tfrac{1}{-2})$ by Dehn filling. 
In fact, the manifold $N(\tfrac{4}{-3}, \tfrac{25}{-17}, -5 )$ (resp. $N(\tfrac{29}{-27}, \tfrac{5}{-3}, -6)$) is 
a $\varSigma_8$-bundle (resp. a $\varSigma_{13}$-bundle) over the circle, 
where $N(r_1, r_2, r_3)$ is the closed manifold obtained from $N$ by Dehn filling all cusps 
along the slopes $r_1$, $r_2$ and $r_3$. 
Its dilatation is smaller than that of   any $\varSigma_8$-bundle (resp. $\varSigma_{13}$-bundle) over the circle 
obtained from either $N(\tfrac{3}{-2})$ or $N(\tfrac{1}{-2})$ by Dehn filling. 
Theorem~\ref{thm_main}  says  that if $g$ is large, then  
among elements of $ \widehat{\mathcal{M}}$, 
the pseudo-Anosovs defined on $\varSigma_g$ with the smallest dilatation 
 are the same examples identified by Hironaka \cite{Hironaka}, Aabor-Dunfield \cite{AD} and Kin-Takasawa \cite{KT1}. 

However, we can find new examples in  $\widehat{\mathcal{M}}$ defined on $\varSigma_g$  with orientable invariant foliations when $g \equiv 0 \pmod{6}$. 
Let $\widehat{\mathcal{M}}^+$ be the set of pseudo-Anosov elements of $\widehat{\mathcal{M}}$ with orientable invariant foliations. 
%For $g \ge 2$, $\widehat{\mathcal{M}}^+ \, \cap \, \mathrm{Mod}(\varSigma_g) \ne \emptyset$, see Lemma~\ref{lem_nonempty}. 
Let $\widehat{\delta}_g^+$ be the minimum among dilatations of elements in $\widehat{\mathcal{M}}^+ \, \cap \, \mathrm{Mod}(\varSigma_g)$. 
Since $\widehat{\mathcal{M}}^+ \, \cap \, \mathrm{Mod}(\varSigma_g) \ne \emptyset$ for $g \ge 2$ (Lemma~\ref{lem_nonempty}), 
$\widehat{\delta}_g^+$ is well-defined. 
Clearly $\delta_g \le \delta_g^+ \le \widehat{\delta}_g^+$. 

The following describes the asymptotic behavior of $ \widehat{\delta}_g^+$'s in the case $g \not\equiv 0 \pmod{6}$. 

\begin{thm}
\label{thm_smallest_dil_ori}
\ 
\begin{enumerate}
\item[(1)] 
We have $\displaystyle \lim_{\substack{g \not\equiv 0 \pmod{6} \\ g \to \infty}}  g  \log \widehat{\delta}^+_g = \log( \tfrac{3+ \sqrt{5}}{2}) $. 

\item[(2)] 
For large $g$ such that $g \equiv 2,4 \pmod{6}$ or $g \equiv 3 \pmod{10}$, 
$\widehat{\delta}^+_g$ is achieved by the monodromy of some $\varSigma_g$-bundle over the circle 
obtained from $N(\tfrac{1}{-2})$  by Dehn filling both cusps.  

\item[(3)] 
For large $g$ such that $g \equiv 1,5,7,9 \pmod{10}$, 
 $\widehat{\delta}^+_g$ is achieved by the monodromy of some $\varSigma_g$-bundle over the circle 
obtained from $N(\tfrac{3}{-2})$ by Dehn filling both cusps.  
\end{enumerate}
\end{thm}

\noindent
Theorem~\ref{thm_smallest_dil_ori}(1) leads to the following estimate, which was proved  by Hironaka first \cite{Hironaka}. 
$$\limsup_{\substack{g \not\equiv 0 \pmod{6} \\ g \to \infty}} g \log  \delta_g^+ \le \log(\tfrac{3+ \sqrt{5}}{2}). $$
If $g \not\equiv 0 \pmod{6}$ is large, then 
elements  of $\widehat{\mathcal{M}}^+$ defined on $\varSigma_g$ with the smallest dilatation 
 are the same examples discovered in \cite{Hironaka,AD,KT1}.

In the case $g \equiv 0 \pmod{6}$, 
there exist no examples of elements in $\widehat{\mathcal{M}}^+$ defined on $\varSigma_g$ 
which occur as monodromies of fibrations on manifolds obtained from 
$N(\tfrac{1}{-2})$ or $N(\tfrac{3}{-2})$ by Dehn filling both cusps \cite{Hironaka, AD,KT1}. 
To the best of our knowledge, 
the smallest known upper bound on $\delta_g^+$ for $g \equiv 0 \pmod{6}$ is 
\begin{equation}
\label{equation_MinakawaHK}
\delta_g^+ \le \lambda_{(g,g,-1)}, 
\end{equation}
 where 
$\lambda_{(g,g,-1)}$  is the largest root of 
$f_{(g,g,-1)}(t)= t^{2g+1} - 2t^{g+1}- 2t^g +1$, 
see \cite{Minakawa,HK}. 
By using the bound (\ref{equation_MinakawaHK}), 
Minakawa and independently Hironaka-Kin proved that 
\begin{equation}
\label{equation_MinakawaHK2}
g \log \delta_g^+ \le \log(2+ \sqrt{3}) \approx 1.3169.
\end{equation}
We would like to point out that 
the set $\widehat{\mathcal{M}}^+$ could be a source to provide a sharper upper bound on $\delta_g^+$ 
than the bound (\ref{equation_MinakawaHK})
in the case $g \equiv 0 \pmod{6}$. 
In fact, we will find elements of  $\widehat{\mathcal{M}}^+$ defined on $\varSigma_g$ for $g \equiv 6 \pmod{12} $ 
whose normalized entropies go to $4 \log \delta(D_5)$ as $g$ goes to $\infty$, 
see Lemmas~\ref{lem_minEnt4} and \ref{lem_6+12i}. 
These examples occur as monodromies of fibrations on manifolds obtained from 
$N(4)$ or $N(-6)$ by Dehn filling both cusps. 
As a corollary, we have

\begin{thm}
\label{thm_6+12i}
$\displaystyle \limsup_{\substack{g \equiv 6 \pmod{12} \\ g \to \infty}}  g  \log \delta^+_g  \le 2 \log  \delta(D_5) \approx 1.0870 $.
\end{thm}

\noindent
By using our examples, 
we give the following upper bound on $\delta_g^+$ for $g \equiv 6 \pmod{12}$ 
which is sharper than the previous one (\ref{equation_MinakawaHK}). 
(See also Table~\ref{table_6i}.) 

\begin{thm}[Upper bound on $\delta_g^+$ for $g \equiv 6 \pmod{12}$]
\label{thm_bound6+12i}
\ 
\begin{enumerate}
\item 
$\delta_g^+ \le \lambda_{(\tfrac{3g}{2}+1, \tfrac{3g}{2}-1, \tfrac{g}{2})} $ 
if $g \equiv 6,30,42,54,78 \pmod{84}$.

\item 
$\delta_g^+ \le \lambda_{(g+2,g-2, -\tfrac{g}{2})}  $ 
if $g \equiv18,66 \pmod{84}$.
\end{enumerate}
\end{thm}

In the case $g \equiv 0 \pmod{12}$, 
we improve the  bound (\ref{equation_MinakawaHK}) for many $g$, see Table~\ref{table_6i}.

\begin{table}[hbtp]
\caption{upper bound on $ \delta_g^+$ for $6 \le g \le 216$ such that $g \equiv 0 \pmod{6}$. 
[the bounds   on $\delta_g^+$ on the left  in the case $g \equiv 6 \pmod{12}$ 
 come from Theorem~\ref{thm_bound6+12i}. 
 in the case $g \equiv 0 \pmod{12}$ and $g>12$, 
  the bounds on $\delta_g^+$ on the left  
  are given by elements of $\widehat{\mathcal{M}}^+$ 
which occur as monodromies of  fibrations on manifolds obtained from  $N(\tfrac{5}{-4})$ by Dehn filling both cusps.]} 
\label{table_6i}
\begin{center}
$\left|\begin{array}{c|c|c}
\hline
g  & \mbox{upper bound on\ }\delta_g^+& \mbox{upper bound on\ } \delta_g^+\ \mbox{in \cite{Minakawa,HK}}  \\
\hline
\hline
6& \lambda_{(10,8,3)} \approx 1.20189 & \lambda_{(6,6,-1)} \approx 1.22571  \\
 \hline 
 12& \lambda_{(12,20,3)} \approx 1.10240 &\lambda_{(12,12,-1)} \approx  1.11124  \\
 \hline 
18& \lambda_{(20,16,-9)}  \approx 1.06276 &\lambda_{(18,18,-1)} \approx  1.07382  \\
 \hline 
 24& \lambda_{(32,28,3)}  \approx 1.04757 &\lambda_{(24,24,-1)} \approx  1.05524  \\
 \hline 
30&  \lambda_{(46,44,15)} \approx 1.03692  &\lambda_{(30,30,-1)} \approx  1.04413  \\
 \hline 
 36&  \lambda_{(50,52,15)} \approx 1.03148  &\lambda_{(36,36,-1)} \approx 1.03674  \\
 \hline 
42& \lambda_{(64,62,21)} \approx 1.02622  &\lambda_{(42,42,-1)} \approx 1.03147 \\
 \hline 
 48& \lambda_{(66,68,19)} \approx 1.02367  &\lambda_{(48,48,-1)} \approx 1.02752 \\
 \hline 
54&     \lambda_{(82,80,27)} \approx 1.02033      &\lambda_{(54,54,-1)} \approx  1.02446     \\
  \hline 
60&     \lambda_{(80,76,15)} \approx 1.01903      &\lambda_{(60,60,-1)} \approx  1.02200       \\
  \hline 
66&   \lambda_{(68,64,-33)} \approx 1.01661        &\lambda_{(66,66,-1)} \approx  1.02000    \\  
 \hline 
72&   \lambda_{(96,92,19)} \approx 1.01586        &\lambda_{(72,72,-1)} \approx  1.01833       \\  
 \hline 
78&  \lambda_{(118,116,39)} \approx 1.01403    & \lambda_{(78,78,-1)} \approx 1.01691       \\  
 \hline 
 84&  \lambda_{(114,116,31)} \approx 1.01357    &\lambda_{(84,84,-1)} \approx  1.01570        \\  
 \hline 
90&   \lambda_{(136,134,45)} \approx 1.01215     &\lambda_{(90,90,-1)} \approx  1.01465      \\  
     \hline 
 96&   \lambda_{(132,140,43)} \approx 1.01190     &\lambda_{(96,96,-1)} \approx  1.01374         \\  
     \hline  
102& \lambda_{(104,100,-51)}  \approx 1.01071    & \lambda_{(102,102,-1)} \approx  1.01293    \\  
     \hline 
  108& \lambda_{(146,148,39)}  \approx 1.01057    &\lambda_{(108,108,-1)} \approx   1.01221      \\  
     \hline 
114 & \lambda_{(172,170,57)} \approx 1.00958  &\lambda_{(114,114,-1)} \approx    1.01156   \\  
     \hline 
 120 & \lambda_{(164,172,51)} \approx 1.00952  & \lambda_{(120,120,-1)} \approx   1.01098    \\  
     \hline 
126 &     \lambda_{(190,188,63)} \approx 1.00841     &\lambda_{(126,126,-1)} \approx  1.01046     \\  
 \hline 
 132 &     \lambda_{(174,164,31)} \approx 1.00869     &\lambda_{(132,132,-1)} \approx  1.00998      \\  
 \hline 
 138  &   \lambda_{(208,206,69)}  \approx 1.00790     &  \lambda_{(138,138,-1)} \approx 1.00955     \\
     \hline  
 144  &   \lambda_{(194,196,51)}  \approx 1.00793     &\lambda_{(144,144,-1)} \approx     1.00915      \\
     \hline  
150   &   \lambda_{(152,148,-75)} \approx 1.00727    &   \lambda_{(150,150,-1)} \approx  1.00878       \\
     \hline  
156   &   \lambda_{(210,212,55)} \approx 1.00732    & \lambda_{(156,1566,-1)} \approx    1.00845        \\
     \hline   
162 &   \lambda_{(244,,242,81)} \approx 1.00673      &  \lambda_{(162,162,-1)} \approx   1.00813        \\
     \hline  
 168  &   \lambda_{(228,236,67)} \approx 1.00680      &  \lambda_{(168,168,-1)} \approx    1.00784        \\
     \hline  
 174  &   \lambda_{(262,260,87)} \approx 1.00626      &  \lambda_{(174,174,-1)} \approx   1.00757      \\
     \hline  
 180  &   \lambda_{(240,236,55)} \approx 1.00635      &  \lambda_{(180,180,-1)} \approx    1.00732       \\
     \hline  
 186  &   \lambda_{(188,184,-93)} \approx 1.00586       &  \lambda_{(186,186,-1)} \approx   1.00708      \\
     \hline  
 192  &   \lambda_{(258,260,67)} \approx 1.00595       &  \lambda_{(192,192,-1)} \approx    1.00686          \\
     \hline  
 198 &   \lambda_{(298,296,99)} \approx 1.00550         & \lambda_{(198,198,-1)} \approx     1.00665    \\
     \hline 
  204 &   \lambda_{(276,284,79)} \approx 1.00560         & \lambda_{(204,204,-1)} \approx    1.00646     \\
     \hline  
210  &  \lambda_{(316,314,105)}  \approx 1.00519        &\lambda_{(210,210,-1)} \approx      1.00627        \\
     \hline  
216  &  \lambda_{(290,292,75)}  \approx 1.00529        &  \lambda_{(216,216,-1)} \approx    1.00610         \\
\hline     
\end{array}\right|$
\end{center} 
\end{table}

Section~\ref{subsection_WhiteheadLink} concerns the monodromies of fibrations on the Whitehead link exterior $N(1)$. 
The manifold $N(1)$ is very special among other $N(r)$'s. 
It is the only manifold among the $N(r)$'s which admits fibers with arbitrarily many boundary components (Lemma~\ref{lem_num-boundary}). 
Moreover the invariant foliation of the monodromy of each fibration on $N(1)$ has the property such that 
each boundary component of the fiber has a $1$ prong. 
(Remark~\ref{rem_Whitehead}). 
We shall show in Section~\ref{subsection_WhiteheadLink} that 
there exists the monodromy $\Phi_n: \varSigma_{1,n} \rightarrow \varSigma_{1,n}$ 
of a particular fibration on $N(1)$ whose normalized entropy  tends to $ 2 \log \delta(D_4) $ 
as $n$ tends to $\infty$ (Proposition~\ref{prop_MiniDil_W} and Lemma~\ref{lem_direction_W}). 
Thus we have

\begin{thm} 
\label{thm_tsai}
$\displaystyle \limsup_{n \to \infty} \, n \log \delta_{1,n} \le  2 \log \delta(D_4) $. 
\end{thm}

\noindent
This implies the upper bound 
$\displaystyle \limsup_{n \to \infty} \, n \log \delta_{1,n} \le 2 \log 9$ by Tsai, 
%see \cite[Section~3.2.1]{Tsai1} (in particular \cite[Theorem~3.2.2]{Tsai1}). 
see \cite[Section~3.2.1 and Theorem~3.2.2]{Tsai1}.

\subsection{Thurston norm equivalence, entropy equivalence on manifolds $N(r)$}
\label{subsection_equivalence}

In the course of analyzing the magic manifold, 
we discovered many ``twins"  among the  $N(r)$'s.  
The particular ones are 
$N(\tfrac{3}{-2})$ and $N(\tfrac{1}{-2})$    which will be critical in the proof of Theorem~\ref{thm_main}.   
They are different manifolds but 
have  common properties from entropy computational viewpoints.  
To formulate ideas, 
we say that   $3$-manifolds  $M$  and  $M'$ are  {\it Thurston norm equivalent}, 
denoted by 
\begin{equation*} 
	M  \underset{\mathrm{T}}{\sim} M', 
\end{equation*}   
if there exists an isomorphism 
$f: H_2(M, \partial M ; {\Bbb Z}) \rightarrow H_2(M', \partial M'; {\Bbb Z})$  which preserves the Thurston norm, i.e, 
$\|a\| = \|f(a)\|$ for any  $a \in H_2(M, \partial M ; {\Bbb Z})$. 
We call such  $f$ the {\it Thurston norm preserving isomorphism}. 
For example  $N(r) \underset{\mathrm{T}}{\sim}  N(-2-r)$ when 
$r, -2-r \in {\mathcal Hyp}$ (Proposition~\ref{prop_norm-equiv}). 
We introduce two more equivalence relations, 
both called the {\it entropy equivalence}, 
of which the precise definitions are given in Section~\ref{subsubsection_Def}. 
The first one is defined on the pairs  $(M,  \Omega)$, 
where  $M$  is a fibered $3$-manifold and  $\Omega$  is its fibered face. 
Namely,  
$(M, \Omega)$  and  $(M', \Omega')$  are entropy equivalent,  
denoted by 
\begin{equation*} 
	(M, \Omega) \underset{\mathrm{ent}}{\sim} (M', \Omega'), 
\end{equation*}  
if  there exists a Thurston norm preserving isomorphism 
$f: H_2(M, \partial M ; {\Bbb Z}) \rightarrow H_2(M', \partial M'; {\Bbb Z})$ 
such that  $f$  maps  $int(C_{\Omega}(\Bbb Z))$ to $int(C_{\Omega'}(\Bbb Z))$  preserving the entropy function. 
In particular  $(M, \Omega) \underset{\mathrm{ent}}{\sim} (M', \Omega')$  implies 
that  $\min \mathrm{Ent}(M, \Omega) = \min \mathrm{Ent}(M', \Omega')$. 
The second equivalence relation is defined on the fibered $3$-manifolds.  
Fibered $3$-manifolds $M$ and $M'$ are entropy equivalent, 
denoted by 
\begin{equation*} 
	M \underset{\mathrm{ent}}{\sim} M',  
\end{equation*} 
if  there exists a Thurston norm preserving isomorphism 
$f: H_2(M, \partial M ; {\Bbb Z}) \rightarrow H_2(M', \partial M'; {\Bbb Z})$  
such that  $f$  preserves 
both fibered classes and the entropy functions. 
If  $M \underset{\mathrm{ent}}{\sim} M'$, 
then  $\min \mathrm{Ent}(M) = \min \mathrm{Ent}(M')$. 
We shall prove in Theorem~\ref{thm_entropy_equiv} that 
$$(N(2), \Omega_S)  \underset{\mathrm{ent}}{\sim}  (N(\tfrac{3}{-2}), \Omega_A) \underset{\mathrm{ent}}{\sim}  (N(\tfrac{1}{-2}), \Omega_A). $$
For the definition of fibered faces  $\Omega_S$  and  $\Omega_A$,  
see Section~\ref{subsubsection_EntropyEqu}.  
We also prove that $N(r)  \underset{\mathrm{ent}}{\sim} N(-2-r)$ for 
`almost all' $r \in {\mathcal Hyp} $, see Theorem~\ref{thm_entropy_equiv}. 
This is derived from  the symmetry of the Thurston norm ball and the symmetry of 
the entropy function of $N$. 
In particular 
\begin{equation*} 
	N(\tfrac{3}{-2}) \underset{\mathrm{ent}}{\sim}  N(\tfrac{1}{-2}). 
\end{equation*} 

\noindent
Recall that the quantity 
$\min \mathrm{Ent}$ is defined to be the minimum of 
the normalized entropies of the classes in $\displaystyle\bigcup_{\Omega} int(C_{\Omega})$, 
where $\Omega$ is taken over all fibered faces of $M$. 
The number $\log( \tfrac{3+ \sqrt{5}}{2})$ in Theorem~\ref{thm_three} appears in the equalities 
\begin{equation*} 
	\min \mathrm{Ent}(N(\tfrac{3}{-2})) =  \min \mathrm{Ent}(N(\tfrac{1}{-2})) =  \min \mathrm{Ent}(N(2), \Omega_S)= 2 \log( \tfrac{3+ \sqrt{5}}{2}).  
\end{equation*}

\subsection{Question by Lanneau-Thiffeault}  
\label{subsection_LTQuestion}

Let $k$ and $\ell$ be integers such that $0< \ell < k$. 
We consider following fibered classes in $int(C_{\Delta})$: 
$$(2k \pm  \ell, 2k \pm  2\ell, k \pm 2 \ell) \in S_{\beta} (\tfrac{3}{-2}),\ 
(k, 2k \pm 2\ell, \pm \ell) \in S_{\beta} (\tfrac{1}{-2}), \ 
(k \pm \ell, k \mp \ell, -k) \in S_{\gamma}(2). $$
By using the  Teichm\"{u}ller polynomial (\ref{equation_TpolyMagic}), 
we see that  the dilatation of each  fibered class above is equal to 
the largest real root $\lambda_{(k,\ell)}$ of the following {\it Lanneau-Thiffeault polynomial} 
$$f_{(k,\ell)}(t)= t^{2k} - t^{k+\ell} - t^k - t^{k-\ell}+1.$$ 
(In fact, $f_{(k,\ell)}(t)$ is a common factor of $f_{(2k \pm  \ell, 2k \pm  2\ell, k \pm 2 \ell)}(t)$, 
$f_{(k, 2k \pm 2\ell, \pm \ell)}(t)$ and $f_{(k \pm \ell, k \mp \ell, -k)}(t)$.)

It is known that 
$\delta_2^+= \lambda_{(2,1)}$, 
%$\delta_3^+ =\lambda_{(3,1)}= \lambda_{(4,3)}$, 
$\delta_4^+ = \lambda_{(4,1)}$, 
%$\delta_5^+ =\lambda_{(6,1)} = \lambda_{(7,4)} $, 
$\delta_6^+ \ge \lambda_{(6,1)}$, 
%$\delta_7^+ = \lambda_{(9,2)}$ and  
$\delta_8^+=\lambda_{(8,1)} $, see \cite{Zhirov,LT,Hironaka}. 
Motivated by these results, Lanneau and Thiffeault asked the following. 

\begin{ques}[\cite{LT}]
\label{ques_LT}
For $g$ even,   is  $\delta_g^+$ equal to $\lambda_{(g,1)}$? 
\end{ques}

\noindent
We consider Question~\ref{ques_LT} in the set $\widehat{\mathcal{M}}^+$. 
The results in this paper imply that 
there exists a gap between $\widehat{\delta}^+_g $ and $\lambda_{(g,1)}$ 
for large $g$ such that $g \equiv 0 \pmod{6}$.

\begin{thm}
\label{thm_LTQuestion}
\ 
\begin{enumerate}
\item[(1)] 
We fix any $\epsilon>0$ so that $1.97475 - \epsilon > 2\log( \tfrac{3+ \sqrt{5}}{2})$. 
For large $g$ such that $g \equiv 0 \pmod{6}$, we have 
$\widehat{\delta}^+_g > \lambda_{(g,1)}$ and 
$$|\chi(\varSigma_g)| \log \widehat{\delta}^+_g > 1.97475 - \epsilon > 2\log( \tfrac{3+ \sqrt{5}}{2}).$$

\item[(2)]
$\widehat{\delta}^+_g = \lambda_{(g,1)}$ for large $g$ such that $g \equiv 2,4 \pmod{6}$. 
\end{enumerate}
\end{thm}

\subsection{Idea of proofs and conjectures}
\label{subsection_Idea}
This subsection describes the outline of the proof of Theorem~\ref{thm_main}. 
(The proof of Theorem~\ref{thm_smallest_dil_ori} is similar.) 
First, let us  recall the approach to  (\ref{equation_asyEneq}) by Hironaka \cite{Hironaka}, Aaber-Dunfield \cite{AD} and Kin-Takasawa \cite{KT1}. 
Take a  particular single $2$-cusped hyperbolic fibered $3$-manifold $M$
(which is either  $N(\tfrac{1}{-2}), N(\tfrac{3}{-2})$ or $N(2)$.) 
Compute the Teichm\"{u}ller polynomial $P_{\Omega}$ and $\min \mathrm{Ent}(M, \Omega)$ ($=2 \log( \tfrac{3+ \sqrt{5}}{2})$ in this case) 
for a fibered face $\Omega$ of $M$. 
Then determine the topological type of each fiber $F$ such that $[F] \in int(C_{\Omega})$. 
We can find a fiber $F_g$ of genus $g$ for large $g$ which enjoys the following. 
The ray of $[F_g]$ goes to the ray 
whose normalized entropy $\mathrm{Ent}$ achieves $\min \mathrm{Ent}(M, \Omega)$ 
as $g$ goes to $\infty$. 
(Then $\mathrm{Ent}([F_g])$ goes to $\min \mathrm{Ent}(M, \Omega)$ as $g$ goes to $\infty$.) 
Moreover, the number of boundary components of $F_g$ is bounded by some constant. 
Finally  check that the stable foliation for the monodromy of the fibration associated to $[F_g]$ 
satisfies that 
each boundary component of $F_g$ has no  $1$ prong.  
Then we obtain the equality in Theorem~\ref{thm_three} which implies (\ref{equation_asyEneq}). 
%The key is to take  a fiber $F_g$ whose ray goes to the ray which achieves $\min \mathrm{Ent}(M, \Omega)$. 

Compared to the above approach, 
a difficulty for the proof of  Theorem~\ref{thm_main} is that 
for each $r \in {\mathcal Hyp} \setminus \{1\}$, the manifold $N(r)$ has a fiber of arbitrarily large genus. 
Because of this, it is not clear 
which manifold $N(r)$ we should look in. 
Thus it is not a straightforward task to  identify  a primitive fibered class   $a_g \in H_2(N, \partial N)$  
such that  $\phi_{a_g} \in \mathcal{M}$  and  $\widehat{\delta}_g$  is achieved by  
$\widehat{\phi}_{a_g} \in \widehat{\mathcal{M}} \cap \mathrm{Mod}(\varSigma_g)$. 
Also it is not obvious at all that 
one of the boundary slopes of $a_g$ becomes a constant for large $g$. 
(As we will see, one of the boundary slopes of $a_g$ must be in 
$\{-4, \frac{3}{-2}, \frac{1}{-2}, 2 \}$ for large $g$.) 
The key observation to prove Theorem~\ref{thm_main} is 

\begin{thm}
\label{thm_key}
For $r \in {\mathcal Hyp}$, 
let  $\Omega$ be any fibered face of $N(r)$  which enjoys the following. 
\begin{quote}
$(*)$ 
Let $a \in S_{\beta}(r)$ be a  primitive fibered class of $N$ 
such that $\overline{a} \in int(C_{\Omega})$. 
Let $\Phi_a: F_a \rightarrow F_a$ be the monodromy of the fibration associated to $a$. 
Then the stable foliation $\mathcal{F}_a$ of $\Phi_a$ has the property such that 
any boundary component of $F_a$ lying on $T_{\beta}$ has no $1$ prong. 
\end{quote}
Then 
\begin{enumerate}
\item[(1)]
	$\min \mathrm{Ent}(N(1), \Omega) = 2 \log \delta(D_4) \approx 1.6628$, 
\item[(2)] 
	if  $r=-4, \tfrac{3}{-2}, \tfrac{1}{-2}, 2$, 
	then  $\min \mathrm{Ent}(N(r), \Omega) =  2 \log \delta(D_3) = 2 \log( \tfrac{3+ \sqrt{5}}{2}) \approx 1.9248$, 
	and 
\item[(3)] 
	if  $r\ne-4, \tfrac{3}{-2}, \tfrac{1}{-2}, 1,  2$, 
	then  $\min \mathrm{Ent}(N(r), \Omega) >1.97475$.  
\end{enumerate}
\end{thm}

\noindent
We remark here that 
for any $r \in {\mathcal Hyp}$, 
there exists  $\Omega$  having the condition $(*)$ (Proposition~\ref{prop_goodFiberFace}).  
Thus,  
there are no slopes $r$ for which we cannot apply Theorem~\ref{thm_key}.  
Also, 
since the three claims in Theorem~\ref{thm_key} cover  
all slopes  $r \in {\mathcal Hyp}$,  
the conclusions we could draw from Theorem~\ref{thm_key}  
are expected to be fairly sharp. 

To see this, 
consider the set of pairs 
\begin{equation*} 
	\mathcal{D} = \{ (N(r), \Omega)  \, | \, r \in {\mathcal Hyp}, \, 
	\Omega \; \; \text{is a fibered face of  $N(r)$  with $(*)$} \}. 
\end{equation*} 
Theorem~\ref{thm_key} shows for instance that 
there exist 
a minimum and a second minimum  of  
$$\{ \min \mathrm{Ent}(N(r), \Omega) \, | \, (N(r), \Omega) \in \mathcal{D} \}$$ 
and 
they are $2 \log \delta(D_4)$ and $2 \log \delta(D_3)$ respectively. 
The minimum is attained only by $N(1)$ and the second minimum  is attained  by $N(r)$ for  $r \in \{ -4, \frac{3}{-2}, \frac{1}{-2}, 2\}$.  
Furthermore, 
if  $r \not\in \{ -4, \frac{3}{-2}, \frac{1}{-2}, 1, 2 \}$,  
then Theorem~\ref{thm_key}(3) says that $\min \mathrm{Ent}(N(r), \Omega) $ 
is greater than the second minimum  with a uniform gap.  

The condition  $(*)$  on one boundary component,  
$T_{\beta}$,  of  $N$  
is a weaker version of the condition  (\ref{equation_1plong})  
on all three boundary components.  
If  $\Omega$  enjoys  $(*)$,  
then the dilatation of $a \in S_{\beta}(r)$  for  $N$ equals the dilatation of  $\overline{a}$  for $N(r)$. 
Thus one can compute the dilatation of $\overline{a}$ by using  
the Teichm\"{u}ller polynomial of $N$. 
In Section~\ref{subsection_Ent-symmteries},  
we shall see that 
the entropy function    for  $N$  has symmetries. 
This property together with the strict concavity of $\tfrac{1}{\mathrm{ent}}$ works well 
in the proof of Theorem~\ref{thm_key}.

\begin{proof}[Outline of the proof of Theorem~\ref{thm_main}] 
It is known that $N(-4) \simeq N( \tfrac{3}{-2})$ and 
$N(1) \simeq$ the Whitehead link exterior, see \cite{MP}. 
Recall that  $a_g$ is a primitive fibered class of  $H_2(N, \partial N)$  
such that  $\phi_{a_g} \in \mathcal{M}$  and  $\widehat{\delta}_g$  is achieved by  
$\widehat{\phi}_{a_g} \in \widehat{\mathcal{M}} \cap \mathrm{Mod}(\varSigma_g)$. 
There exists such a fibered class $a_g$ for any  $g \geq 3$. 
(In fact, the existence of the fiber of the fibration of genus $g$ for any  $g \geq 3$  is 
guaranteed by Theorem~\ref{thm_three}.   
By \cite[Lemma~4.7]{KT1}, one sees that 
the monodromy of this fibration is in the set $\mathcal{M}$.)

Since we know from the computation that  $N(1)$  has no fiber of genus greater than $1$, 
$a_g$  does not have a boundary slope  $1$  for  $g \geq 2$.  
On the other hand, 
each of three manifolds $N(-4) \simeq N(\tfrac{3}{-2})$, $N(\tfrac{1}{-2})$ and $N(2)$  has 
a fiber of genus  $g$  for large  $g$.  

Now,  
if we fill two other cusps of  $N(r)$, 
the normalized entropy  of $\widehat{\phi}_{a_g}$  decreases 
from that of  $\phi_{a_g}$  
and we have to consider its defect.  
We will show that the set of normalized entropies of monodromies 
of the fibrations on the closed manifolds, 
obtained from  $N$  by Dehn filling all cusps along the slopes 
not in $\{-4, \frac{3}{-2}, \frac{1}{-2}, 2 \}$, 
have no accumulation values  $\leq 2 \log (\frac{3+\sqrt{5}}{2})$. 
Thus,  
one sees that  $a_g$  has to have a boundary slope in   $\{-4, \frac{3}{-2}, \frac{1}{-2}, 2 \}$  eventually.  
Moreover the set of normalized entropies of the monodromies of the fibrations 
on the closed manifolds obtained from  $N$  by Dehn filling all cusps along the slopes, 
one of which is in $\{-4, \frac{3}{-2}, \frac{1}{-2}, 2 \}$,  
have no accumulation values  $< 2 \log (\frac{3+\sqrt{5}}{2})$. 
This together with  
Theorem~\ref{thm_three}  implies  Theorem~\ref{thm_main}(1). 

 The proof of Theorem~\ref{thm_main}(1) together with a claim in \cite{KT1} leads to Theorem~\ref{thm_main}(2). 
 For more details of the proofs, see Sections~\ref{subsection_Proof_thm_main1}, \ref{subsection_Proof_thm_main2}. 
\end{proof}

Based on the study of the magic manifold above, 
we propose conjectures. 
(The first half of Conjecture~\ref{conj_closed}(1),(2) is also stated in \cite[Question~1.12]{Hironaka}.)    

\begin{conj}
\label{conj_closed}
\ 
\begin{enumerate}
\item[(1)] 
We have 
$\displaystyle \lim_{g \to \infty} g  \log \delta_g = \log( \tfrac{3+ \sqrt{5}}{2})$. 
For large $g$, 
$\delta_g$ is  achieved by the monodromy 
of some $\varSigma_g$-bundle over the circle  
obtained from either $N(\tfrac{3}{-2})$ or $N(\tfrac{1}{-2})$  by  
Dehn filling both cusps. 

\item[(2)] 
We have $\displaystyle \lim_{\substack{g \not\equiv 0 \pmod{6} \\ g \to \infty}}  g  \log \delta^+_g = \log( \tfrac{3+ \sqrt{5}}{2}) $. 
For large $g$ such that $g \not\equiv 0 \pmod{6}$, 
$\delta^+_g$ is achieved by the monodromy of some $\varSigma_g$-bundle over the circle 
obtained from $N(\tfrac{3}{-2})$ or $N(\tfrac{1}{-2})$  by Dehn filling both cusps.  
\end{enumerate}

\end{conj}

\begin{conj}
\label{conj_torus} 
We have 
$\displaystyle \lim_{n \to \infty} n \log \delta_{1,n} =  2 \log \delta(D_4)$. 
For large $n$, 
$\delta_{1,n}$ is   achieved by 
the monodromy of a 
fibration on $N(1)$. 
\end{conj}

\subsection{Organization of the paper}

In Section~\ref{section_Magic}, first we describe properties of the entropy function for $N$. 
Next we construct the Thurston norm ball of $N(r)$. 
Finally we discuss the Thurston norm equivalence and entropy equivalence on the manifolds $N(r)$. 
In Section~\ref{section_Proof} we prove main results. 
In Section~\ref{section_Cusped-manifolds} 
we exhibit the computation of   $\min \mathrm{Ent}$ for some manifolds $N(r)$ which appeared in 
Gabai-Meyerhoff-Milley's work (Theorem~\ref{thm_GMM}, Table~\ref{table_2cusp_census_MNE}). 
We also  exhibit the normalized entropy of the monodromy of a fibration on each $1$-cusped hyperbolic fibered $3$-manifold 
with volume  at most $2.848$.  
\medskip

\noindent
{\bf Acknowledgments.} 
We would like to thank the referee for the careful reading of the paper and for valuable comments and suggestions. 
Due to the referee's efforts, the exposition of the paper has improved greatly.

\section{Magic manifold} 
\label{section_Magic}

\subsection{Fibered face} 
\label{subsection_Fibered}

Recall that  $\Delta$ is the fibered face of $N$ as in Section~\ref{subsection_TTM}. 
The open face $int(\Delta)$ is written by 
\begin{equation}
\label{equation_int-delta}
 int(\Delta) = \{ (x,y,z)\ |\ x+y-z =1, \ x >0,\  y>0,\   x >z,\   y>z\}.
 \end{equation}
 The Thurston norm of  $(x,y,z) \in int(C_{\Delta})$ is given by $ x+y-z$. 
We  recall some formulas in Lemmas~\ref{lem_topological-type} and \ref{lem_sing-data}. 
Lemma~\ref{lem_sing-data} tells us 
the singularity data of the stable foliation $\mathcal{F}_{a}$ for a primitive fibered class $ a \in int(C_{\Delta})$. 
First of all, we  explain that 
one can  compute the dilatation $\lambda(a)$ and  the singularity data of the stable foliation $\mathcal{F}_{a}$ 
for {\it any}  primitive fibered class  $a \in H_2(N, \partial N)$ 
by using the symmetries of $H_2(N, \partial N)$.

We consider a homeomorphism (in fact, a rotation map) 
$$h: (S^3, \mathcal{C}_3) \rightarrow (S^3, \mathcal{C}_3)$$ 
which sends $K_{\alpha}$, $K_{\beta}$, $K_{\gamma}$ to $K_{\beta}$, $K_{\gamma}$, $K_{\alpha}$ respectively, 
see Figure~\ref{fig_poly}(right).  
Then $h$ induces the isomorphism 
$h_*: H_2(N, \partial N) \rightarrow H_2(N, \partial N) $ 
of order $3$ which sends 
$\alpha$, $\beta$, $\gamma$ to $\beta$, $\gamma$, $\alpha$ respectively.

Let us pick the two fibered faces 
$\Delta_1$ with the vertices 
$(0,0,1)$, $(1,1,1)$, $(1,0,0)$, $(0,-1,0)$,  
and 
$\Delta_2$ with the vertices 
$(0,1,0)$, $(1,1,1)$, $(0,0,1)$, $(-1,0,0)$, 
see Figure~\ref{fig_poly}(left).  
We denote the opposite fibered faces of $\Delta$, $\Delta_1$, $\Delta_2$ by 
$\Delta'$, $\Delta'_1$, $\Delta'_2$ respectively. 
Consider the set 
$$Int\ C= \bigcup_{\widehat{\Delta}} int (C_{\widehat{\Delta}}),$$
where $\widehat{\Delta}$ is taken over all fibered face of $N$. 
We define the map 
$\sigma: Int\ C \rightarrow int(C_{\Delta})$ as follows. 
For $a = (x,y,z) \in Int\ C$, 
\begin{eqnarray*}
\sigma(a)&=& a  \hspace{2mm} \mbox{if}\ a \in int(C_{\Delta}),
\\
\sigma(a)&=& h_*(a)= (z,x,y) \hspace{2mm} \mbox{if}\ a \in int(C_{\Delta_1}), 
\\
\sigma(a)&=& (h^2)_*(a)= (y,z,x)\hspace{2mm} \mbox{if}\ a \in int(C_{\Delta_2}),\ \mbox{and}
\\
\sigma(a)&=& \sigma(-a)\hspace{2mm} \mbox{if}\ a \in int(C_{\Delta'}) \cup int(C_{\Delta_1'}) \cup int(C_{\Delta_2'}),
\end{eqnarray*}
where $h^2= h \circ h$, and 
$(h^2)_*: H_2(N, \partial N) \rightarrow H_2(N, \partial N) $ 
is the isomorphism induced from $h^2$. 
Clearly, 
$a \in H_2(N, \partial N)$ is a fibered class if and only if $-a \in H_2(N, \partial N)$ is a fibered class. 
In this case the inverse $(\Phi_a)^{-1}$ of the monodromy $\Phi_a$ of the fibration on $N$ associated to $a$ 
is isotopic to the monodromy $\Phi_{-a}$ of the fibration on $N$  associated to $-a$. 
In particular $\lambda(a)= \lambda(-a)$. 
Moreover the singularity datum of $\mathcal{F}_a$ and $\mathcal{F}_{-a}$ are the same.

Let us assume that $a$ is a primitive fibered class such that 
$a  \in int(C_{\Delta_1})$ (resp. $a \in int(C_{\Delta_2})$). 
Then two  fibered classes $a$ and $\sigma(a)\in int(C_{\Delta})$ have the fiberes 
$F_a$ and $F_{\sigma(a)}$ with the same topology, and 
the monodromies $\Phi_a$ and $\Phi_{\sigma(a)}$ are conjugate. 
This is because 
the isomorphism $h_*$ (resp. $(h^2)_*$)  
is coming from the homeomorphism on the pair $(S^3, \mathcal{C}_3)$. 
In particular $\lambda(a)= \lambda(\sigma(a))$, and 
$\lambda(a)$ is the largest real root of $f_{(z,x,y)}(t)$ (resp. $f_{(y,z,x)}(t)$), 
see (\ref{equation_TpolyMagic}). 
%Although $\Phi_a$ and $\Phi_{\sigma(a)}$ are conjugate, 
Notice that 
the conjugacy homeomorphism $g: F_a \rightarrow F_{\sigma(a)}$ 
between $\Phi_a: F_a \rightarrow F_a$ and $\Phi_{\sigma(a)}: F_{\sigma(a)} \rightarrow F_{\sigma(a)}$ 
permutes the boundary components of the fiber. 
More precisely, 
$g$ maps the boundary components of $F_a$ which lie on $T_{\alpha}$, $T_{\beta}$, $T_{\gamma}$ 
to the boundary components of $F_{\sigma(a)}$ which lie on  
$T_{\beta}, \ T_{\gamma}, \ T_{\alpha}$   
(resp. $T_{\gamma}, \ T_{\alpha}, \ T_{\beta}$). 
Thus, to apply Lemma~\ref{lem_sing-data} below  for such a primitive fibered class $a$ in $int(C_{\Delta_1})$ (resp. $int(C_{\Delta_2})$), 
first apply the lemma for $\sigma(a) \in int(C_{\Delta})$. 
Then translate the claim into the one for the fibered class $a$ 
by permuting  the boundary components of the fiber.

\begin{lem}
\label{lem_topological-type}
Let  $a=  (x,y,z) $ be a primitive fibered class in $H_2(N, \partial N)$. 
Then $\sharp(\partial F_{a})$ equals 
%the sum of the three greatest common divisors 
$$\gcd(x,y+z)+ \gcd(y,z+x)+ \gcd(z,x+y),$$ 
where $\gcd(0,w)$ is defined by $|w|$.  
More precisely 
$$\sharp( \partial_{\alpha} F_{a}) = \gcd(x,y+z), \ \sharp (\partial_{\beta} F_{a} )=  \gcd(y,z+x),\ \sharp( \partial_{\gamma} F_{a}) = \gcd(z,x+y).$$
\end{lem}

\begin{proof}
The proof in the case $a \in int(C_{\Delta})$ can be found in \cite[Lemma~3.1]{KT}. 
Because of the symmetries of $H_2(N, \partial N)$, 
the formula for primitive fibered classes over $\Delta$ can be extended to any primitive fibered classes in $ H_2(N, \partial N)$. 
\end{proof}

\begin{lem}[Lemma~3.1 in \cite{KT1}]
\label{lem_sing-data}
Let  $a= (x,y,z) $ be a primitive fibered class in $ int(C_{\Delta})$. 
The stable foliation  $\mathcal{F}_{a}$ of the monodromy $\Phi_a$ has the property such that 
each component  of $\partial_{\alpha} F_{a}$, $\partial_{\beta} F_{a}$ and  $\partial_{\gamma} F_{a}$ 
has $\tfrac{x}{\gcd(x,y+z)}$ prongs, $\tfrac{y}{\gcd(y,x+z)}$ prongs and $\tfrac{x+y-2z}{\gcd(z,x+y)}$ prongs respectively. 
 Moreover $\mathcal{F}_{a}$  does not have singularities in the interior of $F_{a} $. 
\end{lem}

For a rational class $a=(x,y,z) \in H_2(N, \partial N; {\Bbb R})$, 
let us define $\tfrac{p_{\alpha}(a)}{q_{\alpha}(a)}$, $\tfrac{p_{\beta}(a)}{q_{\beta}(a)}$, $\tfrac{p_{\gamma}(a)}{q_{\gamma}(a)}$ 
as follows. 

$$\mbox{slope}(a)= (b_{\alpha}(a), b_{\beta}(a), b_{\gamma}(a)) = 
(\tfrac{p_{\alpha}(a)}{q_{\alpha}(a)}, \tfrac{p_{\beta}(a)}{q_{\beta}(a)}, \tfrac{p_{\gamma}(a)}{q_{\gamma}(a)} ), $$ 
where  
$\tfrac{p_{\alpha}(a)}{q_{\alpha}(a)}$, $\tfrac{p_{\beta}(a)}{q_{\beta}(a)}$, $ \tfrac{p_{\gamma}(a)}{q_{\gamma}(a)}$ are irreducible forms 
so that $p_{\alpha}(a)$, $p_{\beta}(a)$, $p_{\gamma}(a) \in {\Bbb N}$. 
%Each of $b_{\alpha}(a), b_{\beta}(a)$, and $ b_{\gamma}(a)$ is called the {\it boundary slope} of $a$. 
%If two rational classes $a$ and $a'$ lie on the same ray from the origin, then $\mbox{slope}(a)= \mbox{slope}(a')$. 

\begin{rem}
\label{rem_slope-gam}
Suppose that a rational class $a= (x,y,z)$ is an element of $int(C_{\Delta})$. 
Then $x>0$, $y>0$, $x>z$ and $y >z$ from (\ref{equation_int-delta}). 
Thus if $z \ne 0$ then 
$b_{\gamma}(a) \in (-\infty, -2)$ or $b_{\gamma}(a) \in (0, \infty)$. 
In particular $p_{\gamma}(a)+ 2 q_{\gamma}(a)>0$ and $p_{\gamma}(a)+  q_{\gamma}(a)>0$. 
\end{rem}

\subsection{Entropy function with symmetries} 
\label{subsection_Ent-symmteries}

In this subsection we will see that the entropy function for $N$ possesses symmetries. 
Some claims given here  play an important role in the proof of Theorem~\ref{thm_key}.

Before we state Lemma~\ref{lem_conjugateXY}, we note that 
when $(x,y,z) $ is a primitive fibered class in $ int(C_{\Delta})$, then 
$(y,x,z) $ is also a primitive fibered class in $ int(C_{\Delta})$. 
The topological types of the fibers  $F_{(x,y,z)}$ and $F_{(y,x,z)} $ are the same 
by Lemma~\ref{lem_topological-type}.

\begin{lem}
\label{lem_conjugateXY}
Let $(x,y,z) $ be a primitive fibered class in $ int(C_{\Delta})$. 
Then the inverse $(\Phi_{(x,y,z)})^{-1}$ of the monodromy 
$\Phi_{(x,y,z)}: F_{(x,y,z)} \rightarrow F_{(x,y,z)}$ of the fibration on $N$ associated to $(x,y,z)$ 
 is conjugate to the monodromy 
$\Phi_{(y,x,z)}: F_{(x,y,z)} \rightarrow F_{(y,x,z)}$ of the fibration on $N$ associated to $(y,x,z) \in int(C_{\Delta})$. 
In particular $\lambda_{(x,y,z)}= \lambda_{(y,x,z)}$. 
\end{lem}

\begin{proof}
Let us denote by $\mathcal{C}_3^{-}$, 
the $3$ chain link such that 
the orientation of each component is opposite to each one for $\mathcal{C}_3$. 
We denote the components of $\mathcal{C}_3^{-}$ by $K_{\alpha}^{-}$, $K_{\beta}^{-}$ and $K_{\gamma}^{-}$. 
There exists a homeomorphism 
$i_{\gamma}: (S^3, \mathcal{C}_3) \rightarrow (S^3, \mathcal{C}_3^{-})$ 
which sends $K_{\alpha}$, $K_{\beta}$, $K_{\gamma}$ to 
$K_{\beta}^{-}$, $K_{\alpha}^{-}$, $K_{\gamma}^{-}$ respectively. 
Then $i_{\gamma}$ induces the isomorphism 
$(i_{\gamma})_*: H_2(N, \partial N) \rightarrow H_2(N, \partial N) $ 
which sends 
$\alpha$, $\beta$, $\gamma$ to $- \beta$, $-\alpha$, $-\gamma$ respectively. 
If we take $a= (x,y,z) \in int(C_{\Delta})$, then 
$(i_{\gamma})_*(a)= (-y,-x,-z) \in int(C_{\Delta'})$. 
Since $(i_{\gamma})_*$ is induced by the homeomorphism $i_{\gamma}$, 
the monodromies $\Phi_a$ and $\Phi_{(i_{\gamma})_*(a)}$ must be conjugate. 
(Hence $\Phi_a^{-1}$ and $(\Phi_{(i_{\gamma})_*(a)})^{-1}$ are conjugate.) 
On the other hand, 
  $(\Phi_{(i_{\gamma})_*(a)})^{-1}$ is isotopic to the monodromy $\Phi_{-(i_{\gamma})_*(a)}$ of the fibration on $N$ 
associated to $-(i_{\gamma})_*(a)= (y,x,z) \in int(C_{\Delta})$. 
Thus $(\Phi_{(x,y,z)})^{-1}$ and $\Phi_{(y,x,z)}$ are conjugate. 
\end{proof}

Observe that if  $(x,y,z) \in int(C_{\Delta})$, then 
$(y-z,y, y-x), (y-z,x-z,-z), (x,x-z, x-y) \in int(C_{\Delta})$. 
These four classes have the same Thurston norm.

\begin{lem} 
\label{lem_symmetry}
The four classes 
$(x,y,z)$, $(y-z,y, y-x)$, $(y-z,x-z,-z)$ and $(x,x-z, x-y) $ of $int(C_{\Delta})$  have  the same entropy. 
\end{lem}

\begin{proof} 
One sees that 
 $f_{(y-z,y, y-x)}(t)$, $ f_{(y-z,x-z,-z)}(t)$ and $ f_{(x,x-z, x-y)}(t)$ are equal to the same polynomial 
$ f_{(x,y,z)}(t) $. 
\end{proof}

\begin{rem} 
If $(x,y,z)$ is a primitive fibered class in  $ int(C_{\Delta})$, 
then the other three classes  in Lemma~\ref{lem_symmetry} are also primitive. 
Although these classes have the same Thurston norm, 
the topological types of their  minimal representatives  may be different. 
\end{rem}

Let $(x,y,z) \in \Delta$. 
Since $x+y-z=1$, one may represent $(x,y,z)$ without $z$.  
Let us denote the class $(x,y,z) $ by $[x,y]$. 
Then the open face $int(\Delta)$ can be written by 
$$int(\Delta)= \{[x,y]\ |\ 0 < x < 1,\ 0 < y < 1\}$$
(Figure~\ref{fig_parameter}).  
We shall see in Remark~\ref{rem_symmetry-ent} that 
this parametrization for the points of $int(\Delta)$ makes it easy to see the symmetry of the entropy function for $N$. 
We denote by $\lambda_{[x,y]}$ the dilatation of  $[x,y] \in int (\Delta)$. 
By Lemma~\ref{lem_symmetry} one obtains

\begin{cor}
\label{cor_symmetry}
Let $(x,y,z) \in int(C_{\Delta})$. 
Then $[\tfrac{x}{x+y-z}, \tfrac{y}{x+y-z}]$, $[\tfrac{y-z}{x+y-z}, \tfrac{y}{x+y-z}]$, $[\tfrac{y-z}{x+y-z}, \tfrac{x-z}{x+y-z}]$ and 
$[\tfrac{x}{x+y-z}, \tfrac{x-z}{x+y-z}]$  have the same entropy. 
(See Figure~\ref{fig_symmetry}.) 
\end{cor}

\begin{rem}
\label{rem_symmetry-ent}
Corollary~\ref{cor_symmetry} says that any two classes of $int(\Delta)$ having a line symmetry about  $x = \tfrac{1}{2}$ (resp. $y = \tfrac{1}{2}$) 
have the same entropy. 
In addition by Lemma~\ref{lem_conjugateXY}, 
$\lambda_{(x,y,z)}= \lambda_{(y,x,z)}$ holds  for  $(x,y,z) \in int(C_{\Delta})$. 
This implies that any two classes $a= [x,y], \widetilde{a} \in [y,x] \in int(\Delta)$ with a line symmetry about $y=x$ have the same entropy. 
Putting all things together, one has another line symmetry about  $y= -x+1$ for the entropy function of $N$. 
Thus $8$ classes  $ b_0, \widetilde{b_0}, \cdots, b_3,\widetilde{b_3} \in  int(\Delta)$ as in Figure~\ref{fig_symmetry} have the same entropy. 
\end{rem}

\begin{figure}
\begin{center}
\includegraphics[width=1.7in]{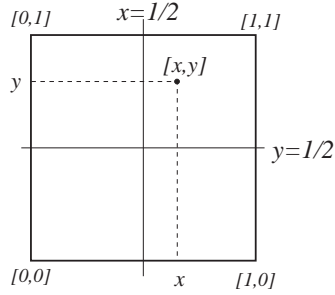}
\caption{$[x,y] \in int(\Delta)$.}
\label{fig_parameter}
\end{center}
\end{figure}

\begin{figure}
\begin{center}
\includegraphics[width=3.5in]{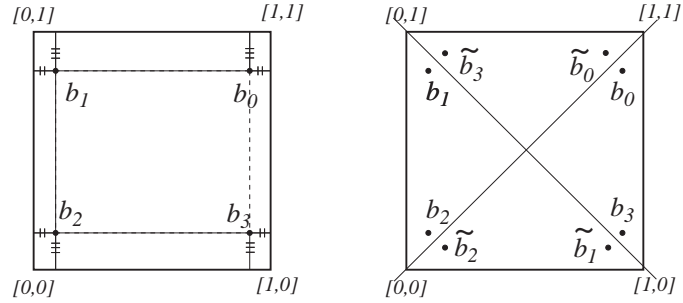}
\caption{
$b_0 = [\tfrac{x}{x+y-z},  \tfrac{y}{x+y-z}]$, $b_1= [\tfrac{y-z}{x+y-z}, \tfrac{y}{x+y-z}]$,  
$b_2= [\tfrac{y-z}{x+y-z}, \tfrac{x-z}{x+y-z}]$, $b_3= [\tfrac{x}{x+y-z}, \tfrac{x-z}{x+y-z}]  \in int(\Delta)$ and $\widetilde{b_i} \in int(\Delta)$.}
\label{fig_symmetry}
\end{center}
\end{figure}

By Corollary~\ref{cor_symmetry}, one obtains 

\begin{lem}
\label{lem_concave_sym}
Fix $0 < x_0 < 1$, $0 < y_0 < 1$ and  $0 <  c <2$. 
\begin{enumerate}
\item[(1)] 
$\lambda_{[x_0, \tfrac{1}{2} - t]} = \lambda_{[x_0, \tfrac{1}{2} + t]}$ for $0 \le t < \tfrac{1}{2}$, and 
$\lambda_{[x_0, \tfrac{1}{2}]} = \min \{ \lambda_{[x_0, y]}\ |\ 0 < y <1  \}$. 

\item[(2)] 
$\lambda_{[\tfrac{1}{2} - t, y_0]} = \lambda_{[\tfrac{1}{2} + t, y_0]}$ for $0 \le t < \tfrac{1}{2}$, and 
$\lambda_{[\tfrac{1}{2} , y_0]} = \min \{ \lambda_{[x, y_0]}\ |\ 0 < x <1  \}$. 

\item[(3)] 
$\lambda_{[\tfrac{c}{2}+t, \tfrac{c}{2}-t]}= \lambda_{[\tfrac{c}{2}-t, \tfrac{c}{2}+t]}$ for $0 \le t < 1 - \tfrac{c}{2}$, and 
$$\lambda_{[\tfrac{c}{2}, \tfrac{c}{2}]} = \min \{\lambda_{[x,y]}\ |\  [x,y] \in int(\Delta), y = -x+c\}.$$
\end{enumerate}
\end{lem}

\begin{proof} 
We prove (3). 
The first equality follows since $\lambda_{[x,y]}= \lambda_{[y,x]}$ for $0 < x < 1$ and $0 < y < 1$. 
The function $\tfrac{1}{\log \lambda}$ restricted to the set $\{[x,y] \in int(\Delta)\ |\ y = -x+c\}$ is strictly concave.  
This together with the first equality implies that $[\tfrac{c}{2}, \tfrac{c}{2}]$ reaches a minimum. 

The proofs of (1),(2) are similar to that of (3). 
\end{proof}

\noindent
By using Lemma~\ref{lem_concave_sym} one sees that 
the center $[\tfrac{1}{2}, \tfrac{1}{2}]  \in int(\Delta)$ achieves  $\min \mathrm{Ent}(N, \Delta)$. 
Because of the symmetries of $H_2(N, \partial N)$, 
the equality 
$\min \mathrm{Ent}(N, \Delta)= \min \mathrm{Ent}(N, \widehat{\Delta})$ 
holds for any fibered face $\widehat{\Delta}$. 
Thus one has

\begin{prop}
\label{prop_MNE-magic}
$\min \mathrm{Ent}(N)= \mathrm{Ent}([\tfrac{1}{2}, \tfrac{1}{2}])= 2 \log (2+ \sqrt{3}) \approx 2.6339$. 
\end{prop}

\noindent
By Proposition~\ref{prop_MNE-magic}, one sees the following: 
When $[x,y] \in int(\Delta)$ such that $[x,y] \ne [\tfrac{1}{2}, \tfrac{1}{2}]$, 
$$\log \lambda_{[x,y]} > \log \lambda_{[\tfrac{1}{2}, \tfrac{1}{2}]} = 2 \log (2+ \sqrt{3})> 2.633.$$

\subsection{Thurston norm of manifolds $N(r)$}

Let  $N(r)$  be the manifold obtained from the magic manifold  $N$  
by Dehn filling 
the cusp specified by the torus  $T_{\beta}$  along the slope  $r$, 
and  $D(r)$  an attached solid torus in  $N(r)$  
so that  $\partial D(r) = T_{\beta}$.  
Consider the exact sequence of the homology group of the triple  
$(N(r), \partial N(r) \cup D(r), \partial N(r))$  with real coefficients,   
\begin{equation*}
	\cdots \to 
	H_2(N(r), \partial N(r)) \overset{j}\to 
	H_2(N(r), \partial N(r) \cup D(r)) \overset{\partial}\to 	
	H_1(\partial N(r) \cup D(r), \partial N(r)) \to \cdots
\end{equation*} 
The first homomorphism  $j$  is injective 
since  $H_2(\partial N(r) \cup D(r), \partial N(r)) = 0$.  
Also by excision, 
we have an isomorphism 
\begin{equation*}
	e : H_2(N, \partial N) \to H_2(N(r), \partial N(r) \cup D(r)).  
\end{equation*} 
Notice that the composition  
\begin{equation*} 
	\partial \circ e : H_2(N, \partial N) \to H_1(\partial N(r) \cup D(r), \partial N(r)) \cong \mathbb{Z}
\end{equation*}   
can be identified with the intersection number for a cycle in  
$H_2(N, \partial N)$  with a slope  $r$  on   $\partial D(r) = T_{\beta}$.  

On the other hand, 
since the composition of the boundary map with 
a quotient homomorphism 
\begin{equation*} 
	H_2(N, \partial N) \overset{\partial}\to H_1(\partial N) \to 
	H_1(\partial N) / H_1(T_{\alpha} \cup T_{\gamma}) \cong H_1(T_{\beta}) 
\end{equation*} 
sends  $\alpha$  and  $\gamma$  to the minus meridian on  $T_{\beta}$ (see Figure~\ref{fig_meridian})  and  $\beta$ to a longitude, 
the kernel of  $\partial \circ e$  is identified with 
\begin{equation*}
	S_{\beta}(r) = \{ (x, y, z) \in H_2(N, \partial N) \, | \, -ry = x+z \}.   
\end{equation*}
Thus,  
we have proved 

\begin{prop}
\label{prop_S_r}
Take a slope  $r \in \mathbb{Q}$  on a boundary torus for  $N$,  
say  $T_{\beta}$.  
Let  $N(r)$  be the manifold obtained from $N$ by Dehn filling the cusp specified by  $T_{\beta}$  
along the slope  $r$.  
Then there is a natural injection  
\begin{equation*} 
	\iota_{\beta} = e^{-1} \circ j : H_2(N(r), \partial N(r)) \to H_2(N, \partial N) 
\end{equation*} 
such that  ${\rm Im} \, \iota_{\beta} = S_{\beta}(r)$.  
\end{prop}

\begin{figure}
\begin{center}
\includegraphics[width=1in]{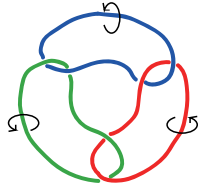}
\caption{meridians of the components of $\mathcal{C}_3$}
\label{fig_meridian}
\end{center}
\end{figure}

\noindent
For  $a = (x,y,z) \in S_{\beta}(r)$, 
we denote by $\overline{a}= \overline{(x,y,z)}$, 
the element of $H_2(N(r), \partial N(r))$ 
such that $\iota_{\beta}(\overline{a}) = a$. 
We sometimes denote $N(r)$ by $N_{\beta}(r)$ 
when we need to specify the cusp which is filled. 
By using this notation, we may write 
$\overline{a} \in H_2(N_{\beta}(r), \partial N_{\beta}(r))$.

Similarly, 
when  $N(r)$  is the manifold obtained from $N$ by Dehn filling the cusp specified 
by  $T_{\alpha}$  or  $T_{\gamma}$  along the slope  $r$, 
one has natural injections, 
\begin{align*}
	\iota_{\alpha} & : H_2(N(r), \partial N(r)) \to H_2(N, \partial N), \\ 
	\iota_{\gamma} & : H_2(N(r), \partial N(r)) \to H_2(N, \partial N) 
\end{align*}
such that their images are 
\begin{align*}
	S_{\alpha}(r) & = \{ (x, y, z) \in H_2(N, \partial N) \, | \, -rx = y+z \}, \\
	S_{\gamma}(r) & = \{ (x, y, z) \in H_2(N, \partial N) \, | \, -rz = x+y \}.  
\end{align*}
We also denote by $N_{\alpha}(r)$ or $N_{\gamma}(r)$, the manifold $N(r)$ in this case.

Hereafter we denote the Thurston norm of $N$ by $\| \cdot \|$ and 
its Thurston norm ball with radius $d$  by $B(d)$. 
(Hence $U_N= B(1)$.) 
The entropy function and the normalized entropy function of $N$ are denoted  
by $\mathrm{ent}$ and $\mathrm{Ent}$ respectively as usual.  
We also denote the Thurston norm of $N(r)$  by $\|\cdot \|_{r}$ and 
the Thurston norm ball with radius $d$  by $B_{r}(d)$. 
The dilatation, entropy function and the normalized entropy function of $N(r)$ are denoted by 
$\lambda_r$, $\mathrm{ent}_{r}$  and  $\mathrm{Ent}_{r}$  respectively.

Let us define the set $\widehat{B}_{\beta, r}(1)$ to be $\widehat{B}_{\beta, r}(1)= B(1) \cap  S_{\beta}(r)$, 
see Figure~\ref{fig_rokumentai}. 
It is   parallelogram when $r \in (-2,0)$ 
(resp. hexagons when $r \in (-\infty, -2) \cup (0, \infty)$).

\begin{figure}
\begin{center}
\includegraphics[width=5in]{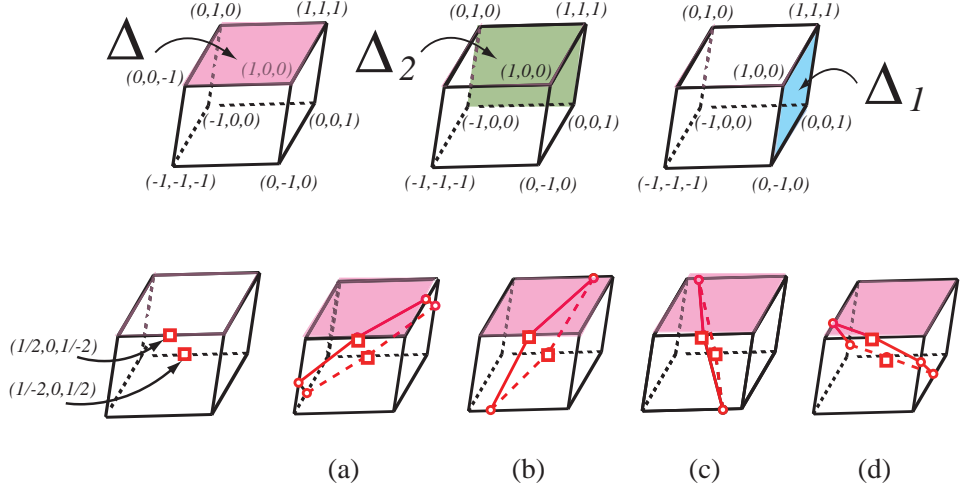}
\caption{(top) fibered faces $\Delta$, $\Delta_2$ and $\Delta_1$. 
(bottom) $\widehat{B}_{\beta,r}(1)$ in the case 
(a) $r \in (-\infty, -2)$, (b) $r \in (-2, -1)$, (c) $r \in (-1,0)$, (d) $r \in (0, \infty)$.}  
\label{fig_rokumentai}
\end{center}
\end{figure}

Now we consider the sets 
$\Delta \cap S_{\alpha}(r)$, $\Delta \cap S_{\beta}(r)$ and $\Delta \cap S_{\gamma}(r)$ for $r \in {\mathcal Hyp}$, see Figure~\ref{fig_DeltaCapS}. 
Note  that $\Delta \cap S_{\gamma}(r) \ne \emptyset$  if and only if 
$r \in  (-\infty, -2) \cup (0, \infty)$, see Remark~\ref{rem_slope-gam}. 

\begin{lem}
\label{lem_boundary}
\ 
\begin{enumerate}
\item[(1)] 
$\Delta \cap S_{\alpha}(r)$ is a segment 
$ \{[x,y] \in \Delta\ |\ y=( \tfrac{1+r}{-2})x+ \tfrac{1}{2}\}$. 
The set of its endpoints  equals 
\begin{enumerate}
\item[(i)] 
$\{[0, \tfrac{1}{2}], [\frac{-1}{1+r}, 1]\}$ 
when $r \in (- \infty, -2)$, 

\item[(ii)] 
$\{[0, \tfrac{1}{2}], [1, \tfrac{r}{-2}]\}$  
when $r \in (-2,0)$,

\item[(iii)] 
$\{[0, \tfrac{1}{2}], [\frac{1}{1+r}, 0]\}$ 
when $r \in (0,\infty)$. 
\end{enumerate}

\item[(2)] 
$\Delta \cap S_{\beta}(r)$ is a segment  $\{[x,y] \in \Delta\ |\ y=( \tfrac{-2}{1+r})x+ \tfrac{1}{1+r}\}$. 
The set of its endpoints  equals 

\begin{enumerate}
\item[(i)] 
$\{[\tfrac{1}{2}, 0], [1, \tfrac{-1}{1+r}] \}$ 
when $r \in (- \infty, -2)$, 

\item[(ii)]  
$\{[\tfrac{1}{2}, 0], [\tfrac{r}{-2}, 1] \}$ 
when $r \in (-2,0)$, 

\item[(iii)] 
$ \{[\tfrac{1}{2}, 0], [0, \tfrac{1}{1+r}] \}$ 
when $r \in (0,\infty)$. 
\end{enumerate}

\item[(3)] 
$ \Delta \cap S_{\gamma}(r)$ is a segment  $\{[x,y] \in \Delta\ |\ y= -x + \tfrac{r}{1+r}\}$ when  $r \in (-\infty, -2) \cup (0, \infty)$. 
In this case the set of its endpoints  equals  
\begin{enumerate}
\item[(i)] 
$\{[\tfrac{-1}{1+r}, 1], [1, \tfrac{-1}{1+r}]\}$ 
when $r \in (- \infty, -2)$, 

\item[(ii)] 
$\{[0, \tfrac{r}{1+r}], [\tfrac{r}{1+r}, 0]\}$ 
when $r \in (0, \infty)$.
\end{enumerate}
\end{enumerate}
\end{lem}

\begin{proof} 
We prove the claim (1). 
Let $a = (x,y,z) \in \Delta \cap S_{\alpha}(r)$. 
Then $\|a\| = x+y-z=1$ and $-rx = y+z$. 
Substituting  $z= x+y-1$ for   $-rx = y+z$, 
one obtains $y=( \tfrac{1+r}{-2})x+ \tfrac{1}{2}$. 
It is immediate to check (i),(ii),(iii). 

The proofs of (2),(3) are similar to that of (1). 
\end{proof}

\begin{figure}
\begin{center}
\includegraphics[width=5in]{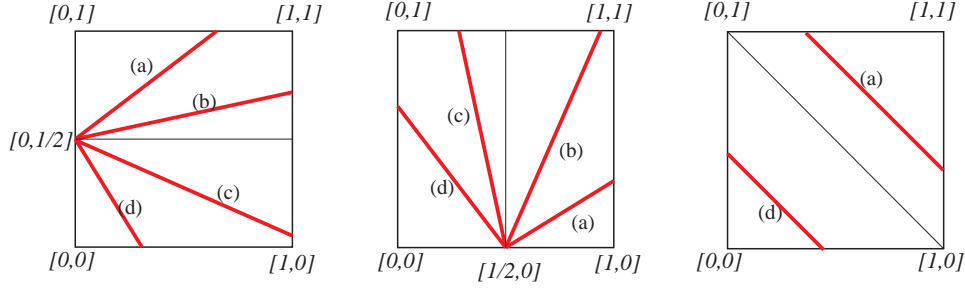}
\caption{(from left to right) $\Delta \cap S_{\alpha}(r)$, 
$\Delta \cap S_{\beta}(r)$, 
 $\Delta \cap S_{\gamma}(r)$. 
 [(a) $r \in (- \infty, -2)$, 
(b) $r \in (-2,-1)$, 
(c) $r \in (-1,0)$, 
(d) $r \in (0, \infty)$.]} 
\label{fig_DeltaCapS}
\end{center}
\end{figure}

\begin{rem}
\label{rem_SymmetrySetXY}
We note that 
$(\Delta \cap S_{\alpha}(r)) \cup (\Delta \cap S_{\beta}(r)) $ has a line symmetry about $y=x$.  
\end{rem}

\begin{lem}
\label{lem_slope1}
Suppose that one of the  boundary slopes of a rational class of $ H_2(N, \partial N)$ equals $1$. 
Then the other two  boundary slopes also  equal $1$. 
\end{lem}

\begin{proof} 
Because of the symmetries of $H_2(N, \partial N)$, 
it suffices to suppose that the rational class lives in $int(C_{\Delta})$. 
By Lemma~\ref{lem_boundary}, 
$\Delta \cap S_{\alpha}(1)= \Delta \cap S_{\beta}(1)=  \Delta \cap S_{\gamma}(1) $. 
This leads to  the lemma. 
\end{proof}

We present a formula for the Thurston norm  of $N(r)$ by using  the Thurston norm of $N$.

\begin{lem}
\label{lem_NormChange}
Let $p \in {\Bbb N}$ and $q \in {\Bbb Z}$ be coprime such that $ \tfrac{p}{q} \in {\mathcal Hyp}$. 
If $a =(x,y,z)\in  S_{\gamma}(\tfrac{p}{q})$, then 
the Thurston norm of $\overline{a} \in H_2(N_{\gamma}(\tfrac{p}{q}), \partial N_{\gamma}(\tfrac{p}{q}))$ 
equals $ \|a\| -  |\tfrac{z}{q}|$. 
In particular 
$$  \|\overline{a}\|_{p/q}  = 1 - \tfrac{1}{p+q}  \hspace{4mm}\mbox{if}\   a \in \Delta \cap S_{\gamma}(\tfrac{p}{q}).$$ 
\end{lem}

\begin{proof} 
Suppose that $a \in S_{\gamma}(\tfrac{p}{q})$ is an integral class. 
Then  $ \|\overline{a}\|_{p/q} $ equals $\|a\|$ minus the number of the boundary components of $F_a$  which lie on $T_{\gamma}$, that is 
$$\|\overline{a}\|_{p/q} = \|a\| - \gcd(z, x+y) =\|a\| - |\tfrac{z}{q}|.$$ 
The Thurston norm  $\|\cdot\|_{p/q} $ defined on integral classes admits a unique continuous extension to 
$H_2(N(\tfrac{p}{q}), \partial N( \tfrac{p}{q}); {\Bbb R})$. 
Thus the above formula holds for any class $a \in S_{\gamma}(\tfrac{p}{q})$. 

Suppose that $a \in \Delta \cap S_{\gamma}(\tfrac{p}{q})$. 
Then  $p+q>0$. 
One has 
$-pz = q(x+y)$ and $\|a\|= x+y-z=1$. 
Hence $-pz = q(1+z)$, and one obtains $z= \tfrac{q}{-(p+q)}$. 
Thus $ \|\overline{a}\|_{p/q}  = 1 -  |\tfrac{z}{q}|= 1 - \tfrac{1}{p+q} $. 
\end{proof}

\noindent
Similarly, we have: 

\begin{lem}
\label{lem_NormChange2}
If one fills the cusp of $N$ specified by  the torus $T_{\alpha}$ (resp. $T_{\beta}$) along the slope $\tfrac{p}{q}$, 
then the Thurston norm of $N_{\alpha}(\tfrac{p}{q})$ (resp. $N_{\beta}(\tfrac{p}{q})$) is given by 
\begin{eqnarray*}
 \|\overline{a}\|_{p/q} &=& \|a\| -  |\tfrac{x}{q}|\ \mbox{for} \ a= (x,y,z) \in S_{\alpha}(\tfrac{p}{q})
 \\
(\mbox{resp.}\   \|\overline{a}\|_{p/q} &=& \|a\| -  |\tfrac{y}{q}|\ \mbox{for}\  a= (x,y,z) \in S_{\beta}(\tfrac{p}{q})). 
\end{eqnarray*}
\end{lem}

\subsection{Thurston norm equivalence on manifolds $N(r)$}

Let $p \in {\Bbb N}$ and $q \in {\Bbb Z}$ be coprime such that $r= \tfrac{p}{q} \in {\mathcal Hyp}$. 
We shall investigate the shape of the Thurston norm ball of $N(r)= N_{\beta}(r)$.

First we take $ \mathfrak{a}_r, \mathfrak{b}_r \in S_{\beta}(r)$ as follows. 
\begin{eqnarray*}
\mathfrak{a}_r = (\tfrac{p+1}{2}, -q, \tfrac{p-1}{2}),\  
\mathfrak{b}_r = (\tfrac{p-1}{2}, -q, \tfrac{p+1}{2})\ &\mbox{if}&\ p\ \mbox{is\ odd}, 
\\
\mathfrak{a}_r = (\tfrac{p}{2}+1, -q, \tfrac{p}{2}-1),\  
\mathfrak{b}_r = (\tfrac{p}{2}, -q, \tfrac{p}{2})\ &\mbox{if}&\ p\ \mbox{is\ even}.
\end{eqnarray*}

\begin{lem}
\label{lem_basis}
The set $\{\overline{\mathfrak{a}_r}, \overline{\mathfrak{b}_r}\}$ is a basis of $H_2(N_{\beta}(r), \partial N_{\beta}(r); {\Bbb Z})$. 
\end{lem}

\begin{proof} 
By Proposition~\ref{prop_S_r},  ${\rm Im} \, \iota_{\beta}= S_{\beta}(r)$. 
Thus it is enough to show  that for any integral class $a= (x,y,z) \in S_{\beta}(r)$, 
there exist integers $k_0,\ell_0$ such that $a=k _0\mathfrak{a}_r  + \ell_0 \mathfrak{b}_r $. 
One has $-py= q(x+z)$. 
Since $p$ and $q$ are coprime, 
there exists an integer $t$ such that $x+z= pt$. 
Hence $z= pt-x$. 
Substitute $x+z= pt$ for $-py= q(x+z)$, then one obtains $ (x,y,z)= (x, -qt, pt-x)$. 
Now let us take $k_0=(\tfrac{1-p}{2}) t +x$, 
$\ell_0 = (\tfrac{1+p}{2})t-x$ if $p$ is odd 
(resp. $k_0= (\tfrac{-p}{2}) t+x$, $\ell_0= (1+\tfrac{p}{2})t-x$ if $p$ is even). 
One can check that $a= k _0\mathfrak{a}_r  + \ell_0 \mathfrak{b}_r $. 
\end{proof}

\begin{lem}
\label{lem_shape}
Let $r= \tfrac{p}{q}$ be as above. 
The Thurston norm ball of $N(r)$   is described by using $(\overline{\mathfrak{a}_r}, \overline{\mathfrak{b}_r})$ coordinates  as follows. 
(See Figure~\ref{fig_Ball_Nr}.) 
\begin{enumerate}

\item[(1)] 
Suppose that $r \in (- \infty, -2)$. 
\begin{enumerate}
\item[(i)]  
If $|q| (= -q) \ne 1$, then 
$B_r(p+q-1)$ is a hexagon with vertices  
\begin{eqnarray*}
\pm (\tfrac{p+2q+1}{2}, \tfrac{p+2q-1}{-2}), 
\pm (\tfrac{p+2q-1}{2}, \tfrac{p+2q+1}{-2}), 
\pm (\tfrac{p+q-1}{2}, \tfrac{p+q-1}{-2})&\ &\mbox{when}\ p\ \mbox{is\ odd}, 
\\
\pm (\tfrac{p+2q}{2}, \tfrac{p+2q-2}{-2}), 
\pm (\tfrac{p+2q}{2}, \tfrac{p+2q+2}{-2}), 
\pm (\tfrac{p+q-1}{2}, \tfrac{p+q-1}{-2})&\ & \mbox{when}\ p\ \mbox{is\ even}.
\end{eqnarray*}

\item[(ii)] 
If $|q| (= -q)=1$, then 
$B_r(p+q-1)$ is a rectangle 
with vertices 
\begin{eqnarray*}
\pm (\tfrac{p+2q+1}{2}, \tfrac{p+2q-1}{-2}), 
\pm (\tfrac{p+2q-1}{2}, \tfrac{p+2q+1}{-2})&\ & \mbox{when}\ p\ \mbox{is\ odd}, 
\\
\pm (\tfrac{p+2q}{2}, \tfrac{p+2q-2}{-2}), 
\pm (\tfrac{p+2q}{2}, \tfrac{p+2q+2}{-2})&\ & \mbox{when}\ p\ \mbox{is\ even}.
\end{eqnarray*}
\end{enumerate}

\item[(2)] 
Suppose that $r \in (-2,0)$. 
$B_r(-q)$ is a parallelogram with vertices 
\begin{eqnarray*}
\pm (\tfrac{q}{2q+2}, \tfrac{q}{2q+2}),\   \pm (\tfrac{q}{-2}, \tfrac{q}{2}) &\ &\mbox{when}\ p\ \mbox{is\ odd,\ and}
\end{eqnarray*}
$B_r(-q-1)$ is a parallelogram with vertices 
\begin{eqnarray*}
\pm (0,1),\   \pm (\tfrac{q+1}{-2}, \tfrac{q+1}{2}) &\ & \mbox{when}\ p\ \mbox{is\ even}.
\end{eqnarray*}

\item[(3)] 
Suppose that $r \in (0, \infty)$. 
\begin{enumerate}
\item[(i)] 
If $|q|(= q) \ne 1$, then 
$B_r(p+q-1)$ is a hexagon with vertices 
\begin{eqnarray*}
\pm (\tfrac{p+1}{2}, \tfrac{p-1}{-2}), 
\pm (\tfrac{p-1}{2}, \tfrac{p+1}{-2}), 
\pm (\tfrac{p+q-1}{2}, \tfrac{p+q-1}{-2})&\ & \mbox{when}\ p\ \mbox{is\ odd}, 
\\
\pm (\tfrac{p}{2}, \tfrac{p-2}{-2}), 
\pm(\tfrac{p}{2}, \tfrac{p+2}{-2}), 
\pm (\tfrac{p+q-1}{2}, \tfrac{p+q-1}{-2})&\ & \mbox{when}\ p\ \mbox{is\ even}.
\end{eqnarray*}

\item[(ii)] 
If  $|q| (= q)=1$, then 
$B_r(p+q-1)$ is a rectangle with vertices 
\begin{eqnarray*}
\pm (\tfrac{p+1}{2}, \tfrac{p-1}{-2}), 
\pm (\tfrac{p-1}{2}, \tfrac{p+1}{-2})&\ &\mbox{when}\ p\ \mbox{is\ odd}, 
\\
\pm (\tfrac{p}{2}, \tfrac{p-2}{-2}), 
\pm(\tfrac{p}{2}, \tfrac{p+2}{-2})&\ &\mbox{when}\ p\ \mbox{is\ even}.
\end{eqnarray*}
\end{enumerate}
\end{enumerate}
\end{lem}

\begin{proof}
Let us consider the classes in $\widehat{B}_{\beta,r}(1)(= B(1) \cap S_{\beta}(r))$, 
see Figure~\ref{fig_rokumentai}. 
Then $\|a\|=1$ and $\|\overline{a}\|_r= 1 - |\tfrac{y}{q}|$ 
for all $a = (x,y,z) \in \widehat{B}_{\beta,r}(1)$. 
To find the Thurston norm ball of $N(r)$, 
one needs to shear $\widehat{B}_{\beta,r}(1)$ 
by an appropriate amount depending on the $y$-coordinate of $a$. 
One can see that 
the shearing turns the parallelogram/hexagon 
into another parallelogram/hexagon 
unless $|q|$ equals $1$. 
The degeneration of the Thurston norm ball of $N(r)$ occurs 
when $|q|$ equals $1$. 
In this case, 
the shearing makes $2$ sides of the hexagon line up, and the hexagon turns into a rectangle. 
By using this argument, it is straightforward to verify the lemma. 
\end{proof}

Note that every top dimensional face on the boundary of the Thurston norm ball  of  $N(r)$  
is a fibered face for each $r \in {\mathcal Hyp}$. 
Figure~\ref{fig_ExBall_Nr} illustrates the Thurston norm balls of $N(\tfrac{5}{-2})$, $N(\tfrac{3}{-2})$, $N(\tfrac{2}{-3})$ and $N(1)$.

\begin{figure}
\begin{center}
\includegraphics[width=5.2in]{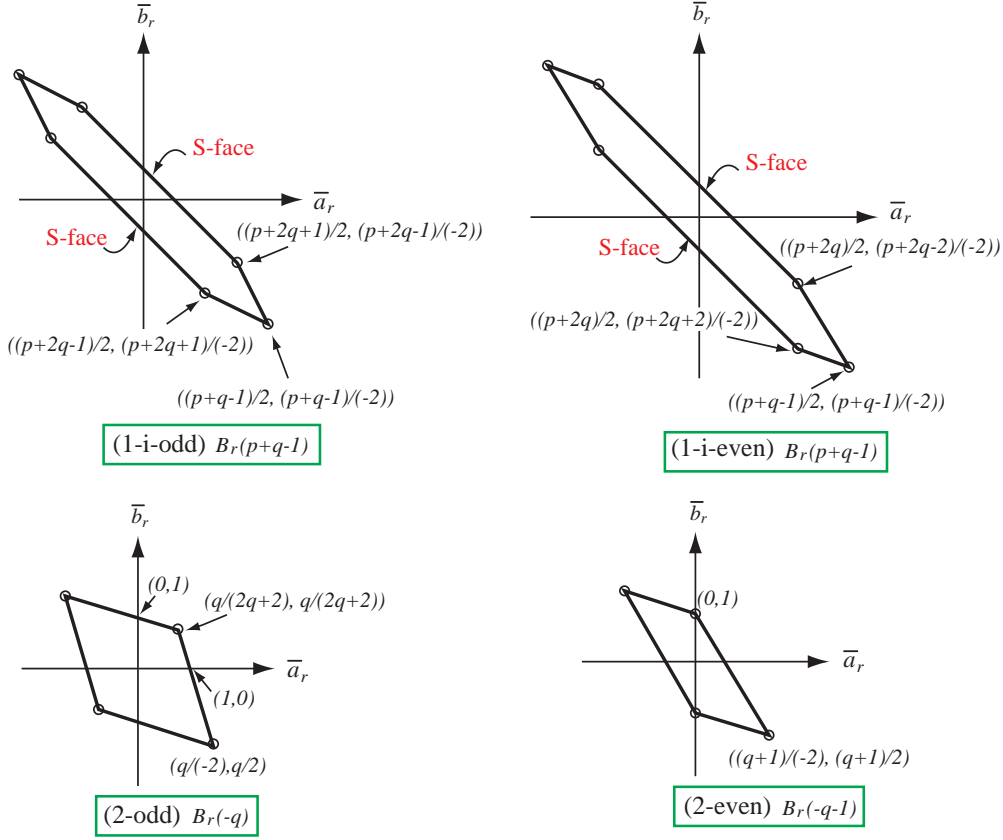}
\caption{Thurston norm ball $B_r(d)$ (with radius $d$) of $N(\tfrac{p}{q})$. 
(1-i-odd) $\tfrac{p}{q} \in (- \infty,-2) $, $q \ne 1$ and $p$ is odd. 
(1-i-even)  $\tfrac{p}{q} \in (- \infty,-2) $, $q \ne 1$ and $p$ is even. 
(2-odd) $\tfrac{p}{q} \in (-2,0)$ and $p$ is  odd. 
(2-even) $\tfrac{p}{q} \in (-2,0)$ and $p$ is even.}
\label{fig_Ball_Nr}
\end{center}
\end{figure}

 \begin{figure}[htbp]
\begin{center}
\includegraphics[width=5.5in]{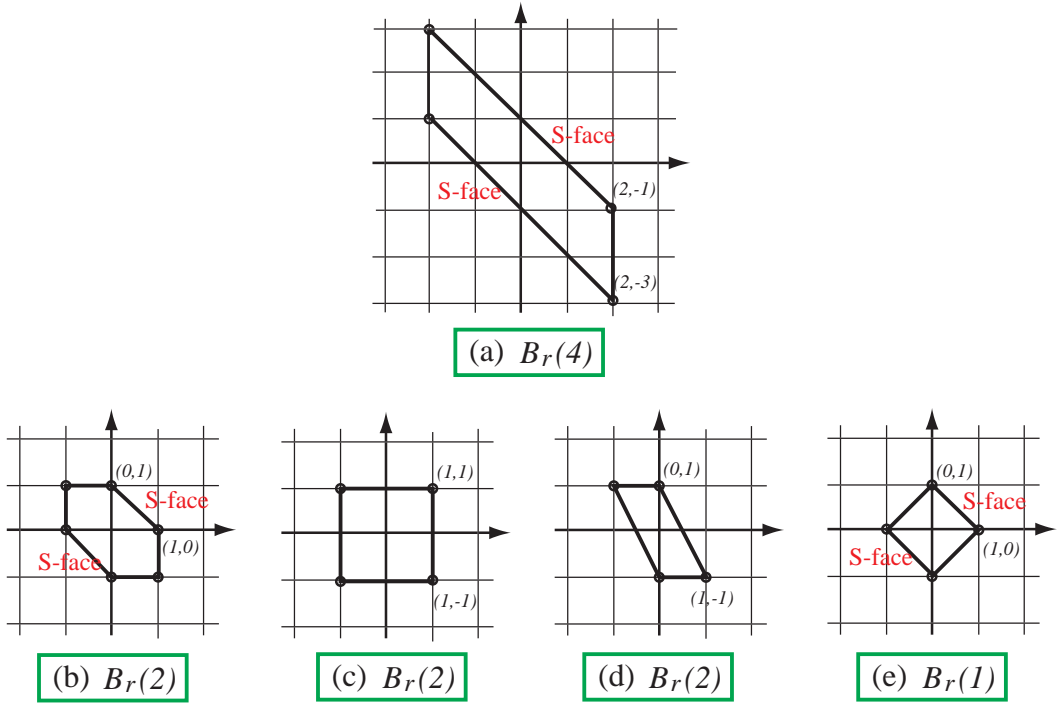}
\caption{Thurston norm ball $B_r(d)$ (with radius $d$)  of $N(r)$ when 
(a) $r= -6,4$, 
(b) $r=\tfrac{5}{-2}, \tfrac{1}{2}$, 
(c) $r = \tfrac{3}{-2},  \tfrac{1}{-2}$, 
(d) $r=\tfrac{4}{-3}, \tfrac{2}{-3}$, 
(e) $r=1$.} 
\label{fig_ExBall_Nr}
\end{center}
\end{figure}

We now prove  that 
there exist infinitely many  Thurston norm equivalent pairs obtained from $N$ by Dehn filling.

\begin{lem}
\label{lem_Teq-bunbo}
Let $p \in {\Bbb N}$ and $q \in {\Bbb Z}$ (resp. $p' \in {\Bbb N}$ and $q \in {\Bbb Z}$) be coprime such that 
$r = \tfrac{p}{q}, r'= \tfrac{p'}{q} \in {\mathcal Hyp} \cap (-2,0)$. 
Suppose that either both $p$ and $p'$ are odd or both $p$ and $p'$ are even. 
Then $N(r)  \underset{\mathrm{T}}{\sim} N(r')$. 
\end{lem}

\begin{proof} 
Suppose that $p$ and $p'$ are odd. 
The numerator does not appear in the vertices of   $B_{r}(-q)$, see  Lemma~\ref{lem_shape}(2).  
The position for the vertices of  $B_{r}(-q)$ is the same as that of $B_{r'}(-q)$. 
Thus the natural isomorphism 
$f: H_2(N(r), \partial N(r); {\Bbb Z}) \rightarrow H_2(N(r'), \partial N(r'); {\Bbb Z}) $ which sends 
$\overline{\mathfrak{a}_r}$ to $\overline{\mathfrak{a}_{r'}}$ and 
$\overline{\mathfrak{b}_r}$ to $\overline{\mathfrak{b}_{r'}}$ becomes a Thurston norm preserving isomorphism. 

The proof in the case $p$ and $p'$ are even is similar. 
\end{proof}

\begin{prop}
\label{prop_norm-equiv}
Suppose that both $r$, $-2-r \in {\mathcal Hyp}$. 
Then  $N(r)  \underset{\mathrm{T}}{\sim} N(-2-r) $. 
\end{prop}

\begin{proof} 
Let $p \in {\Bbb N}$ and $q \in {\Bbb Z}$ be coprime such that $r= \tfrac{p}{q} \in {\mathcal Hyp}$. 
We have shown the claim when $\tfrac{p}{q} \in (-2,-1)$, see  Lemma~\ref{lem_Teq-bunbo}. 
Now suppose that $\tfrac{p}{q} \in (-\infty,-2)$. 
Let us set an irreducible form $r'= \tfrac{p'}{q'}= \tfrac{p+2q}{-q}$ $(p' = p+2q \in {\Bbb N})$. 
By Lemma~\ref{lem_shape}(1)(3), $B_{r'}(p'+q'-1)$ and  $B_{r}(p+q-1)$ are  hexagons when $|q| \ne 1$ 
(resp. rectangle when $|q|=1$). 
The position for the vertices of $B_{r'}(p'+q'-1)$ is the same as that of  $B_{r}(p+q-1)$. 
The Thurston norm balls $B_{r'}(p'+q'-1)$ and  $B_{r}(p+q-1)$ have the same radius, 
i.e,  $p'+q'-1= p+q-1$. 
Thus the natural isomorphism 
$f: H_2(N(r), \partial N(r); {\Bbb Z}) \rightarrow H_2(N(r'), \partial N(r'); {\Bbb Z}) $ which sends 
$\overline{\mathfrak{a}_r}$ to $\overline{\mathfrak{a}_{r'}}$ and $\overline{\mathfrak{b}_r}$ to $\overline{\mathfrak{b}_{r'}}$ 
becomes  a Thurston norm preserving isomorphism. 
\end{proof}

\subsection{Entropy equivalence on fibered $3$-manifolds} 

\subsubsection{Definition of entropy  equivalence} 
\label{subsubsection_Def} 

Let  $(M, \Omega)$ and $(M', \Omega')$ be pairs of $3$-manifolds $M$, $M'$ and their fibered faces $\Omega$, $\Omega'$ respectively. 
Possibly $M \simeq M'$. 
Then $(M, \Omega)$ and $(M', \Omega')$ are {\it entropy equivalent},  
denoted by $(M, \Omega) \underset{\mathrm{ent}}{\sim} (M', \Omega')$, 
 if  there exists a Thurston norm preserving isomorphism 
$f: H_2(M, \partial M ; {\Bbb Z}) \rightarrow H_2(M', \partial M'; {\Bbb Z})$ satisfying the following. 
\begin{itemize}
\item 
$a \in int(C_{\Omega}({\Bbb Z}))$ if and only if $f(a) \in int(C_{\Omega'}({\Bbb Z}))$. 

\item 
$\mathrm{ent}(a) = \mathrm{ent}(f(a))$  for any $a \in int(C_{\Omega}({\Bbb Z}))$. 
\end{itemize}
The second bullet implies that 
$\mathrm{ent}(a) = \mathrm{ent}(f(a))$  for any $a \in int(C_{\Omega})$, 
since   $\mathrm{ent}: int(C_{\Omega}({\Bbb Q})) \rightarrow {\Bbb R}$  admits a unique continuous extension. 
Thus  if  $(M, \Omega) \underset{\mathrm{ent}}{\sim} (M', \Omega')$, then 
  $\min \mathrm{Ent}(M, \Omega) = \min \mathrm{Ent}(M', \Omega')$.

Here is an obvious example. 
If a face $\Omega'$ of $M$ is opposite to a fibered face $\Omega$, 
then $\Omega'$ is also a fibered face. 
The pairs $(M, \Omega)$ and $(M, \Omega')$ are entropy equivalent, 
because the isomorphism on $H_2(M, \partial M ; {\Bbb Z})$ 
given by $a \mapsto -a$ preserves the Thurston norm and entropy.

 Fibered $3$-manifolds $M$ and $M'$ are  {\it entropy equivalent}, denoted by $M \underset{\mathrm{ent}}{\sim} M'$, 
if  there exists a Thurston norm preserving isomorphism 
$f: H_2(M, \partial M ; {\Bbb Z}) \rightarrow H_2(M', \partial M'; {\Bbb Z})$ satisfying the following. 
\begin{itemize}
\item 
$a \in H_2(M, \partial M ; {\Bbb Z})$ is a fibered class  if and only if 
$f(a) \in H_2(M', \partial M' ; {\Bbb Z})$ is a fibered class. 

\item 
Given a fibered face $\Omega$ of $M$, 
we have  $\mathrm{ent}(a) = \mathrm{ent}(f(a))$  for any $a \in int(C_{\Omega}({\Bbb Z}))$. 
\end{itemize} 
If $M \underset{\mathrm{ent}}{\sim} M'$, 
then  $\min \mathrm{Ent}(M) = \min \mathrm{Ent}(M')$.

\subsubsection{Entropy equivalence on manifolds $N(r)$}
\label{subsubsection_EntropyEqu}

In this subsection, first of all we focus on the set $S_{\beta}(r)$ and   the stable foliation $\mathcal{F}_a$ for  $a \in S_{\beta}(r)$. 
We compute the number of prongs on each boundary component of $F_a$ lying on  $T_{\beta}$. 
We will see this number depends on the slope $r$ and the fibered face $\Omega$ of $N(r)$  
with the property  $\overline{a} \in int(C_{\Omega})$. 
Then we discuss the entropy equivalence between  $N(r)$ and $N(-2-r)$ 
when $r, -2-r \in {\mathcal Hyp}$.

We begin with the definition of  {\it $A$-faces} and {\it $S$-faces}. 
They  are  top dimensional faces $\Omega$ on the boundary of  the Thurston norm ball of $N(r)$ for $r = \tfrac{p}{q} \in {\mathcal Hyp}$. 
\begin{itemize}
\item 
Suppose that $|q| \ne 1$.  
Then $\Omega$   is called an {\it $A$-face}  
if an element of $ \partial \Omega$ is equal to $\overline{\alpha- \gamma}$ projectively. 
Equivalently, $\Omega$ is an $A$-face 
if  an element  of $\partial \Omega$ is equal to $\overline{\mathfrak{a}_r} -  \overline{\mathfrak{b}_r}$ projectively.  
A face  $\Omega$ is called an {\it $S$-face} if it is not  an $A$-face.

\item 
Suppose that $|q|=1$. 
Then  $\Omega$  is called an  {\it $A$-face} 
if the interior of the cone over $\Omega$ contains   $\overline{\alpha- \gamma}$ projectively. 
Equivalently, $\Omega$ is an $A$-face 
if  the interior of the cone over $\Omega$ contains $\overline{\mathfrak{a}_r} -  \overline{\mathfrak{b}_r}$ projectively. 
A  face  $\Omega$ is called an {\it $S$-face} if it is not  an $A$-face.  
\end{itemize}

It follows from Lemma~\ref{lem_shape} that 
every top dimensional face for $N(r)$ is an $A$-face if  $r \in (-2,0)$. 
When $r \in (- \infty, -2) \cup (0, \infty)$ such that $|q| \ne 1$ (resp.  $|q|=1$), 
the Thurston norm ball for $N(r)$ is a hexagon (resp. rectangle) having two $S$-faces and four $A$-faces 
(resp. having two $S$-faces and two $A$-faces), see Figures~\ref{fig_Ball_Nr} and \ref{fig_ExBall_Nr}. 

It is worthwhile to point out that 
the two $S$-faces come from the fibered face $\Delta_1$ and its opposite face $\Delta_1'$  for $N$. 
(See Figure~\ref{fig_rokumentai}(a)(d).)
Let us turn to the $A$-faces. 
If $|q| \ne 1$, then the Thurston norm ball of $N(r)$ has four $A$-faces, and 
they come from the four fibered faces $\Delta$, $\Delta_2$ and their opposite faces $\Delta'$, $\Delta_2'$. 
(See Figure~\ref{fig_rokumentai}(a)(b)(c)(d).) 
The degeneration of $A$-faces occur when $|q|=1$. 
In this case, the Thurston norm ball of $N(r)$ has two $A$-faces. 
One of the $A$-faces comes from the pair $\Delta$ and $\Delta_2'$. 
The other $A$-face comes from the pair $\Delta'$ and $\Delta_2$. 
This observation leads to the following.

\begin{lem}
\label{lem_Cannonical}
Let $\sigma: Int\ C \rightarrow int(C_{\Delta})$ be the map 
given in Section~\ref{subsection_Fibered}. 
We take a class $a \in S_{\beta}(r) \cap Int\ C$. 
\begin{itemize}
\item[(1)] 
Suppose that $a \in S_{\beta}(r) \cap (int(\Delta) \cup int(\Delta'))$. 
Then $\overline{a} \in int(C_{\Omega_A}) \subset H_2(N_{\beta}(r), \partial N_{\beta}(r))$ 
for some $A$-face, and 
$\sigma(a) \in int(C_{\Delta}) \cap S_{\beta}(r)$.

\item[(2)]
Suppose that $a \in S_{\beta}(r) \cap (int(\Delta_1) \cup int(\Delta_1'))$. 
Then $\overline{a} \in int(C_{\Omega_S}) \subset H_2(N_{\beta}(r), \partial N_{\beta}(r))$ 
for some $S$-face, and 
$\sigma(a) \in int(C_{\Delta}) \cap S_{\gamma}(r)$.

\item[(3)]
Suppose that $a \in S_{\beta}(r) \cap (int(\Delta_2) \cup int(\Delta_2'))$. 
Then $\overline{a} \in int(C_{\Omega_A}) \subset H_2(N_{\beta}(r), \partial N_{\beta}(r))$ 
for some $A$-face, and 
$\sigma(a) \in int(C_{\Delta}) \cap S_{\alpha}(r)$. 

\end{itemize}
\end{lem}

\begin{lem}
\label{lem_EeqSA}
Let $r \in {\mathcal Hyp}$. 
Any two $S$-faces of $N(r)$ are entropy equivalent, and 
any two $A$-faces of $N(r)$ are entropy equivalent. 
\end{lem}

\begin{proof}
An $S$-face of $N(r)$ is opposite to the other $S$-face, and hence they are entropy equivalent. 
(See the example after the definition of entropy equivalence.) 
Similarly, if an $A$-face $\Omega$ is opposite side to an $A$-face $\Omega'$, 
then they are entropy equivalent. 
Thus the proof in the case $r \in {\Bbb Z}$ is done.

We assume that $r = \tfrac{p}{q} \notin {\Bbb Z}$, i.e, $q \ne 1$. 
We need to show that an $A$-face $\Omega$ is entropy equivalent to an $A$-face $\widehat{\Omega}$ 
which is not the opposite face $\Omega'$. 
To do this, it is enough to prove that 
the $A$-face of $N(r)$ coming from $\Delta$, say $\Omega_{A, \Delta}$, 
and the $A$-face of $N(r)$ coming from $\Delta_2$, say $\Omega_{A, \Delta_2}$ 
are entropy equivalent. 
We first find the Thurston norm preserving isomorphism 
$$f: H_2(N(r), \partial N(r) ;{\Bbb Z}) \rightarrow  H_2(N(r), \partial N(r); {\Bbb Z})$$
which sends $int(C_{\Omega_{A, \Delta_2}}(\Bbb Z))$ to $int(C_{\Omega_{A, \Delta}}(\Bbb Z))$. 
We recall the two isomorphisms: 
\begin{eqnarray*}
(h^2)_*:H_2(N, \partial N ;{\Bbb Z}) &\rightarrow& H_2(N, \partial N ;{\Bbb Z})
\\
(x,y,z) &\mapsto& (y,z,x)
\end{eqnarray*}
and 
\begin{eqnarray*}
- (i_{\gamma})_*:H_2(N, \partial N ;{\Bbb Z}) &\rightarrow& H_2(N, \partial N ;{\Bbb Z})
\\
(x,y,z) &\mapsto& (y,x,z), 
\end{eqnarray*}
see the proof of Lemma~\ref{lem_conjugateXY}. 
Observe that $(h^2)_*(S_{\beta}(r))= S_{\alpha}(r)$. 
This shows that we have the isomorphism 
\begin{eqnarray*}
\overline{(h^2)_*}: H_2(N_{\beta}(r), \partial N_{\beta}(r) ;{\Bbb Z}) &\rightarrow&  H_2(N_{\alpha}(r), \partial N_{\alpha}(r) ;{\Bbb Z})
\\
\overline{(x,y,z)} &\mapsto& \overline{(y,z,x)}
\end{eqnarray*}
induced from $(h^2)_*$. 
On the other hand, we have  $-(i_{\gamma})_* (S_{\alpha}(r))= S_{\beta}(r)$. 
Thus $-(i_{\gamma})_*$ induces the isomorphism 
\begin{eqnarray*}
\overline{-(i_{\gamma})_*}: H_2(N_{\alpha}(r), \partial N_{\alpha}(r) ;{\Bbb Z}) &\rightarrow&  H_2(N_{\beta}(r), \partial N_{\beta}(r) ;{\Bbb Z})
\\
\overline{(x,y,z)} &\mapsto& \overline{(y,x,z)}.
\end{eqnarray*}
Let us set 
$f= \overline{-(i_{\gamma})_*} \circ \overline{(h^2)_*}$. 
One sees that $f$ sends $int(C_{\Omega_{A, \Delta_2}}(\Bbb Z))$ to $int(C_{\Omega_{A, \Delta}}(\Bbb Z))$, because 
\begin{eqnarray*}
(h^2)_*(S_{\beta}(r) \cap int(C_{\Delta_2}({\Bbb Z})))&=& S_{\alpha}(r) \cap int(C_{\Delta}({\Bbb Z}))\ \mbox{and}
\\
-(i_{\gamma})_*(S_{\alpha}(r) \cap int(C_{\Delta}({\Bbb Z}))) &=& S_{\beta}(r) \cap int(C_{\Delta}({\Bbb Z})). 
\end{eqnarray*}
Then $f$ preserves the Thurston norm, 
since both $ \overline{(h^2)_*}$ and $\overline{-(i_{\gamma})_*} $ preserve the Thurston norm 
by Lemma~\ref{lem_NormChange2}. 

We now prove that $f$ preserves the entropies on $int(C_{\Omega_{A, \Delta_2}}({\Bbb Z}))$. 
Let $(x,y,z) \in S_{\beta}(r) \cap int(C_{\Delta_2}({\Bbb Z}))$. 
Then $\overline{(x,y,z)}$ and $\overline{(h^2)_*}(\overline{(x,y,z)}) = \overline{(y,z,x)}$ have the same entropy, 
since $\overline{(h^2)_*}$ is induced from the homeomorphism 
$h^2: (S^3, \mathcal{C}_3) \rightarrow (S^3, \mathcal{C}_3)$. 
Next, let us take $(y,z,x) \in S_{\alpha}(r) \cap int(C_{\Delta}({\Bbb Z}))$. 
As a consequence of Lemma~\ref{lem_conjugateXY}, 
$\overline{(y,z,x)}$ and $-(i_{\gamma})_*(\overline{(y,z,x)}) = \overline{(z,y,x)}$ have the same entropy. 
In fact,  the inverse $(\Phi_{\overline{(y,z,x)}})^{-1}$ of the monodromy $\Phi_{\overline{(y,z,x)}}$ 
of the fibration on $N_{\alpha}(r)$ associated to $\overline{(y,z,x)}$ 
is conjugate to the monodromy $\Phi_{\overline{(z,y,x)}}$ of the fibration on $N_{\beta}(r)$ 
associated to $\overline{(z,y,x)}$. 
Putting all things together, we see that for $(x,y,z) \in S_{\beta}(r) \cap int(C_{\Delta_2}({\Bbb Z}))$, 
the two fibered classes 
$\overline{(x,y,z)} \in int(C_{\Omega_{A, \Delta_2}}({\Bbb Z}))$ and $\overline{(z,y,x)} \in int(C_{\Omega_{A, \Delta}}({\Bbb Z}))$ 
have the same entropy. 
This completes the proof. 
\end{proof}

Because of the lemma above, 
we denote  by $\Omega_A = \Omega_{A,r}$ (resp. $\Omega_S = \Omega_{S,r}$), 
any $A$-face (resp. $S$-face) of $N(r)$.  
The first letter `$A$'  (resp. `$S$') represents `asymmetry' (resp. `symmetry'). 
cf. Remark~\ref{rem_symmetry}.

\begin{lem}
\label{lem_face-prong} 
Let $a \in S_{\beta}(\tfrac{p}{q}) \subset H_2(N, \partial N; {\Bbb Z})$ be a primitive fibered class, and 
let $\Omega$ be the fibered face of $N(\tfrac{p}{q})$ such that 
$\overline{a} \in int(C_{\Omega}) \subset H_2(N_{\beta}(r), \partial N_{\beta}(r))$. 
If  $\Omega$ is an $S$-face (resp. $A$-face), 
then  $\mathcal{F}_{a}$  has the property that 
each boundary component on $T_{\beta}$ has $p+2q$ prongs (resp. $|q|$ prongs).  
%For the dilatation $\lambda_{p/q}(\overline{a})$ of $\overline{a}$ and  the dilatation $\lambda(a)$,  
The inequality $\lambda_{p/q}(\overline{a}) \le \lambda(a)$ holds, and 
the equality is achieved   if $p+2q \ne 1$ (resp. if $|q| \ne 1$). 
\end{lem}

\noindent
Note that $p+2q \ge 1$ when $\Omega$ is an $S$-face (cf. Remark~\ref{rem_slope-gam}).

\begin{proof} 
Let $\mathcal{S}= \mathcal{S}_a$ be the suspended stable foliation constructed from 
$\mathcal{F}_a \times I \subset F_a \times I$ by gluing $\mathcal{F}_a \times \{1\}$ to $\mathcal{F}_a \times \{0\}$ 
using $\Phi_a$. 
It is known that such a foliation $\mathcal{S}$  depends only on the fibered face, 
that is $\mathcal{S}_a$ is isotopic to $\mathcal{S}_{a'}$ 
if  $a$ and $a'$ are  primitive fibered classes in  the cone over the same fibered face \cite[Corollary~3.2]{McMullen}. 
When $a \in S_{\beta}(r)$, 
the number of prongs on each boundary component of $T_{\beta}$ 
is determined by how  $\mathcal{S}$ intersects with the simple closed curve representing the slope $r$ on $T_{\beta}$. 
Thus such a number  depends only on the slope $r$ and the fibered face. 

Given a fibered face $\Omega$ of $N(\tfrac{p}{q})$, 
it is enough to take one primitive fibered class $a \in S_{\beta}(\tfrac{p}{q})$  such that $\overline{a} \in int(C_{\Omega})$. 
Then one can apply Lemma~\ref{lem_sing-data} to know  the desired number of prongs on each boundary component. 
Let us compute the desired number when $r= \tfrac{p}{q} \in (-2,0)$. 
(The rest of the cases can be proved similarly.) 
In this case, every face of $N(r)$ is an $A$-face. 
One sees that $\mathfrak{a}_r \in int(C_{\Delta}) \cap S_{\beta}(r)$ and 
$\overline{\mathfrak{a}_r} \in int(C_{\Omega_A})$. 
By Lemma~\ref{lem_sing-data}, the desired number equals $|q|$. 

The second half of the claim on the inequality between $\lambda_{p/q}(\overline{a})$ and $\lambda(a)$ is clear. 
The equality holds if $\mathcal{F}_a$ has the property such that 
 any boundary component on $T_{\beta}$ has no $1$ prong. 
\end{proof}

\begin{prop}
\label{prop_goodFiberFace}
For $r \in {\mathcal Hyp}$, there exists a fibered face $\Omega$ of $N(r)$ which enjoys  $(*)$ in Theorem~\ref{thm_key}.  
\end{prop}

\begin{proof} 
We use Lemma~\ref{lem_face-prong}. 
Let $p \in {\Bbb N}$ and $q \in {\Bbb Z}$ be coprime such that  $\tfrac{p}{q} \in {\mathcal Hyp}$. 
There exists no pair $(p,q)$ such that $|q|=1$ and $p+2q=1$.

Suppose that  $|q| \ne 1$ and $p+2q \ne 1$. 
Then each fibered face of $N(r)$ 
enjoys $(*)$ in Theorem~\ref{thm_key}. 
Suppose that $|q| \ne 1$ and $p+2q =1$ (resp.  Suppose that $|q| = 1$ and $p+2q \ne 1$). 
Then only $A$-faces (resp. only $S$-faces) of $N(r)$ fulfill $(*)$. 
\end{proof}

\begin{lem} 
\label{lem_Ent_r}
For $r= \tfrac{p}{q} \in {\mathcal Hyp}$, 
let $\Omega$ be a fibered face of $N(r)$  enjoying $(*)$ in Theorem~\ref{thm_key}. 
We take $a \in S_{\beta}(r) \cap Int\ C$ such that $\|a\|=1$. 
\begin{enumerate}
\item[(1)]
If $\Omega= \Omega_S$ and 
$\overline{a} \in int(C_{\Omega_S}) \subset H_2(N_{\beta}(r), \partial N_{\beta}(r))$, then 
$$\mathrm{Ent}_r(\overline{a}) =   (1 - \tfrac{1}{p+q}) \log \lambda(\sigma(a)).$$

\item[(2)]
If $\Omega= \Omega_A$ and 
$\overline{a} \in int(C_{\Omega_A}) \subset H_2(N_{\beta}(r), \partial N_{\beta}(r))$, then 
\begin{eqnarray*}
 \mathrm{Ent}_r(\overline{a}) &=&   (1 - |\tfrac{y}{q}|) \log \lambda(\sigma(a))\  \hspace{2mm}
\mbox{when}\   \sigma(a) = (x,y,z) \in int(\Delta) \cap S_{\beta}(r), \ \mbox{and} 
\\
 \mathrm{Ent}_r(\overline{a}) &=&   (1 - |\tfrac{x}{q}|) \log \lambda(\sigma(a))\  \hspace{2mm}
\mbox{when}\   \sigma(a) = (x,y,z) \in int(\Delta) \cap S_{\alpha}(r). 
\end{eqnarray*}
\end{enumerate}
\end{lem}

\begin{proof} 
(1) 
We have $\sigma(a) \in int(\Delta) \cap S_{\gamma}(r)$ 
since $\overline{a} \in int(C_{\Omega_S})$, 
see Lemma~\ref{lem_Cannonical}. 
The Thurston norms of both classes 
$\overline{a} \in H_2(N_{\beta}(r), \partial N_{\beta}(r))$ and 
$\overline{\sigma(a)} \in H_2(N_{\gamma}(r), \partial N_{\gamma}(r))$ are equal, and 
hence 
$\|\overline{a}\|_r = \|\overline{\sigma(a)}\|_r = 1 - \tfrac{1}{p+q}$ 
by Lemma~\ref{lem_NormChange}. 
On the other hand, 
the condition $(*)$ in Theorem~\ref{thm_key} ensures that 
$\lambda_r(\overline{a})$ is equal to $\lambda(a)$. 
We have the equality $\lambda(a)= \lambda(\sigma(a))$, since the monodromies $\Phi_a$ and $\Phi_{\sigma(a)}$ are conjugate. 
Thus 
$$\mathrm{Ent}_r(\overline{a}) =  \|\overline{a}\|_r \log \lambda_r(\overline{a}) = (1 - \tfrac{1}{p+q}) \log \lambda(\sigma(a)) .$$

\noindent
(2) 
By using Lemma~\ref{lem_NormChange2}, one can prove the claim similarly. 
\end{proof}

\begin{thm}
\label{thm_entropy_equiv}
Let $p \in {\Bbb N}$ and $q \in {\Bbb Z}$ be coprime such that $\tfrac{p}{q} \in {\mathcal Hyp}$.  
\begin{enumerate}
\item[(1)]  
Suppose that  $\tfrac{p}{q} \in (- \infty, -2)$ and $p+2q \ne 1$. 
Then $(N(\tfrac{p}{q}), \Omega_S) \underset{\mathrm{ent}}{\sim} (N(\tfrac{2q + p}{-q}), \Omega_S)$.

\item[(2)] 
Suppose that   $\tfrac{p}{q} \in (- \infty, -1)$ and $|q| \ne 1$. 
Then $(N(\tfrac{p}{q}), \Omega_A) \underset{\mathrm{ent}}{\sim} (N(\tfrac{-2q - p}{q}), \Omega_A)$. 

\item[(3)] 
Suppose that   $\tfrac{p}{q} \in (- \infty, -1)$,  $p+2q \ne 1$ and $|q| \ne 1$. 
Then $N(\tfrac{p}{q}) \underset{\mathrm{ent}}{\sim} N(\tfrac{-2q - p}{q})$. 
\end{enumerate}
\end{thm}

\begin{proof} 
For $r = \tfrac{p}{q} \in {\mathcal Hyp} \cap (- \infty, -1)$, 
set $r' = -2-r$. 
Recall that 
$$f: H_2(N(r), \partial N(r); {\Bbb Z}) \rightarrow  H_2(N(r'), \partial N(r'); {\Bbb Z}) $$ 
is the Thurston norm preserving isomorphism 
as in the proof of Lemma~\ref{lem_Teq-bunbo} and Proposition~\ref{prop_norm-equiv}. 
Then $f$ maps $A$-faces (resp. $S$-faces) of $N(r)$ to $A$-faces (resp. $S$-faces) of $N(r')$. 

Let $b_0, \widetilde{b_0}, \cdots, b_3, \widetilde{b_3} \in int(\Delta)$ be as in Remark~\ref{rem_symmetry-ent}. 
\medskip
\\
(1) 
Let $\Omega_{S,r}$ (resp. $\Omega_{S, r'}$) be the $S$-face of $N(r)$ (resp. $N(r')$) 
coming from $\Delta_1$ of $N$. 
Observe that 
$$f(int(C_{\Omega_{S,r}}({\Bbb Z}))) = int(C_{\Omega_{S,r'}}({\Bbb Z})).$$
It suffices to prove that for each $a \in int(\Delta_1) \cap S_{\beta}(r)$, 
the two classes 
\begin{eqnarray*}
\overline{a} &\in&  int(C_{\Omega_{S, r}}) \subset H_2(N_{\beta}(r), \partial N_{\beta}(r)) \ \mbox{and}
\\
f(\overline{a}) &\in& int(C_{\Omega_{S, r'}}) \subset H_2(N_{\beta}(r'), \partial N_{\beta}(r'))
\end{eqnarray*}
%$\overline{a} \in  int(C_{\Omega_{S, r}}) \subset H_2(N_{\beta}(r), \partial N_{\beta}(r))$ 
%and $$ 
have the same entropy. 
To do this, 
we consider the sets $int(\Delta) \cap S_{\gamma}(r)$ and $int(\Delta) \cap S_{\gamma}(r')$ 
which are the images of $int(\Delta_1) \cap S_{\beta}(r)$ and $int(\Delta_1) \cap S_{\beta}(r')$ under $\sigma$.  
If we write $b_0= \sigma(a) \in int(\Delta) \cap S_{\gamma}(r)$, then 
$\widetilde{b_2} = \sigma(a') \in int(\Delta) \cap S_{\gamma}(r')$, 
where $a'= \iota_{\beta}(f(\overline{a})) \in int(\Delta_1) \cap S_{\beta}(r')$. 
(See Figure~\ref{fig_DeltaCapS}(right).) 
As a consequence of  Corollary~\ref{cor_symmetry}  and Lemma \ref{lem_face-prong}, it follows that 
$$\mathrm{ent}_r( \overline{b_0})= \mathrm{ent}_{r'} (\overline{\widetilde{b_2}})$$
i.e, $\overline{b_0} \in H_2(N_{\gamma}(r), \partial N_{\gamma}(r))$ and 
$\overline{\widetilde{b_2}} \in H_2(N_{\gamma}(r'), \partial N_{\gamma}(r'))$ 
have the same entropy. 
Since $\mathrm{ent}_r( \overline{b_0})= \mathrm{ent}_r( \overline{a})$ 
and $\mathrm{ent}_{r'} (\overline{\widetilde{b_2}})=\mathrm{ent}_{r'} (f(\overline{a}))$, 
we conclude that $\mathrm{ent}_r( \overline{a})= \mathrm{ent}_{r'}( f(\overline{a}) )$. 
This completes the proof. 
\medskip
\\
(2) 
Let $\Omega_{A, r}$ be the $A$-face of $N(r)$ coming from $\Delta$ of $N$, and 
let $\Omega_{A,r'}$ be the $A$-face of $N(r)$ coming from $\Delta_2'$ of $N$. 
One sees that 
 $$f(int(C_{\Omega_{A,r}}({\Bbb Z}))) = int(C_{\Omega_{A,r'}}({\Bbb Z})).$$
It is enough to prove that for each $a \in int(\Delta) \cap S_{\beta}(r)$, 
the two classes 
\begin{eqnarray*}
\overline{a} &\in&  int(C_{\Omega_{A, r}}) \subset H_2(N_{\beta}(r), \partial N_{\beta}(r)) \ \mbox{and}
\\
f(\overline{a}) &\in& int(C_{\Omega_{A, r'}}) \subset H_2(N_{\beta}(r'), \partial N_{\beta}(r'))
\end{eqnarray*}
%$\overline{a} \in H_2(N_{\beta}(r), \partial N_{\beta}(r))$ and $f(\overline{a}) \in H_2(N_{\beta}(r'), \partial N_{\beta}(r'))$ 
have the same entropy. 
Now, we consider the sets $int(\Delta) \cap S_{\beta}(r)$ and $int(\Delta) \cap S_{\alpha}(r')$ 
which are the images of $int(\Delta) \cap S_{\beta}(r)$ and $int(\Delta_2') \cap S_{\beta}(r')$ under $\sigma$.  
If one writes $b_0 = \sigma(a) \in int(\Delta) \cap S_{\beta}(r)$, then 
one can write 
$\widetilde{b_0} = \sigma(a') \in int(\Delta) \cap S_{\alpha}(r')$, 
where $a'= \iota_{\beta}(f(\overline{a})) \in int(\Delta'_2) \cap S_{\beta}(r')$. 
As a consequence of Corollary~\ref{cor_symmetry}  and Lemma \ref{lem_face-prong}, it follows that 
$$\mathrm{ent}_r( \overline{b_0})= \mathrm{ent}_{r'} (\overline{\widetilde{b_0}})$$  
i.e, 
$\overline{b_0} \in H_2(N_{\beta}(r), \partial N_{\beta}(r))$ and 
$\overline{\widetilde{b_0}} \in H_2(N_{\alpha}(r'), \partial N_{\alpha}(r'))$ 
have the same entropy. 
Since $\mathrm{ent}_r( \overline{b_0})= \mathrm{ent}_r( \overline{a})$ 
and $\mathrm{ent}_{r'} (\overline{\widetilde{b_0}})=\mathrm{ent}_{r'} (f(\overline{a}))$, 
the map  $f$ preserves the entropy, i.e, 
$\mathrm{ent}_r( \overline{a})= \mathrm{ent}_{r'}( f(\overline{a}) )$. 
This completes the proof. 
\medskip
\\
(3) The proof of (3) is similar to that of (1) or (2).  
\end{proof}

Let us check the entropy equivalence on some pairs 
which we promised to prove in Section~\ref{subsection_equivalence}. 
Theorem~\ref{thm_entropy_equiv} tells us that 
$(N(-4), \Omega_S)   \underset{\mathrm{ent}}{\sim}  (N(2), \Omega_S)$ and 
$N(\tfrac{3}{-2}) \underset{\mathrm{ent}}{\sim}  N(\tfrac{1}{-2})$.  
Since $N(-4) \simeq N(\tfrac{3}{-2})$, 
we see that 
$$(N(2), \Omega_S)  \underset{\mathrm{ent}}{\sim}  (N(\tfrac{3}{-2}), \Omega_A)  \underset{\mathrm{ent}}{\sim}  (N(\tfrac{1}{-2}), \Omega_A)  .$$

\section{Proofs of main results}
\label{section_Proof}

%{\bf ここの書き方注意! }
%By the symmetries of $H_2(N, \partial N)$, we may only consider the primitive fibered classes $a \in int(C_{\Delta})$ 
%such that $\phi_a \in \mathcal{M}$ 
%for the proofs of Theorems~\ref{thm_main} (Sections~\ref{subsection_Proof_thm_main1} and \ref{subsection_Proof_thm_main2})
% and \ref{thm_smallest_dil_ori} (Section~\ref{subsection_orientable}). 

\subsection{Proof of Theorem~\ref{thm_key}(Theorem~\ref{thm_key-SA})}
\label{subsection_theorem_key}

In this subsection, we shall prove the next theorem which is equivalent to Theorem~\ref{thm_key} 
(see Lemma~\ref{lem_face-prong} or proof of Proposition~\ref{prop_goodFiberFace}). 

\begin{thm}
\label{thm_key-SA}
Let $p \in {\Bbb N}$ and $q \in {\Bbb Z}$ be coprime such that $\tfrac{p}{q} \in {\mathcal Hyp}$.  
\begin{enumerate}
\item[(1)] 
Suppose that $\tfrac{p}{q} \in (- \infty, -2) \cup (0, \infty)$ and $p+ 2q \ne 1$.  
Then 
\begin{enumerate}
\item[(i)] 
$\min \mathrm{Ent}(N(1), \Omega_S) = 2 \log \delta(D_4) \approx 1.6628$, 

\item[(ii)] 
$\min \mathrm{Ent}(N(\tfrac{p}{q}), \Omega_S) = 2 \log( \tfrac{3+ \sqrt{5}}{2}) \approx 1.9248$ for  $\tfrac{p}{q} = -4, 2$, and 

\item[(iii)] 
$\min \mathrm{Ent}(N(\tfrac{p}{q}), \Omega_S) > 1.97475$ if $\tfrac{p}{q} \ne -4, 1,2$. 
\end{enumerate}

\item[(2)] 
Suppose that $|q| \ne 1$. 
Then 
\begin{enumerate}
\item[(i)] 
$\min \mathrm{Ent}(N(\tfrac{p}{q}), \Omega_A) =  2 \log( \tfrac{3+ \sqrt{5}}{2}) \approx 1.9248$ for $\tfrac{p}{q}= \tfrac{3}{-2}, \tfrac{1}{-2}$, and 

\item[(ii)] 
$\min \mathrm{Ent}(N(\tfrac{p}{q}), \Omega_A) >  1.97475$ if  $\tfrac{p}{q} \ne  \tfrac{3}{-2}, \tfrac{1}{-2}$. 
\end{enumerate}
\end{enumerate}
\end{thm}

We start by the computation of  $\min \mathrm{Ent}(N(\tfrac{p}{q} ), \Omega_S) $ and 
$\min \mathrm{Ent}(N(\tfrac{p}{q} ), \Omega_A) $.

\begin{lem}
\label{lem_MNE-S}
Let $\tfrac{p}{q} \in (- \infty, -2) \cup  (0, \infty)$. 
Then 
$$\min \mathrm{Ent}(N(\tfrac{p}{q} ), \Omega_S) =  (1- \tfrac{1}{p+q}) \log \lambda_{[\tfrac{p}{2p+2q}, \tfrac{p}{2p+2q}]} \hspace{2mm}
\mbox{if}\ p+2q \ne 1.$$
\end{lem}

\begin{proof} 
By Lemma~\ref{lem_EeqSA}, 
we have the equalities 
\begin{eqnarray*}
\min \mathrm{Ent}(N(\tfrac{p}{q} ), \Omega_S) 
&=& 
\min\{\|\overline{[x,y]}\|_{p/q}  \log \lambda_{p/q} (\overline{[x, y]})\ |\ 
[x,y] \in int(\Delta) \cap S_{\gamma}(\tfrac{p}{q}) \}
\\
&=& \min \{ (1- \tfrac{1}{p+q}) \log \lambda_{p/q} (\overline{[x, y]})\ |\ 
[x,y] \in int(\Delta) \cap S_{\gamma}(\tfrac{p}{q}) \}
\\
&=& 
\min \{ (1- \tfrac{1}{p+q}) \log \lambda_{[x, y]}\ |\ 
[x,y] \in int(\Delta) \cap S_{\gamma}(\tfrac{p}{q}) \}. 
\end{eqnarray*}
The first equality comes from Lemma~\ref{lem_Cannonical}(2).  
The second equality and the third one  follow from Lemma~\ref{lem_NormChange} and Lemma~\ref{lem_face-prong} respectively.  
Lemmas~\ref{lem_concave_sym}(3) and \ref{lem_boundary}(3) imply that 
the minimum is achieved by the center  $[\tfrac{p}{2p+2q}, \tfrac{p}{2p+2q}] \in int(\Delta) \cap S_{\gamma}(\tfrac{p}{q})$. 
This completes the proof.
\end{proof} 

\begin{lem}
\label{lem_MNE-A}
Let $\tfrac{p}{q} \in (- \infty, \infty)$.  
Then 
\begin{eqnarray*}
\min \mathrm{Ent}(N(\tfrac{p}{q}), \Omega_A)  
&=& \min \{(1 - |\tfrac{y}{q}| ) \log \lambda_{[x,y]}\ |\ [x,y] \in int (\Delta) \cap S_{\beta} (\tfrac{p}{q})\}
\\
&=& \min \{(1 - |\tfrac{x}{q}| ) \log \lambda_{[x,y]}\ |\ [x,y] \in int (\Delta) \cap S_{\alpha} (\tfrac{p}{q})\}.
\end{eqnarray*}
\end{lem}

\begin{proof}
The claim follows from Lemmas~\ref{lem_Cannonical}(1)(3) and \ref{lem_EeqSA}. 
\end{proof}

\begin{rem}
\label{rem_symmetry}
If an $S$-face  enjoys $(*)$ in Theorem~\ref{thm_key}, 
one can compute $\min \mathrm{Ent}(N(r), \Omega_S) $ explicitly from Lemma~\ref{lem_MNE-S}. 
This is because 
$\mathrm{ent}: int(\Delta) \rightarrow {\Bbb R}$ 
on $int(\Delta) \cap S_{\gamma}(\tfrac{p}{q})$ 
has a symmetry with respect to the center. 
There exists no  symmetry of 
$\mathrm{ent}: int(\Delta) \rightarrow {\Bbb R}$ 
on $int(\Delta) \cap S_{\beta}(\tfrac{p}{q})$ (hence on $int(\Delta) \cap S_{\alpha}(\tfrac{p}{q})$) in general.   
Later we shall compute $ \min \mathrm{Ent}(N(r), \Omega_A)$ for some manifolds  having a symmetry themselves 
(see Lemma~\ref{lem_bunbo2} and Proposition~\ref{prop_whitehead}), 
but these cases are rare.  
\end{rem}

We prove the following monotonicity of $ \min \mathrm{Ent}(\cdot, \Omega_A)$. 

\begin{lem}
\label{lem_monotonicity}
Let $p,p' \in {\Bbb N}$ and $q \in {\Bbb Z}$ such that 
$(p,q)$ and $(p',q)$ are coprime pairs. 
If $|q| \ne 1$ and $|1 +  \tfrac{p'}{q}| > |1 +  \tfrac{p}{q}|$, then 
$\min \mathrm{Ent}(N(\tfrac{p'}{q}), \Omega_A) > \min \mathrm{Ent}(N(\tfrac{p}{q}), \Omega_A)$. 
\end{lem}

\begin{proof} 
We use Lemma~\ref{lem_MNE-A}. 
Put $r'= \tfrac{p'}{q}$ and $r=\tfrac{p}{q}$. 
The sets 
$\Delta \cap S_{\beta}(r')$ and $\Delta \cap S_{\beta}(r)$ lie on the lines 
$y= (\tfrac{-2}{1+r'})x + \tfrac{1}{1+r'}$ and $y= (\tfrac{-2}{1+r})x + \tfrac{1}{1+r}$ respectively. 
(These lines go through $[\tfrac{1}{2},0] \in \partial \Delta$.) 
One has the inequality $ |\tfrac{-2}{1+r'}|< |\tfrac{-2}{1+r}|$  between the slopes. 
Thus for any $a'= [x', y'] \in int (\Delta) \cap S_{\beta}(r')$, there exists a unique point 
$a=  [x, y'] \in int (\Delta) \cap S_{\beta}(r)$ with the same second coordinate $y'$. 
Since $|\tfrac{1}{2}-x| < |\tfrac{1}{2}-x'|$,  one sees that 
$\lambda_{[x, y']}< \lambda_{[x', y']}$ (cf.  Lemma~\ref{lem_concave_sym}(2)). 
The condition $|q| \ne 1$ says that 
$A$-faces for both $N(\tfrac{p'}{q})$ and $N(\tfrac{p}{q})$ enjoy $(*)$ in Theorem~\ref{thm_key}. 
Hence by Lemma~\ref{lem_Ent_r}(2), 
$$\mathrm{Ent}_{r'}(\overline{a'})= (1 - |\tfrac{y'}{q}|) \lambda_{[x', y']}> 
(1 - |\tfrac{y'}{q}|) \lambda_{[x, y']}= \mathrm{Ent}_{r}(\overline{a}).$$
Since this holds for any $a' \in int (\Delta) \cap S_{\beta}(r')$, the proof is done. 
\end{proof}

\begin{lem}
\label{lem_bunbo2}
Suppose that $|q|=2$. 
Then 
\begin{enumerate}
\item[(1)] 
$\min \mathrm{Ent}(N(\tfrac{p}{q} ), \Omega_A)  = 2 \log( \tfrac{3+ \sqrt{5}}{2})$ if $\tfrac{p}{q} = \tfrac{3}{-2}, \tfrac{1}{-2}$, 

\item[(2)] 
$\min \mathrm{Ent}(N(\tfrac{p}{q} ), \Omega_A)  = 4 \log \lambda_{(4,2,1)} \approx 2.5318$ if $\tfrac{p}{q} = \tfrac{5}{-2}, \tfrac{1}{2}$, and 

\item[(3)]
$\min \mathrm{Ent}(N(\tfrac{p}{q} ), \Omega_A)  >   4 \log \lambda_{(4,2,1)}$ otherwise.   
\end{enumerate}
\end{lem}

\begin{proof} 
For the proof of  (1), see  \cite[Proposition~4.13]{KT1}. 
In fact  in this case,   
the  center of each $A$-face $\Omega_A$  reaches $\min \mathrm{Ent}(N(\tfrac{p}{q} ), \Omega_A)$.

Let us turn to the proof of (2). 
By Theorem~\ref{thm_entropy_equiv}, $N(\tfrac{5}{-2})$ and $N(\tfrac{1}{2})$ are entropy equivalent. 
Put $r_0= \tfrac{5}{-2}$. 
We consider the $A$-face $\Omega_A$ (on $\partial B_{r_0}(2)$) 
whose endpoints are 
$\overline{(3,2,2)}, \overline{(1,0,-1)} \in H_2(N(r_0), \partial N(r_0))$.  
We now prove that $\mathrm{ent}_{r_0}|_{\Omega_A}$ has a minimum at the center of $\Omega_A$. 
The ray from the origin and  through 
$\overline{(4,2,1)} \in int(C_{\Omega_A}) $ passes through the center of $\Omega_A$. 
(In other words, the ray from the origin, through $(2,-1)$ in the $(\overline{\mathfrak{a}_{r_0}}, \overline{\mathfrak{b}_{r_0}})$ coordinates, 
passes through the center of $\Omega_A$, see Figure~\ref{fig_ExBall_Nr}.) 
For $k > \ell$, let
$$p_{\pm}(k, \ell)= (4,2,1)k \pm  (2,2,3) \ell = (4k \pm 2\ell, 2k \pm  2\ell, k \pm 3 \ell).$$ 
Observe that $p_{\pm}(k, \ell)$ are elements of $ int(C_{\Delta}) \cap S_{\beta}(r_0)$, 
and $\overline{p_{\pm}(k, \ell)} \in int(C_{\Omega_A}) $ have the same Thurston norm. 
To show that the center of $\Omega_A$ achieves the minimum of $\mathrm{ent}_{r_0}|_{\Omega_A}$, 
it suffices to prove that $\overline{p_+(k, \ell)}$ and $\overline{p_-(k, \ell)}$ have the same entropy for each $k$, $\ell$ such that $k > \ell$. 
To do this, we  show that 
$p_+(k, \ell)$ and $p_-(k, \ell)$ have the same dilatation 
(since in this case, $\lambda(a) = \lambda_{r_0}(\overline{a})$ for $a \in S_{\beta}(r_0)$ such that $\overline{a} \in int(C_{\Omega_A})$). 
The dilatation $\lambda(p_+(k, \ell))$ (resp. $\lambda(p_-(k, \ell))$) is the largest real root of the polynomial 
\begin{eqnarray*}
f_{(4k+2 \ell, 2k+ 2\ell, k+ 3\ell)}(t) 
&=& - t^{- \ell} (1+ t^{k+ \ell}) (t^k + t^{3k} - t^{\ell} - t^{2k+ \ell} - t^{4k+ \ell} + t^{k + 2 \ell} + t^{3k + 2 \ell})
\\
\mbox{(resp.\ }f_{(4k- 2 \ell, 2k -  2\ell, k- 3\ell)}(t) 
&=& - t^{- 2\ell} (t^k + t^{\ell}) (t^k + t^{3k} - t^{\ell} - t^{2k+ \ell} - t^{4k+ \ell} + t^{k + 2 \ell} + t^{3k + 2 \ell}).
\end{eqnarray*}
Since each of polynomials $ - t^{- 2\ell} (t^k + t^{\ell}) $ and $- t^{- \ell} (1+ t^{k+ \ell}) $ have no real roots greater than $1$, the proof of (2) is done. 

The claim (2) together with Lemma~\ref{lem_monotonicity} leads to (3). 
\end{proof}

\begin{lem} 
\label{lem_bunbo3}
Suppose that $|q| =3$. 
Then $\min \mathrm{Ent}(N(\tfrac{p}{q}), \Omega_A) > 2.0918$. 
\end{lem}

\begin{proof} 
By Lemma~\ref{lem_monotonicity} and Theorem~\ref{thm_entropy_equiv}, 
$$\min \mathrm{Ent}(N(\tfrac{p}{q}), \Omega_A) >  \min \mathrm{Ent}(N(\tfrac{2}{-3}), \Omega_A)=  \min \mathrm{Ent}(N(\tfrac{4}{-3}), \Omega_A)
\ \mbox{if}\ |q|=3\ \mbox{and\ } \tfrac{p}{q} \ne \tfrac{2}{-3}, \tfrac{4}{-3}. 
$$
Thus it suffices to prove that $\min \mathrm{Ent}(N(\tfrac{2}{-3}), \Omega_A)> 2.0918$. 
We consider the $A$-face $\Omega_A$ (on $ \partial B_{2/-3}(2)$) for $N(\tfrac{2}{-3})$ whose endpoints are 
$\overline{(1,3,1)}, \overline{(1,0,-1)} \in H_2(N(\tfrac{2}{-3}), \partial N(\tfrac{2}{-3}))$. 
Take fibered classes 
$$a_1= (201, 312, 7),\  a_2 = (201, 309, 5),\  a_3= (201, 306, 3) \in int(C_{\Delta}) \cap S_{\beta}(\tfrac{2}{-3}).$$
Then  $\overline{a_1}$, $\overline{a_2}$ and $\overline{a_3}$ are elements of  $int(C_{\Omega_A})$. 
One can check that the Thurston norms of $\overline{a_1}$, $\overline{a_2}$ and $\overline{a_3}$ are the same. 
Note that $\lambda(a)= \lambda_{2/(-3)}(\overline{a})$ for $a \in S_{\beta}(\tfrac{2}{-3})$ such that $\overline{a} \in int(C_{\Omega_A})$. 
One sees that 
\begin{eqnarray*}
&\ & \lambda_{(201, 312, 7)}= 1.00542189 \cdots
\\ 
&>& \lambda_{ (201, 309, 5)}= 1.00542166 \cdots
\\
&<& \lambda_{(201, 306, 3)}= 1.00542185 \cdots. 
\end{eqnarray*}
The fibered class $a_1$ is equal to 
$(\tfrac{201}{506}, \tfrac{312}{506}, \tfrac{7}{506}) \in int(\Delta)$ projectively. 
Hence $\min \mathrm{Ent}(N(\tfrac{2}{-3}), \Omega_A)$ is achieved by a unique point 
$[x,y] \in int(\Delta) \cap  S_{\beta}(\tfrac{2}{-3})$ such that 
$0 < y < \tfrac{312}{506}$. 
This together with Lemma~\ref{lem_MNE-A} implies that 
\begin{eqnarray*}
\min \mathrm{Ent}(N(\tfrac{2}{-3}), \Omega_A) &=& 
\min \{(1 - |\tfrac{y}{-3}| ) \log \lambda_{[x,y]}\ |\ [x,y] \in int (\Delta) \cap S_{\beta} (\tfrac{2}{-3}),\ 0 < y < \tfrac{312}{506}\} 
\\
&>& (1 - \tfrac{312}{3 \times 506}) \log \lambda_{[\tfrac{1}{2}, \tfrac{1}{2}]} 
\\
&>& \tfrac{402}{506} \times 2.633
\\ 
&>&  2.0918. 
\end{eqnarray*}
This completes the proof. 
\end{proof}

\begin{lem}
\label{lem_large-bunbo}
\ 
\begin{enumerate}
\item[(1)] 
Let  $\tfrac{p}{q} \in (- \infty, -2) \cup (0, \infty)$ such that  $p+2q \ne 1$. 
Suppose that $p+q \ge 4$. 
Then  $\min \mathrm{Ent}(N(\tfrac{p}{q}), \Omega_S) > 1.97475 $. 

\item[(2)] 
Suppose that  $|q| \ge 4$. 
Then $\min \mathrm{Ent}(N(\tfrac{p}{q}), \Omega_A) >1.97475$. 
\end{enumerate}
\end{lem}

\begin{proof} 
The claim (1) is immediate from 
$$\min \mathrm{Ent}(N(\tfrac{p}{q} ), \Omega_S) =  (1- \tfrac{1}{p+q}) \log \lambda_{[\tfrac{p}{2p+2q}, \tfrac{p}{2p+2q}]}
> (1 - \tfrac{1}{4}) \times 2.633 = 1.97475 .$$

Let us turn to the claim (2). 
By Lemma~\ref{lem_MNE-A}, 
$$\min \mathrm{Ent}(N(\tfrac{p}{q}), \Omega_A)  = \min \{(1 - |\tfrac{y}{q}| ) \log \lambda_{[x,y]}\ |\ [x,y] \in int (\Delta) \cap S_{\beta} (\tfrac{p}{q})\}.$$
Since $|q| \ge 4$, one sees that for any $ [x,y] \in int (\Delta) \cap S_{\beta} (\tfrac{p}{q})$, 
$$(1 - |\tfrac{y}{q}| ) \log \lambda_{[x,y]} >   (1 - \tfrac{1}{4}) \times 2.633 = 1.97475.$$
This completes the proof. 
\end{proof}

\begin{proof}[Proof of Theorem~\ref{thm_key-SA}.] 
We have already proved the claim (2), see Lemmas~\ref{lem_bunbo2}, \ref{lem_bunbo3} and \ref{lem_large-bunbo}(2). 
Let us prove the claim (1).  
By Lemma~\ref{lem_large-bunbo}(1), it is enough to consider the case $p+q <4$. 
For $\tfrac{p}{q} \in  {\mathcal Hyp} \cap (- \infty, -2)$ such that $p+2q \ne 1$,  
one has by Theorem~\ref{thm_entropy_equiv}, 
$$\min \mathrm{Ent}(N(\tfrac{p}{q}), \Omega_S) = \min \mathrm{Ent} (N(\tfrac{-2q-p}{q}), \Omega_S).$$ 
If $\tfrac{p}{q} \in {\mathcal Hyp} \cap (0, \infty)$, then  $p+2q \ge 3$ (hence $p+2q \ne 1$). 
Thus it suffices to consider the case $\tfrac{p}{q} \in {\mathcal Hyp} \cap (0, \infty)$ such that $p+q <4$. 
The pairs $(p,q)$ with  $p+q<4$ are given by $(p,q)= (1,1), (2,1), (1,2)$. 
By Lemma~\ref{lem_MNE-S}, 
\begin{eqnarray*}
\min \mathrm{Ent}(N(\tfrac{1}{1}), \Omega_S) &=& 2 \log \lambda_{(1,1,-2)}= 2 \log \delta(D_4)  \approx 1.6628, 
\\
\min \mathrm{Ent}(N(\tfrac{2}{1}), \Omega_S) &=& 2 \log \lambda_{(1,1,-1)} =2 \log( \tfrac{3+ \sqrt{5}}{2}) \approx 1.9248, \ \mbox{and}
\\
\min \mathrm{Ent}(N(\tfrac{1}{2}), \Omega_S) &=& 4 \log \lambda_{(1,1,-4)}  \approx  2.9314.
\end{eqnarray*}
This completes the proof. 
\end{proof}

\subsection{Proof of Theorem~\ref{thm_main}(1)}
\label{subsection_Proof_thm_main1}

The idea of the proof is as follows. 
We define a finite set $L_K \subset {\mathcal Hyp}$ for  $K>2$ which consists of 
irreducible rational numbers  $\tfrac{p}{q} \in {\mathcal Hyp}$ with $p \in {\Bbb N}$ such that 
\begin{itemize}
\item 
$|q| \le K$ if $\tfrac{p}{q} \in (-2,0)$, 

\item 
$p+q \le K$ if $\tfrac{p}{q} \in (- \infty, -2) \cup (0, \infty)$. 
\end{itemize}
We fix $K_0= 100000$. 
First we prove that for a primitive fibered class $a \in H_2(N, \partial N)$ such that  $\phi_a \in \mathcal{M}$, 
the normalized entropy of $\widehat{\phi}_a$ is greater than $2.5803$ 
if all  the slopes $b_{\alpha}(a)$, $b_{\beta}(a)$, $b_{\gamma}(a)$ 
of $a$ enjoy $b_{\alpha}(a) , b_{\beta}(a), b_{\gamma}(a) \in {\mathcal Hyp} \setminus L_{K_0}$. 
%(see Lemma~\ref{lem_NEnt_hat} and Proposition~\ref{prop_FiniteSet}). 
Next we prove the following for any $\epsilon>0$: 
For all but finitely many primitive fibered classes $a \in H_2(N, \partial N)$ satisfying $\phi_a \in \mathcal{M}$, 
if one of the boundary slopes of $a$ is an element of $L_{K_0}$, then 
the normalized entropy of $\widehat{\phi}_a$ is greater than $2 \log(\tfrac{3+ \sqrt{5}}{2})-\epsilon$. 
These together with Theorem~\ref{thm_three}  lead to Theorem~\ref{thm_main}(1).

\begin{lem} 
\label{lem_NEnt_hat}
Let $a=(x,y,z)$ be a primitive fibered class of $H_2(N, \partial N)$ such that  $\phi_a \in \mathcal{M}$. 
Then $\phi_{\sigma(a)} \in \mathcal{M}$ for $\sigma(a) \in int(C_{\Delta})$, and 
$$\mathrm{Ent}(\widehat{\phi}_a)= 
\mathrm{Ent}(\widehat{\phi}_{\sigma(a)})= 
(1 - |\tfrac{x'}{q_{\alpha}(\sigma(a))}| -  |\tfrac{y'}{q_{\beta}(\sigma(a))}| -  |\tfrac{z'}{q_{\gamma}(\sigma(a))}|) \log \lambda (a'), $$
where $a' = (x', y', z')$ is the rational class of $int(\Delta)$ which is projectively equal to $\sigma(a)$. 
\end{lem}

\begin{proof} 
Clearly  $\phi_{a} \in \mathcal{M}$ implies that $\phi_{\sigma(a)} \in \mathcal{M}$ and 
$\mathrm{Ent}(\widehat{\phi}_a)= \mathrm{Ent}(\widehat{\phi}_{\sigma(a)})$. 
The dilatation $\lambda(\widehat{\phi}_{\sigma(a)})$ equals $\lambda(\sigma(a))(= \lambda(\phi_{\sigma(a)}))$
since $\phi_{\sigma(a)} \in \mathcal{M}$. 
If we set $\sigma(a)= (x,y,z)$, then 
$$\sharp (\partial F_{\sigma(a)})=  
|\tfrac{x}{q_{\alpha}(\sigma(a))}| +  |\tfrac{y}{q_{\beta}(\sigma(a))}| +  |\tfrac{z}{q_{\gamma}(\sigma(a))}|.$$
By  definition of the normalized entropy, 
$$\mathrm{Ent}(\widehat{\phi}_{\sigma(a)})=
 (\|\sigma(a)\| - |\tfrac{x}{q_{\alpha}(\sigma(a))}| -  |\tfrac{y}{q_{\beta}(\sigma(a))}| -  |\tfrac{z}{q_{\gamma}(\sigma(a))}|) 
 \log \lambda (\sigma(a)).$$
 On the other hand 
$x'= \tfrac{x}{\|\sigma(a)\|}$, $y'= \tfrac{y}{\|\sigma(a)\|}$, $z'= \tfrac{z}{\|\sigma(a)\|} $ and 
$\log \lambda(a')= \|\sigma(a)\| \log \lambda(\sigma(a))$ 
since $a'$ is projectively equal to $\sigma(a)$. 
Substituting these equalities for 
$ (1 - |\tfrac{x'}{q_{\alpha}(\sigma(a))}| -  |\tfrac{y'}{q_{\beta}(\sigma(a))}| -  |\tfrac{z'}{q_{\gamma}(\sigma(a))}|) \log \lambda (a') $, 
one finds that it is equal to $\mathrm{Ent}(\widehat{\phi}_{\sigma(a)})$. 
\end{proof}

%消すな! 
%消すな! 
%\begin{lem} 
%\label{lem_NEnt_hat}
%Let $a=(x,y,z)$ be a primitive fibered class of $int(C_{\Delta})$ such that  $\phi_a \in \mathcal{M}$. 
%Then the normalized entropy of $\widehat{\phi}_a$ is given by 
%$$\mathrm{Ent}(\widehat{\phi}_a)= (1 - |\tfrac{x'}{q_{\alpha}(a)}| -  |\tfrac{y'}{q_{\beta}(a)}| -  |\tfrac{z'}{q_{\gamma}(a)}|) \log \lambda (a'), $$
%where $a' = (x', y', z')$ is the rational class of $int(\Delta)$ which is projectively equal to the class $a$. 
%\end{lem}
%
%
%\begin{proof} 
%The dilatation $\lambda(\widehat{\phi}_a)$ equals $\lambda(a)(= \lambda(\phi_a))$ since $\phi_a \in \mathcal{M}$. 
%Notice that $\sharp (\partial F_a)=  |\tfrac{x}{q_{\alpha}(a)}| +  |\tfrac{y}{q_{\beta}(a)}| +  |\tfrac{z}{q_{\gamma}(a)}|$. 
%Hence by  definition of the normalized entropy, 
%$$\mathrm{Ent}(\widehat{\phi}_a)= (\|a\| - |\tfrac{x}{q_{\alpha}(a)}| -  |\tfrac{y}{q_{\beta}(a)}| -  |\tfrac{z}{q_{\gamma}(a)}|) \log \lambda (a).$$
%On the other hand 
%$x'= \tfrac{x}{\|a\|}$, $y'= \tfrac{y}{\|a\|}$, $z'= \tfrac{z}{\|a\|} $ and $\log \lambda(a')= \|a\| \log \lambda(a)$. 
%Substituting these equalities for $ (1 - |\tfrac{x'}{q_{\alpha}(a)}| -  |\tfrac{y'}{q_{\beta}(a)}| -  |\tfrac{z'}{q_{\gamma}(a)}|) \log \lambda (a') $, 
%one finds that it is equal to $\mathrm{Ent}(\widehat{\phi}_a)$. 
%\end{proof} 
%

\begin{prop}
\label{prop_FiniteSet}
Suppose that $b_{\alpha}(a) , b_{\beta}(a), b_{\gamma}(a) \in {\mathcal Hyp} \setminus L_{K_0}$ for a rational class $a= (x,y,z) \in int (\Delta)$. 
Then 
$$(1 - |\tfrac{x}{q_{\alpha}(a)}| - |\tfrac{y}{q_{\beta}(a)}| - |\tfrac{z}{q_{\gamma}(a)}|) \log \lambda(a)> 2.5803.$$
\end{prop}

\noindent 
We need the following lemma for the proof of Proposition~\ref{prop_FiniteSet}.

\begin{lem}
\label{lem_FiniteSet-b}
Let us take $K> K' >2$. 
Let $p \in {\Bbb N}$ and $q \in {\Bbb Z}$ be coprime such that $\tfrac{p}{q} \in {\mathcal Hyp} \setminus L_K$. 
Then the following holds.  
\begin{enumerate}
\item[(1)] 
If $|\tfrac{p}{q}| \le K'$, then 
$|q| \ge \tfrac{K}{1+K'}$. 
If $|\tfrac{p}{q}| > K'$, then  
$0 <  y < \tfrac{1}{-1+K'}$ for  any $(x,y,z) \in int(\Delta) \cap S_{\beta}(\tfrac{p}{q})$.  

\item[(2)] 
$|\tfrac{y}{q}| < \max \{ \tfrac{1+K'}{K} ,  \tfrac{1}{-1+K'}\}$ for any $(x,y,z) \in int(\Delta) \cap S_{\beta}(\tfrac{p}{q})$. 
\end{enumerate}
 \end{lem}

\begin{proof} 
(1) 
If $\tfrac{p}{q} \in (-2,0)$, then 
$|\tfrac{p}{q} |<2$. 
The assumption $\tfrac{p}{q} \not\in L_K$ implies that $|q|>K > \tfrac{K}{1+K'}$. 

Let us consider the case 
$\tfrac{p}{q} \in (- \infty, -2) \cup (0, \infty)$. 
Suppose that $|\tfrac{p}{q}| \le K'$ and $q  < 0$. 
Then $p \le -K'q$. 
One has $p+q > K$ since $\tfrac{p}{q} \not\in L_K$. 
Hence $-q < p-K \le -K'q - K$. 
One obtains 
$(K'-1)q \le -K$. 
Thus $|q|= -q> \tfrac{K}{K'-1}> \tfrac{K}{K'+1}$. 

Suppose that $|\tfrac{p}{q}| \le K'$ and $q  > 0$. 
In this case $p \le K'q$. 
Since $p+q > K$, one has $q> K-p \ge K- K'q$. 
Thus $q > \tfrac{K}{1+K'}$. 
The proof of the first part is done. 

The second part can be proved by using Lemma~\ref{lem_boundary}(2). 
In fact for $[x,y] \in int(\Delta) \cap S_{\beta}(r)$ such that $|r| > K'>2$, 
one has $0 < y < \tfrac{-1}{1+r}$ when $r \in (- \infty, -2)$ 
(resp. $0 < y < \tfrac{1}{1+r}$ when $r \in (0, \infty)$). 
This leads to the second part. 
\medskip

\noindent
(2) 
If $|\tfrac{p}{q}| \le K'$, then $|\tfrac{y}{q}| < \tfrac{1}{|q|} \le \tfrac{1+K'}{K}$ by the first part of (1). 
If $|\tfrac{p}{q}| > K'$, then by the second part of (1), 
we have 
$|\tfrac{y}{q}| < \tfrac{1}{|q|} \times \tfrac{1}{-1+K'} \le \tfrac{1}{-1+K'}$. 
These imply the desired inequality. 
\end{proof} 

\noindent
Similarly, one can prove: 

\begin{lem}
\label{lem_FiniteSet-a}
Let us take $K> K' >2$. 
Let $p \in {\Bbb N}$ and $q \in {\Bbb Z}$ be coprime such that $\tfrac{p}{q} \in {\mathcal Hyp} \setminus L_K$. 
\begin{enumerate}
\item[(1)] 
If $|\tfrac{p}{q}| \le K'$, then 
$|q| \ge \tfrac{K}{1+K'}$. 
If $|\tfrac{p}{q}| > K'$, then  
$0 <  x < \tfrac{1}{-1+K'}$ for  any $(x,y,z) \in int(\Delta) \cap S_{\alpha}(\tfrac{p}{q})$.  

\item[(2)] 
$|\tfrac{x}{q}| < \max \{ \tfrac{1+K'}{K} ,  \tfrac{1}{-1+K'}\}$ for any $(x,y,z) \in int(\Delta) \cap S_{\alpha}(\tfrac{p}{q})$. 
\end{enumerate}
\end{lem}

\begin{proof}[Proof of Proposition~\ref{prop_FiniteSet}.] 
Let $K= K_0 (= 100000)$ and $K' = 999$. 
Since $b_{\gamma}(a) = \tfrac{p_{\gamma}(a)}{q_{\gamma}(a)} \in   {\mathcal Hyp} \setminus L_{K_0}$ and $b_{\gamma}(a) \in   (- \infty, -2) \cup (0, \infty)$, 
the inequality $p_{\gamma}(a) + q_{\gamma}(a) > K_0$ holds.
By using the same argument in the proof of Lemma~\ref{lem_NormChange}, we obtain the upper bound on $|\tfrac{z}{q_{\gamma}(a)}|$: 
$$|\tfrac{z}{q_{\gamma}(a)}| = |\tfrac{1}{p_{\gamma}(a)+ q_{\gamma}(a)}| < \tfrac{1}{K_0}=\tfrac{1}{100000}.$$
By Lemmas~\ref{lem_FiniteSet-b}(2) and \ref{lem_FiniteSet-a}(2), we have 
$$ |\tfrac{x}{q_{\alpha}(a)}| , |\tfrac{y}{q_{\beta}(a)}| <  \max \{\tfrac{1+K'}{K_0}, \tfrac{1}{-1+K'}\} = \tfrac{1+K'}{K_0}= \tfrac{1}{100}.$$
Thus we have  lower bounds 
$1 - |\tfrac{x}{q_{\alpha}(a)}| - |\tfrac{y}{q_{\beta}(a)}| - |\tfrac{z}{q_{\gamma}(a)}| > 1 - \tfrac{1}{50}- \tfrac{1}{100000}= 0.97999$ and 
$\log \lambda(a) \ge \log \lambda_{[\tfrac{1}{2}, \tfrac{1}{2}]}> 2.633$. 
These two bounds  give us the desired inequality. 
\end{proof}

%For $r \in {\mathcal Hyp} \setminus \{1\}$, let $\Omega$ be a fibered face of $N(r)$ enjoying $(*)$ in Theorem~\ref{thm_key}. 
%Let $\epsilon>0$ be any number. 
%The next proposition implies the following.  
%When $a \in H_2(N, \partial N)$ is a primitive fibered class such that $\phi_a \in \mathcal{M}$, 
%the normalized entropy of $\widehat{\phi}_a$ is greater than $\min \mathrm{Ent}(N(r), \Omega) - \epsilon$ 
%for any such a class  $a$ but finitely many exceptions. 

\begin{prop}
\label{prop_finitely-many}
Let $p \in {\Bbb N}$ and $q \in {\Bbb Z}$ be coprime such that $\tfrac{p}{q} \in {\mathcal Hyp} \setminus \{1\}$. 
Let $\epsilon>0$ be any number. 
\begin{enumerate}
\item[(1)] 
Suppose that $\tfrac{p}{q} \in (- \infty, -2) \cup  (0, \infty)$ and $p+2q \ne 1$. 
Then 
$$(1 - |\tfrac{x}{q_{\alpha}(a)}| - |\tfrac{y}{q_{\beta}(a)}| - |\tfrac{z}{q}|) \log \lambda(a) > \min \mathrm{Ent}(N(\tfrac{p}{q}), \Omega_S) - \epsilon$$
for any rational class $a= (x,y,z) \in int (\Delta) \cap S_{\gamma}(\tfrac{p}{q})$ but finitely many exceptions.

\item[(2)] 
Suppose that $|q| \ne 1$. 
Then 
$$(1 - |\tfrac{x}{q_{\alpha}(a)}| - |\tfrac{y}{q}| - |\tfrac{z}{q_{\gamma}(a)}|) \log \lambda(a) >  \min \mathrm{Ent}(N(\tfrac{p}{q}), \Omega_A) - \epsilon$$
for any rational class $a= (x,y,z) \in int (\Delta) \cap S_{\beta}(\tfrac{p}{q})$ but finitely many exceptions. 
\end{enumerate}
\end{prop}

\noindent
The following lemma is needed for the proof of Proposition~\ref{prop_finitely-many}. 

\begin{lem}
\label{lem_SmallChange}
Let $\epsilon'>0$ be any number. 
\begin{enumerate}
\item[(1)] 
Let $r \in {\mathcal Hyp} \setminus \{1\}$  and $r \in (- \infty, -2) \cup (0, \infty)$. 
Then 
$|\tfrac{x}{q_{\alpha}(a)}| < \epsilon'$ and $|\tfrac{y}{q_{\beta}(a)}| < \epsilon'$ for any rational class $a= (x,y,z) \in int (\Delta) \cap S_{\gamma}(r)$ 
but finitely many exceptions.

\item[(2)] 
Let $r \in {\mathcal Hyp} \setminus \{1\}$. 
Then 
$|\tfrac{x}{q_{\alpha}(a)}| < \epsilon'$ and $|\tfrac{z}{q_{\gamma}(a)}| < \epsilon'$ for any rational class $a= (x,y,z) \in int (\Delta) \cap S_{\beta}(r)$ 
but finitely many exceptions. 
\end{enumerate}
\end{lem}

\begin{proof} 
Take $K> K'>2$ so that 
$\max\{\tfrac{1+K'}{K}, \tfrac{1}{-1+K'}\} < \epsilon'$. 
(Note that $1 \in L_K$.) 
We see that 
$r \ne 1$ implies that 
$int(\Delta) \cap S_{\gamma}(r) \cap S_{\alpha}(r')$ or 
$int(\Delta) \cap S_{\gamma}(r) \cap S_{\beta}(r')$ is at most a single point for any $r'$. 
This means that  the set of rational classes $a= (x,y,z) \in int (\Delta) \cap S_{\gamma}(r)$ 
such that $b_{\alpha}(a) \in  L_K$ or $b_{\beta}(a) \in  L_K$ is finite whenever  $r \ne 1$. 
If $b_{\alpha}(a) \in {\mathcal Hyp} \setminus L_K$ (resp.  $b_{\beta}(a) \in {\mathcal Hyp} \setminus L_K$), then 
$|\tfrac{x}{q_{\alpha}(a)}| < \epsilon'$ (resp. $|\tfrac{y}{q_{\beta}(a)}| < \epsilon'$), see 
Lemma~\ref{lem_FiniteSet-a}(2) (resp. Lemma~\ref{lem_FiniteSet-b}(2)). 
Thus the proof of (1) is done. 
(Note that this is not true for $r=1$, since $\Delta \cap S_{\alpha}(1) = \Delta \cap S_{\beta}(1)= \Delta \cap S_{\gamma}(1)$.)

The proof of (2) is similar to that of (1). 
\end{proof}

\begin{proof}[Proof of Proposition~\ref{prop_finitely-many}.] 
(1) 
By Lemma~\ref{lem_NormChange}, one has 
$$(1- |\tfrac{x}{q_{\alpha}(a)}| - |\tfrac{y}{q_{\beta}(a)}| - |\tfrac{z}{q}|) \log \lambda(a) 
= (1- |\tfrac{x}{q_{\alpha}(a)}| - |\tfrac{y}{q_{\beta}(a)}| - \tfrac{1}{p+q}) \log \lambda(a).$$
Lemma~\ref{lem_SmallChange} says that 
for any $\epsilon'>0$, the following inequality holds: 
$$ (1- |\tfrac{x}{q_{\alpha}(a)}| - |\tfrac{y}{q_{\beta}(a)}| - \tfrac{1}{p+q}) \log \lambda(a) >(1- \tfrac{1}{p+q} - 2 \epsilon') \log \lambda(a)$$
for any rational class $a= (x,y,z) \in int(\Delta) \cap S_{\gamma}(\tfrac{p}{q})$ but finitely many exceptions. 
If we set $a_0= [\tfrac{p}{2p+2q}, \tfrac{p}{2p+2q}]$, then 
by Lemma~\ref{lem_MNE-S}, 
$$ (1 - \tfrac{1}{p+q}) \log \lambda (a_0) = \min \mathrm{Ent}(N(\tfrac{p}{q}), \Omega_S).$$ 
For any $\epsilon>0$, choose a small number $\epsilon'>0$ so that 
$$(1 - \tfrac{1}{p+q} - 2 \epsilon') \log \lambda(a_0) >  \min \mathrm{Ent}(N(\tfrac{p}{q}), \Omega_S) - \epsilon.$$
For any $a \in int(\Delta) \cap S_{\gamma}(\tfrac{p}{q})$, 
one has 
\begin{eqnarray*}
(1 - \tfrac{1}{p+q} - 2 \epsilon') \log \lambda(a) &\ge& (1 - \tfrac{1}{p+q} - 2 \epsilon') \log \lambda(a_0) 
\\
&>& \min \mathrm{Ent}(N(\tfrac{p}{q}), \Omega_S) - \epsilon.
\end{eqnarray*}
The proof of (1) is done. 
\medskip

\noindent
(2)
Let $a_0= [x_0, y_0]$ be the unique point of $int(\Delta) \cap S_{\beta}(\tfrac{p}{q})$ which enjoys 
$$\min \mathrm{Ent}(N(\tfrac{p}{q}), \Omega_A)= (1- |\tfrac{y_0}{q}|) \log \lambda(a_0),$$
see Lemma~\ref{lem_MNE-A}. 
The function $\tfrac{1}{\log \lambda}$ restricted on $int(\Delta) \cap S_{\beta}(\tfrac{p}{q})$ is strictly concave, and 
the entropy $\log \lambda (a)$ for $a \in int(\Delta) \cap S_{\beta}(\tfrac{p}{q})$ goes to $\infty$ 
as $a$ goes to a point on $\partial \Delta \cap  S_{\beta}(\tfrac{p}{q})$. 
This ensures the existence of $a_i = [x_i, y_i] \in int(\Delta) \cap S_{\beta}(\tfrac{p}{q})$ for $i \in \{-1,1\}$ satisfying the following. 
\begin{itemize}
\item 
$  \log \lambda(a_{-1})=\log \lambda(a_1)$, 
and $0 < y_{-1}< y_0< y_1< 1$.

\item 
$(1 - \tfrac{1}{|q|}) \log \lambda(a_i)> \min \mathrm{Ent}(N(\tfrac{p}{q}), \Omega_A)$ for $i \in \{ -1,1\}$.
\end{itemize}
For any $\epsilon>0$, one can choose a small number $\epsilon'>0$ such that for $i \in \{-1,1\}$, 
\begin{eqnarray*}
\epsilon &>& 2 \epsilon' \log \lambda(a_i), 
\\
(1 - \tfrac{1}{|q|} - 2 \epsilon') \log \lambda(a_i)&>& \min \mathrm{Ent}(N(\tfrac{p}{q}), \Omega_A).
\end{eqnarray*}
Then for any $a= [x,y] \in int(\Delta) \cap S_{\beta}(\tfrac{p}{q})$ but finitely many exceptions, one has the following 
by a consequence of Lemma~\ref{lem_SmallChange}: 
If either $y > y_1$ or $y < y_{-1}$, then 
\begin{eqnarray*}
(1 - |\tfrac{x}{q_{\alpha}(a)}| - |\tfrac{y}{q}| - \tfrac{z}{q_{\gamma}(a)}) \log \lambda(a) 
&>& (1 - \tfrac{1}{|q|} - 2 \epsilon') \log \lambda(a)
\\
&>&(1 - \tfrac{1}{|q|} - 2 \epsilon') \log \lambda(a_{\pm 1}) 
\\
&>&\min \mathrm{Ent}(N(\tfrac{p}{q}), \Omega_A) 
\\
&>&  \min \mathrm{Ent}(N(\tfrac{p}{q}), \Omega_A)  -  \epsilon.
\end{eqnarray*}

\noindent
If  $y_{-1} \le y \le y_1$, then one has $\log \lambda(a_i) \ge \log \lambda(a)$ for $i \in \{-1,1\}$. 
Thus 
\begin{eqnarray*}
(1 - |\tfrac{x}{q_{\alpha}(a)}| - |\tfrac{y}{q}| - \tfrac{z}{q_{\gamma}(a)}) \log \lambda(a) 
&>& (1 - |\tfrac{y}{q}| - 2 \epsilon') \log \lambda(a)
\\
&>& (1 - |\tfrac{y}{q}|) \log \lambda(a) - 2 \epsilon' \log \lambda(a_{\pm 1}) 
\\
&\ge&  \min \mathrm{Ent}(N(\tfrac{p}{q}), \Omega_A) - 2 \epsilon' \log \lambda(a_{\pm 1}) 
\\
&>&   \min \mathrm{Ent}(N(\tfrac{p}{q}), \Omega_A)  -  \epsilon.
\end{eqnarray*}
The third inequality $\ge$ comes from Lemma~\ref{lem_MNE-A}. 
This completes the proof of (2). 
\end{proof}

We are now ready to prove Theorem~\ref{thm_main}(1). 

\begin{proof}[Proof of Theorem~\ref{thm_main}(1).] 
We start by expressing Theorem~\ref{thm_three} together with \cite[Lemma~4.8]{KT1}  in the following way. 
\begin{claim}
\label{claim_thm_main}
Let $r \in \{-4,  \tfrac{3}{-2}, \tfrac{1}{-2}, 2\}$. 
For each $g \ge 3$, there exists a primitive fibered class $ h_g(r) \in H_2(N, \partial N)$ 
with the following properties. 
\begin{itemize}
\item 
One of the boundary slopes of $ h_g(r) $, say $b_{\beta}( h_g(r) )$ equals $r$, i.e, 
$h_g(r) \in S_{\beta}(r)$. 

\item 
$\phi_{h_g(r)}$ is a mapping class on a surface of genus $g$ such that $\phi_{h_g(r)} \in \mathcal{M}$ for large $g$,  and 
$$\displaystyle \lim_{g \to \infty} g  \log \lambda (\widehat{\phi}_{h_g(r)}) = \log( \tfrac{3+ \sqrt{5}}{2}).$$
In other words 
$\displaystyle \lim_{g \to \infty} \mathrm{Ent}  (\widehat{\phi}_{h_g(r)}) = 2\log( \tfrac{3+ \sqrt{5}}{2})$. 
\end{itemize}
\end{claim}
Recall that  $a_g$ is a primitive fibered class of $H_2(N, \partial N)$ such that 
$\phi_{a_g} \in \mathcal{M}$ and 
$\widehat{\delta}_g$ is achieved by $\widehat{\phi}_{a_g} \in \widehat{\mathcal{M}} \cap \mathrm{Mod}(\varSigma_g)$. 
Lemma~\ref{lem_NEnt_hat}, Proposition~\ref{prop_FiniteSet} and  Claim~\ref{claim_thm_main} 
tell us  that for large $g$, one of the boundary slopes of $a_g$ is an element of the finite set $L_{K_0}$. 
No boundary slopes of $a_g$ equal $1$ when $g \ge 2$. 
(We will prove Lemma~\ref{lem_PuncturedTorus} which implies this fact.) 
Thus for large $g$,  one of the boundary slopes of $a_g$ is  an element of  $L_{K_0} \setminus \{1\}$.

We shall  prove that 
for large $g$, one of the boundary slopes of $a_g$ must be either $-4$, $ \tfrac{3}{-2}$, $\tfrac{1}{-2}$ or $2$. 
We fix  $\epsilon>0$ so that $1.97475 - \epsilon > 2\log( \tfrac{3+ \sqrt{5}}{2})$. 
Set $L_{K_0}' = L_{K_0} \setminus \{ -4, \tfrac{3}{-2}, \tfrac{1}{-2}, 1,2\}$. 
Let $a \in H_2(N, \partial N)$ be a primitive fibered class such that 
one of the boundary slopes of $a$ is an element of $L_{K_0}'$ and 
$\{b_{\alpha}(a), b_{\beta}(a), b_{\gamma}(a)\} \cap \{ -4, \tfrac{3}{-2}, \tfrac{1}{-2}, 1,2\} = \emptyset$. 
Suppose that $\phi_a \in \mathcal{M}$. 
Then Theorem~\ref{thm_key} (or Theorem~\ref{thm_key-SA}) implies that 
$$\min \mathrm{Ent}(N(b_{\alpha}(a))), \min \mathrm{Ent}(N(b_{\beta}(a))), \min \mathrm{Ent}(N(b_{\gamma}(a)))> 1.97475.$$
It follows that $\mathrm{Ent}(\widehat{\phi}_a) > 1.97475 - \epsilon $ for any such a class $a$ but finitely many exceptions, 
which is ensured by  Lemma~\ref{lem_NEnt_hat} and Proposition~\ref{prop_finitely-many}. 
Thus for large $g$, one of the boundary slopes of $a_g$ must be an element of $\{-4,  \tfrac{3}{-2}, \tfrac{1}{-2}, 2\}$. 

Again by Proposition~\ref{prop_finitely-many}, 
the set of normalized entropies of mapping classes $\widehat{\phi}_a \in \widehat{\mathcal{M}}$ 
such that 
$\{b_{\alpha}(a), b_{\beta}(a), b_{\gamma}(a)\} \cap \{ -4, \tfrac{3}{-2}, \tfrac{1}{-2}, 2\} \ne \emptyset$ 
have no accumulation values  $< 2 \log (\frac{3+\sqrt{5}}{2})$. 
This together with Claim~\ref{claim_thm_main}  leads to  the conclusion.  
\end{proof}

\subsection{Proof of Theorem~\ref{thm_main}(2)} 
\label{subsection_Proof_thm_main2}

For $r \in {\mathcal Hyp} $, 
let  $\widehat{\delta}_g(r)$  be the minimum among dilatations of elements 
$\widehat{\phi}_a \in \widehat{\mathcal{M}} \cap \mathrm{Mod}(\varSigma_g)$ such that 
$a \in S_{\beta}(r)$ and  $\phi_a \in \mathcal{M}$. 
We set $\widehat{\delta}_g(r) =\infty$ when there exist no such elements. 
Clearly $\delta_g \le \widehat{\delta}_g \le \widehat{\delta}_g(r)$. 

The proof of Theorem~\ref{thm_main}(1) implies that 
for large $g$, $\widehat{\delta}_g$ is either 
$\widehat{\delta}_g(\tfrac{3}{-2})$, $\widehat{\delta}_g(\tfrac{1}{-2})$ or $\widehat{\delta}_g(2)$, 
because $N(-4) \simeq N(\tfrac{3}{-2})$. 
We prove: 

\begin{prop}
\label{prop_least-element}
$\min \{\widehat{\delta}_g(\tfrac{3}{-2}), \widehat{\delta}_g(\tfrac{1}{-2}), \widehat{\delta}_g(2)\} <  \widehat{\delta}_g(2)$ 
for each $g \ge 4$. 
\end{prop}

\begin{proof} 
Set $r= \tfrac{3}{-2}$ and $r' = -2-r = \tfrac{1}{-2}$. 
Recall that $\lambda_{(k,\ell)}$ is the largest real root of  the polynomial $f_{(k,\ell)}(t)$ 
as in Section~\ref{subsection_LTQuestion}. 
Let $k$ and $\ell$ be coprime integers such that $0< \ell < k$. 
By the discussion in Section~\ref{subsection_LTQuestion}, we see that 
$$\lambda(k \overline{\mathfrak{a}_{r}} \pm \ell \overline{\mathfrak{b}_{r}}) = \lambda(k \overline{\mathfrak{a}_{r'}} \pm \ell \overline{\mathfrak{b}_{r'}}) 
= \lambda_{(k,\ell)}.$$ 
The number $\widehat{\delta}_g(\tfrac{1}{-2})$ was computed in \cite[Theorem~1.4]{Hironaka}: 
\begin{eqnarray*}
\widehat{\delta}_g(\tfrac{1}{-2}) &=& \lambda_{(g+1,3)}\ \hspace{2mm}  \mbox{if} \ g \equiv 0,1,3,4 \pmod 6,\ \ g \ge 3,  
\\
\widehat{\delta}_g(\tfrac{1}{-2})  &=& \lambda_{(g+1,1)} \ \hspace{2mm} \mbox{if} \ g \equiv 2,5 \pmod 6,\ g \ge 5. 
\end{eqnarray*}
The following inequalities were proved in \cite[Proposition~4.26]{KT1}: 
\begin{eqnarray*}
\widehat{\delta}_g(\tfrac{3}{-2})&<& \widehat{\delta}_g(\tfrac{1}{-2})\ \hspace{2mm}\mbox{for}\ g \equiv 0,1,5,6,7,9 \pmod{10},\  g \ge 5,
\\
\widehat{\delta}_g(\tfrac{1}{-2}) &<& \widehat{\delta}_g(\tfrac{3}{-2}) \ \hspace{2mm}\mbox{for}\ g \equiv 3,8 \pmod{10},\ g \ge 3.
\end{eqnarray*}
Thus  it suffices to prove that 
$\widehat{\delta}_g(2)> \widehat{\delta}_g(\tfrac{1}{-2})$ for each $g \ge 4$.

Let $a$ be a fibered class of $S_{\beta}(2)$ 
such that $\phi_a \in \mathcal{M}$ and $\widehat{\phi}_a \in \widehat{\mathcal{M}} \cap \mathrm{Mod}(\varSigma_g)$ 
for $g \ge 3$. 
Then $\lambda(a)= \lambda_{(g, \ell)}$ for some $1 \le \ell < g$, 
see \cite[Lemma~4.1]{KT1}. 
There exists such a class $a$ whose dilatation $\lambda(a)$ equals $\lambda_{(g,1)}$. 
This is proved by \cite[Remark~4.18]{KT1} 
together with the monodromy $\Phi_{g \mathfrak{r}+ 1 \mathfrak{s}}$ 
(in the notation of \cite{KT1}). 
The inequality 
$\lambda_{(k,\ell)}< \lambda_{(k, \ell+1)}$ for $1 < \ell+1< k$ (see \cite[Lemma~4.15]{KT1}) 
gives the equality $\widehat{\delta}_g(2)= \lambda_{(g,1)}$ for  $g \ge 3$.

It holds that  $\lambda_{(g,1)}> \lambda_{(g+1,1)}$ for $g \ge 2$, see \cite[Proposition~4.3]{Hironaka}. 
Hence $\widehat{\delta}_g(2)> \widehat{\delta}_g(\tfrac{1}{-2})$ for $g \equiv 2,5 \pmod{6}$. 
We use the following claim  to prove $\widehat{\delta}_g(2)> \widehat{\delta}_g(\tfrac{1}{-2})$ for other cases. 
\begin{claim}[Proposition~4.17 in \cite{KT1}]
If $\lambda_{(k+1, \ell)}< \lambda_{(k,1)}$ for some $k \ge \ell \ge 2$, then 
$$\lambda_{(k+2, \ell)} < \lambda_{(k+1,1)}.$$
\end{claim}

\noindent
One can check that 
$\lambda_{(4,1)} \approx 1.2806 > \lambda_{(5,3)} \approx 1.2612$. 
Thus 
$\lambda_{(g,1)}> \lambda_{(g+1,3)}$ for all $g \ge 4$. 
This implies that 
$\widehat{\delta}_g(2)> \widehat{\delta}_g(\tfrac{1}{-2})$ for $g \equiv 0,1,3,4 \pmod{6}$. 
\end{proof} 

\begin{rem}
\label{rem_main}
From the proof of Proposition~\ref{prop_least-element}, 
we see the following: 
For large $g$ such that $g \equiv 0,1,5,6,7,9 \pmod{10}$ 
(resp.  such that $g \equiv 3,8 \pmod{10}$), 
$\widehat{\delta}_g$ is achieved by the monodromy of some $\varSigma_g$-bundle over the circle 
obtained from $N(\tfrac{3}{-2})$ (resp. $N(\tfrac{1}{-2})$)     
by Dehn filling both cusps.  
For many $g$ such that $g \equiv 2,4 \pmod{10}$, 
we have 
$\widehat{\delta}_g(\tfrac{3}{-2})< \widehat{\delta}_g(\tfrac{1}{-2})$, see \cite[Proposition~4.28]{KT1}. 
It might be true that  
$\widehat{\delta}_g(\tfrac{3}{-2})< \widehat{\delta}_g(\tfrac{1}{-2})$ holds for all $g \equiv 2,4 \pmod{10}$, 
see \cite[Question~4.32]{KT1}. 
\end{rem}

\subsection{Proofs of Theorems~\ref{thm_smallest_dil_ori}, \ref{thm_6+12i}, \ref{thm_bound6+12i} and \ref{thm_LTQuestion}}
\label{subsection_orientable}

First, we prove that there exists an element of $\widehat{\mathcal{M}}^+$ defined on $\varSigma_g$. 

\begin{lem}
\label{lem_nonempty}
For $g \ge 2$, $\widehat{\mathcal{M}}^+ \, \cap \, \mathrm{Mod}(\varSigma_g) \ne \emptyset$. 
\end{lem}

\noindent
We recall 

\begin{lem}[Proposition~3.5 in \cite{KT1}]
\label{lem_criterionOri}
 The mapping class $\phi_{(x,y,z)}$ associated to a primitive fibered class $(x,y,z) \in int(C_{\Delta})$ has orientable invariant foliations 
if and only if $x$ and $y$ are even and $z$ is odd. 
\end{lem}

\begin{proof}[Proof of Lemma~\ref{lem_nonempty}]
For $g \ge 2$ even, let $u_g= (g,g,-1) \in int(C_{\Delta})$. 
For $g \ge 3$ odd, let  $u_g= (g+1,g+1,1) \in int(C_{\Delta})$. 
The minimal representative $F_{u_g}$ is a genus $g$ surface with $3$ boundary components. 
By Lemma~\ref{lem_criterionOri}, we see that $\phi_{u_g} \in \mathcal{M}$, 
since $\phi_{u_g}$ has orientable invariant foliations. 
In particular $\widehat{\phi}_{u_g} \in \widehat{\mathcal{M}}^+ \, \cap \, \mathrm{Mod}(\varSigma_g)$. 
\end{proof}

\noindent
From the proof of Lemma~\ref{lem_nonempty}, 
we have  $\delta_g^+ \le \lambda_{(g,g,-1)}$ for $g$ even, and 
$\lambda_{(g,g,-1)}$ is the largest real root of 
$f_{(g,g,-1)}(t)= t^{2g+1} - 2t^{g+1}- 2t^g+1$. 
Thus $\widehat{\phi}_{u_g} \in  \widehat{\mathcal{M}}^+ \cap  \mathrm{Mod}(\varSigma_g)$ 
have the same dilatation as examples by Minakawa and Hironaka-Kin, see \cite{Minakawa,HK}. 

Next, we recall upper bounds on $\delta_g^+$ when $g \not\equiv 0 \pmod{6}$ by Hironaka, Aaber-Dunfield and Kin-Takasawa 
which are sharper than the bound $\delta_g^+ \le \lambda_{(g,g,-1)}$. 
To do this, let us define $\widehat{\delta}_g^+(r)$ for $r \in {\mathcal Hyp} $. 
Let $\widehat{\delta}^+_g(r)$  be the minimum among dilatations of elements 
$\widehat{\phi}_a \in \widehat{\mathcal{M}}^+ \cap \mathrm{Mod}(\varSigma_g)$ such that 
$a \in S_{\beta}(r)$ and  $\phi_a \in \mathcal{M}$. 
We set $\widehat{\delta}_g^+(r) = \infty$ when  there exist no such elements. 
Clearly $\delta_g^+ \le \widehat{\delta}^+_g \le \widehat{\delta}^+_g(r)$.

\begin{lem}
\label{lem_compare_ori}
\ 
\begin{enumerate}
\item[(1)] 
$\widehat{\delta}^+_g(\tfrac{3}{-2}) = \widehat{\delta}^+_g(\tfrac{1}{-2})= \widehat{\delta}^+_g(2)= \infty$ 
if $g\equiv 0 \pmod{6}$. 

\item[(2)] 
$ \widehat{\delta}^+_g(\tfrac{1}{-2}) = \lambda_{(g,1)}$, 
$\widehat{\delta}^+_g(\tfrac{3}{-2}) =  \widehat{\delta}^+_g(2)= \infty$ 
if $g\equiv 2,4 \pmod{6}$. 

\item[(3)] 
$\min \{\widehat{\delta}^+_g(\tfrac{3}{-2}), \widehat{\delta}^+_g(\tfrac{1}{-2}), \widehat{\delta}^+_g(2)\} 
= \widehat{\delta}^+_g(\tfrac{3}{-2}) = \lambda_{(g+2,4)}$ 
if $g \equiv 1,5 \pmod{10}$. 

\item[(4)] 
$\min \{\widehat{\delta}^+_g(\tfrac{3}{-2}), \widehat{\delta}^+_g(\tfrac{1}{-2}), \widehat{\delta}^+_g(2)\} 
= \widehat{\delta}^+_g(\tfrac{1}{-2}) = \lambda_{(g+1,3)}$ 
if $g \equiv 3 \pmod{10}$ and $g \equiv 1,3 \pmod {6}$. 

\item[(5)] 
$\min \{\widehat{\delta}^+_g(\tfrac{3}{-2}), \widehat{\delta}^+_g(\tfrac{1}{-2}), \widehat{\delta}^+_g(2)\} 
= \widehat{\delta}^+_g(\tfrac{1}{-2}) = \lambda_{(g+1,1)}$ 
if $g \equiv 3 \pmod{10}$ and $g \equiv 5 \pmod {6}$. 

\item[(6)] 
$\min \{\widehat{\delta}^+_g(\tfrac{3}{-2}), \widehat{\delta}^+_g(\tfrac{1}{-2}), \widehat{\delta}^+_g(2)\} 
= \widehat{\delta}^+_g(\tfrac{3}{-2}) = \lambda_{(g+2,2)}$ 
if $g \equiv 7,9 \pmod{10}$. 
\end{enumerate}
\end{lem} 

\begin{proof} 
We have $\widehat{\delta}^+_g(\tfrac{3}{-2})= \widehat{\delta}_g^+(2)= \infty$ if $g$ is even 
by \cite[Corollary~4.5, Lemma~4.11]{KT1}. 
As a consequence of \cite{Hironaka}, $\widehat{\delta}^+_g(\tfrac{1}{-2})=  \infty$ if $g\equiv 0 \pmod{6}$. 
By \cite[Theorem~1.5]{Hironaka}, 
$\widehat{\delta}^+_g(\tfrac{1}{-2}) = \lambda_{(g,1)} $ if $g\equiv 2,4 \pmod{6}$. 
This completes the proofs of (1) and (2). 

By using the same argument as in Proposition~\ref{prop_least-element}, 
one can prove that 
$$\min \{\widehat{\delta}^+_g(\tfrac{3}{-2}), \widehat{\delta}^+_g(\tfrac{1}{-2}), \widehat{\delta}^+_g(2)\} <  \widehat{\delta}^+_g(2) \hspace{2mm} 
\mbox{if\ }g\ \mbox{is\ odd}.$$ 
This together with \cite[Propositions~4.23, 4.34]{KT1} implies the claims  (3)--(6). 
\end{proof}

\noindent
If we fix $\ell>0$, then 
$k  \log \lambda_{(k,\ell)} $ and $k \log \lambda_{(k,k,-1)}$ go to 
$\log  (\tfrac{3 + \sqrt{5}}{2})$ and $\log (2 + \sqrt{3})$ respectively if $k $ goes to $\infty$. 
Thus the upper bound on $\delta_g^+$ in Lemma~\ref{lem_compare_ori} when $g\not\equiv 0 \pmod{6}$ is sharper than 
the bound  $\delta_g^+ \le \lambda_{(g,g,-1)}$ for large $g$.

Putting all things together, we have the following upper bound on $\delta_g^+$ after \cite{Hironaka,AD,KT1,Minakawa,HK}. 
\begin{itemize}
\item 
$\delta_g^+ \le \lambda_{(g,g,-1)}$ when $g \equiv 0 \pmod{6}$, 
\item 
$\delta_g^+ \le \lambda_{(g,1)}$ when $g \equiv 2,4 \pmod{6}$, 
\end{itemize}
and for $g$ odd, 
\begin{itemize}
\item 
$\delta_g^+ \le \lambda_{(g+2,4)}$ when $g \equiv 1,5 \pmod{10}$, 

\item 
$\delta_g^+ \le \lambda_{(g+1,3)}$ when $g \equiv 3 \pmod{10}$ and $g \equiv 1,3 \pmod {6}$, 

\item 
$\delta_g^+ \le \lambda_{(g+1,1)}$ when $g \equiv 3 \pmod{10}$ and $g \equiv 5  \pmod {6}$, 

\item
$\delta_g^+ \le \lambda_{(g+2,2)}$ when $g \equiv 7,9 \pmod{10}$. 
\end{itemize}

\begin{proof}[Proof of Theorems~\ref{thm_smallest_dil_ori}]
The proof of the claim (1) is similar to that of  Theorem~\ref{thm_main}(1). 
The claims (2),(3) hold by Lemma~\ref{lem_compare_ori}(2)--(6) and 
by the same argument as in the proof of Theorem~\ref{thm_main}(2).  
\end{proof}

We are ready to prove Theorem~\ref{thm_LTQuestion}. 

\begin{proof}[Proof of Theorem~\ref{thm_LTQuestion}.] 
Let $a_g^+$ be a primitive fibered class of  $H_2(N, \partial N)$  
such that  $\phi_{a_g^+} \in \mathcal{M}$  and  $\widehat{\delta}^+_g$  is achieved by  
$\widehat{\phi}_{a^+_g} \in \widehat{\mathcal{M}}^+ \cap \mathrm{Mod}(\varSigma_g)$. 
We prove the claim (2) first. 
\medskip

\noindent
(2) Suppose that $g\equiv 2,4 \pmod{6}$. 
By Lemma~\ref{lem_compare_ori}(2), 
we have  $ \widehat{\delta}^+_g(\tfrac{1}{-2}) = \lambda_{(g,1)}$.  
By Theorem~\ref{thm_smallest_dil_ori}(2), 
the fibered class $a_{g}^+$ must have a  boundary slope $ \tfrac{1}{-2}$ for large $g$. 
Thus  $\widehat{\delta}_g^+ =  \widehat{\delta}^+_g(\tfrac{1}{-2}) = \lambda_{(g,1)} $. 
\medskip

\noindent
(1) 
Suppose that $g\equiv 0 \pmod{6}$. 
By Lemma~\ref{lem_compare_ori}(1), 
no boundary slope of $a_g^+$ is an element of $\{-4, \frac{3}{-2}, \frac{1}{-2}, 2 \}$. 
From the proof of Theorem~\ref{thm_main}(1), 
we know that  $\mathrm{Ent}(\widehat{\phi}_{a_g^+}) > 1.97475 - \epsilon $ for any $g\equiv 0 \pmod{6} $ but finitely many exceptions. 
Thus
$$\mathrm{Ent}(\widehat{\phi}_{a_g^+}) = |\chi(\varSigma_g)| \log \widehat{\delta}_g^+> 1.97475 - \epsilon > 2\log( \tfrac{3+ \sqrt{5}}{2})$$ 
for large $g$ such that $g\equiv 0 \pmod{6}$. 
Since $\displaystyle \lim_{g \to \infty} |\chi(\varSigma_g)| \log \lambda_{(g,1)} =  2\log( \tfrac{3+ \sqrt{5}}{2})$, 
the inequality $\widehat{\delta}^+_g > \lambda_{(g,1)}$ holds for such a large $g$. 
\end{proof}

We now prove Theorem~\ref{thm_6+12i} 
which improves  the  upper bound  (\ref{equation_MinakawaHK2}) in Section~\ref{subsection_MainResults}. 
To do this, 
it suffices to prove the following two lemmas.

\begin{lem}
\label{lem_minEnt4}
$\min \mathrm{Ent}(N(-6), \Omega_S)= \min \mathrm{Ent}(N(4), \Omega_S)= 4 \log \delta(D_5)$.
\end{lem}

\begin{proof} 
Both minimal dilatations $\delta_2$ and $\delta(D_5)$ are the largest real root of $t^5 -2t^3-2t^2+1$, see \cite{CH, HS}. 
By  Lemma~\ref{lem_MNE-S}, 
$$\min \mathrm{Ent}(N(4), \Omega_S)=\tfrac{4}{5} \log \lambda_{[\tfrac{2}{5}, \tfrac{2}{5}]}= 4 \log \lambda_{(2,2,-1)}.$$ 
Since $f_{(2,2,-1)}(t)= t^5 -2t^3-2t^2+1$, we have the identities 
$\lambda_{(2,2,-1)}= \delta_2 = \delta(D_5).$ 
By Theorem~\ref{thm_entropy_equiv}(1), it follows that 
$(N(4), \Omega_S)  \underset{\mathrm{ent}}{\sim}  (N(-6), \Omega_S)  $. 
Hence 
$$\min \mathrm{Ent}(N(4), \Omega_S)= \min \mathrm{Ent}(N(-6), \Omega_S).$$
This completes the proof. 
\end{proof}

\begin{lem}
\label{lem_6+12i}
For each $i \ge 0$, there exists a $\Sigma_{6+12i}$-bundle over the circle 
which satisfies the following. 
It is obtained from $N(4)$ by Dehn filling both cusps along boundary slopes of a fiber of $N(4)$, 
and the monodromy $\Phi_i: \Sigma_{6+12i} \rightarrow \Sigma_{6+12i}$ of the fibration has orientable invariant foliations. 
 Moreover 
$$ \min \mathrm{Ent}(N(4), \Omega_S)=\lim_{i \to \infty} \mathrm{Ent}(\Phi_i)  .$$
\end{lem}

\begin{proof}
Consider a primitive fibered class  
$$a_q =(4q+8,4q+4, -2q-3) \in int(C_{\Delta}) \cap S_{\gamma}(4) \hspace{2mm}\mbox{for} \ q \ge 0.$$
 Lemma~\ref{lem_criterionOri} tells us that 
the monodromy  of the fibration on $N$ associated to $a_q$ has orientable invariant foliations. 
In particular, $\phi_{a_q} \in \mathcal{M}$ and $\widehat{\phi}_{a_q} \in \widehat{\mathcal{M}}^+$. 
 Now let $q = 3i$ for  $i \ge 0$. 
 Then  the numbers of the boundary components of $F_{a_q}$ lying on $T_{\alpha}$, $T_{\beta}$ and $T_{\gamma}$ are given by 
 $1$,  $1$, $2q+3$ respectively (see Lemma~\ref{lem_topological-type}), and  
 the genus of $F_{a_q}$ is equal to $6+ 12i$. 
 
The ray of $\overline{a_q} \in H_2(N(4), \partial N((4)))$ goes to the ray of $\overline{(2,2,-1)} \in H_2(N(4), \partial N((4)))$ 
 as $q$ goes to $\infty$. 
Hence 
$$\mathrm{Ent}_4( \overline{(2,2,-1)}) = \lim_{q \to \infty} \mathrm{Ent}_4(\overline{a_q}).$$
On the other hand by Lemma~\ref{lem_minEnt4}, we have 
 $$\min \mathrm{Ent}(N(4), \Omega_S)= 4 \log \lambda_{(2,2,-1)}=\mathrm{Ent}_4( \overline{(2,2,-1)}).$$
 Since the number of the boundary components of $F_{\overline{a_q} }$ is bounded 
 (in fact, it is exactly $2$), 
it follows that 
 $$\mathrm{Ent}_4( \overline{(2,2,-1)}) = \lim_{q \to \infty} \mathrm{Ent}_4(\overline{a_q})
 = \lim_{q \to \infty} \mathrm{Ent}(\widehat{\phi}_{a_q} ).$$
This completes the proof. 
\end{proof}

%残す
%\begin{proof}
%We set a primitive fibered class  $a_q =2 \mathfrak{a}_4+ (2q+1) \mathfrak{b}_4 = (4q+8, -2q-3, 4q+4) \in S_{\beta}(4) $ 
%for $q \ge 0$. 
%The monodromy  of the fibration on $N$ associated to $a_q$ has orientable invariant foliations. 
%This follows from Lemma~\ref{lem_criterionOri} by using the class 
%$\sigma(a_q)= (4q+4,4q+8, -2q-3) \in int(C_{\Delta}) \cap S_{\gamma}(4)$. 
%In particular, $\phi_{a_q} \in \mathcal{M}$ and $\widehat{\phi}_{a_q} \in \widehat{\mathcal{M}}^+$. 
% 
% Now let $q = 3i$ for  $i \ge 0$. 
% Then  the numbers of the boundary components of $F_{a_q}$ lying on $T_{\alpha}$, $T_{\beta}$ and $T_{\gamma}$ are given by 
% $1$,  $2q+3$ and $1$ respectively (see Lemma~\ref{lem_topological-type}), and  
% the genus of $F_{a_q}$ is equal to $6+ 12i$. 
%
%  
% The ray of $\overline{a_q}$ goes to the ray of $\overline{\mathfrak{b_4}}$ 
% as $q$ goes to infinity. 
% Since $\min \mathrm{Ent}(N(4), \Omega_S)= \mathrm{Ent}( \overline{\mathfrak{b}_4})$, 
% we obtain the desired claim. 
%\end{proof}

\begin{proof}[Proof of Theorem~\ref{thm_bound6+12i}]
In the proof of Lemma~\ref{lem_6+12i}, 
we proved that  
for $g \equiv 6 \pmod{12}$ ($\Longleftrightarrow g \equiv 6,18,30,42,54,66,78 \pmod{84}$), 
$$\delta_g^+ \le  \lambda(g+2,  g-2,  -\tfrac{g}{2}).$$
Let us prove that monodromies of the fibrations on  $N(-6)$ give sharper upper bounds on $\delta_g^+$ for some $g$.  
Let 
$$a_q'= (6q+4, 6q+2, 2q+1) \in int(C_{\Delta}) \cap S_{\gamma}(-6) \ \hspace{2mm}\mbox{for}\ q \ge 1.$$
The monodromy  of the fibration on $N$ associated to $a'_q$ has orientable invariant foliations by Lemma~\ref{lem_criterionOri}. 
Hence $\phi_{a'_q} \in \mathcal{M}$ and $\widehat{\phi}_{a'_q} \in \widehat{\mathcal{M}}^+$. 
For  $g \equiv 6,30,42,54,78 \pmod{84}$ and $g >0$, we set $q= \tfrac{g-2}{4}$. 
Then  the numbers of the boundary components of $F_{a'_q}$ lying on $T_{\alpha}$, $T_{\beta}$ and $T_{\gamma}$ are given by 
 $1$,  $1$, $2q+1$ respectively. 
 The genus of $F_{a'_q}$ is equal to $g$. 
 Thus we have 
$$\delta_g^+ \le \lambda(\tfrac{3g}{2}+1, \tfrac{3g}{2}-1,  \tfrac{g}{2}).$$
To check that this bound is sharper than the one above, 
we now prove the inequality 
$$\lambda(\tfrac{3g}{2}+1, \tfrac{3g}{2}-1,  \tfrac{g}{2}) <  \lambda(g+2,  g-2,  -\tfrac{g}{2}).$$
Recall that 
$(N(4), \Omega_S)$ and $(N(-6), \Omega_S)$ are entropy equivalent, and 
$\min \mathrm{Ent}$ is attained by  $\overline{(2,2,-1)}$ for $N(4)$ (resp. $\overline{(3,3,1)}$ for $N(-6)$). 
We note that 
the ray of $\overline{a_q'} \in H_2(N(-6), \partial N((-6)))$ goes to the ray of $\overline{(3,3,1)} \in H_2(N(-6), \partial N((-6)))$ 
 as $q$ goes to $\infty$. 
We have the identity on the Thurston norm: 
$$\|\overline{(\tfrac{3g}{2}+1, \tfrac{3g}{2}-1,  \tfrac{g}{2})}\|_{-6}
= \|\overline{(g+2,  g-2,  -\tfrac{g}{2})}\|_4.$$
One can check that 
the ray of   $\overline{(\tfrac{3g}{2}+1, \tfrac{3g}{2}-1,  \tfrac{g}{2})}$ is closer to the minimal ray than 
the ray  of  $\overline{(g+2,  g-2,  -\tfrac{g}{2})}$. 
Because of the strict concavity of  $\frac{1}{\mathrm{ent}}$, 
we have the desired inequality. 
\end{proof}

Let us turn to the case $g \equiv 0 \pmod{12}$. 
We have not obtained an explicit upper bound on $g \log \delta_g^+$ or $\delta_g^+$ 
as in Theorem~\ref{thm_6+12i} or \ref{thm_bound6+12i} in this case. 
However for many  such $g$, 
we improve the previous bound $\delta_g^+ \le \lambda_{(g,g,-1)}$, see Table~\ref{table_6i}. 
We note that our bound $\delta_{12}^+ \le \lambda_{(12,20,3)}$ in Table~\ref{table_6i} is given by the example 
which occurs as the monodromy of the fibration on a manifold obtained from $N(\tfrac{3}{-4})$ by Dehn filling both cusps. 
 In the left column of Table~\ref{table_6i}, 
 other upper bounds on $\delta_g^+$ when $g \equiv 0 \pmod{12}$ are given by examples 
which occur as the monodromies of  fibrations on manifolds obtained from $N(\tfrac{5}{-4})$. 
By our computer experiments, 
it seems that $\widehat{\delta}_g^+$  is realized by the example obtained from $N(\tfrac{5}{-4})$ by Dehn filling both cusps  
for any  $g \equiv 0 \pmod{12}$ and $g >12$. 
We ask the following.

\begin{ques}
\label{ques_genus12}
Does there exist a primitive fibered class 
$b_i \in int(C_{\Delta}) \cap S_{\beta}(\tfrac{5}{-4})$ for large $i$ 
which enjoys the following? 
\begin{itemize}
\item 
The minimal representative $F_{b_i}$ has genus $12i$, and 
$\phi_{b_i}$ has orientable invariant foliations. 

\item 
$\min \mathrm{Ent}(N(\tfrac{5}{-4})) = \displaystyle \lim_{i \to \infty} \mathrm{Ent}_{5/(-4)}(\overline{b_i})$.
\end{itemize}
\end{ques}

\begin{prop}
If Question~\ref{ques_genus12} is true, then 
$$\displaystyle \limsup_{\substack{g \equiv 0 \pmod{12} \\ g \to \infty}}  g  \log \delta^+_g  \le \tfrac{1}{2} \min \mathrm{Ent}(N(\tfrac{5}{-4})) 
< 1.1466.$$
\end{prop}

\begin{proof}
The existence of primitive integral classes $b_i$ implies the left inequality. 
To see the right inequality, 
we take $a= (292,300,83) \in int(C_{\Delta}) \cap S_{\beta}(\tfrac{5}{-4})$. 
The Thurston norm of $\overline{a} $ equals $\|a\| - \gcd(300,375)= 434$. 
Thus 
$$\min \mathrm{Ent}(N(\tfrac{5}{-4})) \le \mathrm{Ent}_{5/(-4)}(\overline{a}) =434 \log\lambda(a) \approx 2.2930.$$
\end{proof}

\subsection{Proof of Theorem~\ref{thm_tsai}} 
\label{subsection_WhiteheadLink}

This subsection concerns the monodromies of fibrations on the Whitehead link exterior $N(1)$.

\begin{prop}
\label{prop_whitehead}
An $S$-face of $N(1)$  and an $A$-face of $N(1)$ are entropy equivalent. 
\end{prop}

\begin{proof} 
For each $k, \ell \in {\Bbb N}$, the class $k \overline{\mathfrak{a}_1} + \ell \overline{\mathfrak{b}_1}$ 
(resp. $k \overline{\mathfrak{a}_1} - \ell \overline{\mathfrak{b}_1}$) is an element of $int(C_{\Omega_S})$ (resp. $int(C_{\Omega_A})$). 
(See Figure~\ref{fig_ExBall_Nr}.) 
Further $k \overline{\mathfrak{a}_1} \pm \ell \overline{\mathfrak{b}_1} $ have the same Thurston norm. 
Thus it suffices to show that 
$k \overline{\mathfrak{a}_1} \pm \ell \overline{\mathfrak{b}_1} $ have the same entropy. 
Figure~\ref{fig_whitehead} illustrates the projections of the Whitehead link. 
The minimal representatives of $\overline{\mathfrak{a}_1}$ and $\overline{\mathfrak{b}_1}$ are depicted as in Figure~\ref{fig_whitehead}(a). 
One can check that all three oriented links  in this figure are isotopic in $S^3$ to each other. 
In particular  the two oriented links (b) and (c)  are isotopic fixing the trivial component. 
This implies that there exists an involution  $f: N(1) \rightarrow N(1)$. 
This involution induces an   isomorphism $f_*: H_2(N(1), \partial N(1)) \rightarrow H_2(N(1), \partial N(1))$ 
which sends $\overline{\mathfrak{a}_1}$ to itself and $\overline{\mathfrak{b}_1}$ to $-\overline{\mathfrak{b}_1}$. 
Because $f_*$ is induced by the involution on the manifold $N(1)$, 
the entropy of $k \overline{\mathfrak{a}_1} + \ell \overline{\mathfrak{b}_1}$ must be equal to that of 
$k \overline{\mathfrak{a}_1} - \ell \overline{\mathfrak{b}_1}$. 
\end{proof} 

\noindent
By Lemma~\ref{lem_MNE-S}, one sees that 
$\min \mathrm{Ent}(N(1), \Omega_S)= 2 \log \delta(D_4)$. 
This together with Proposition~\ref{prop_whitehead} leads to 
$\min \mathrm{Ent}(N(1), \Omega_A)= 2 \log \delta(D_4)$. 
Thus we obtain

\begin{cor}
\label{cor_MNE-Whitehead}
$\min \mathrm{Ent}(N(1))   = 2 \log \delta(D_4) \approx 1.6628$.  
\end{cor}

\begin{figure}
\begin{center}
\includegraphics[width=4in]{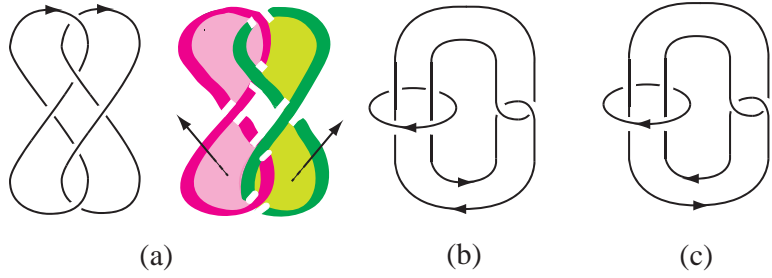}
\caption{projections of the Whitehead link. 
[the minimal representatives of  $\overline{\mathfrak{a}_1}$ and $\overline{\mathfrak{b}_1}$ are illustrated in (a).] }
\label{fig_whitehead}
\end{center}
\end{figure}

%Proposition~\ref{prop_whitehead} says that 
%for the study of monodromies of fibrations on $N(1)$, 
%it is enough to deal with fibers whose homology classes are in the cone over an $S$-face. 

The following lemma is easy to verify by using Lemma~\ref{lem_topological-type}. 
%(It is well-known that the genus of each fiber of $N(1)$ equals $1$.)  

\begin{lem} 
\label{lem_PuncturedTorus}
The genus of each fiber of $N(1)$ equals $1$. 
More precisely, for coprime integers $k, \ell \in {\Bbb N}$, 
the minimal representative  of $k \overline{\mathfrak{a}_1} + \ell \overline{\mathfrak{b}_1}$ 
is a $(k+\ell)$-holed  torus.  
\end{lem}

\begin{rem}
\label{rem_Whitehead}
%Here is a comment on the monodromies of fibrations on $N(1)$. 
For $k$ and $\ell$ as in Lemma~\ref{lem_PuncturedTorus}, 
the stable foliation of the monodromy $\Phi_{k \mathfrak{a}_1 + \ell \mathfrak{b}_1}$ of the fibration on $N$ 
associated to $k \mathfrak{a}_1 + \ell \mathfrak{b}_1$ 
has the following property. 
Each boundary component of the fiber $F_{k \mathfrak{a}_1 + \ell \mathfrak{b}_1}$  
lying on the torus specified by $\alpha$, $\beta$ and $\gamma$ 
has a $1$ prong, $3$ prongs and a $1$ prong respectively. 
Hence $\phi_{k \mathfrak{a}_1 + \ell \mathfrak{b}_1} \notin \mathcal{M}$. 
\end{rem}

For $n \ge 2$, 
let $\mathcal{W}_n \subset H_2(N(1), \partial N(1); {\Bbb Z})$  be the set of primitive fibered classes 
whose minimal representatives are $n$-holed tori. 

\begin{prop}
\label{prop_MiniDil_W}
The following class achieves the minimal dilatation among elements of $\mathcal{W}_n$. 

\begin{enumerate}
\item[(1)] 
$\overline{\mathfrak{a}_1} + \overline{\mathfrak{b}_1}$ when $n= 2$. 
Its  dilatation equals the largest real root of 
$$f_{(1,1,-2)}(t)= t^4 -2t^3 - 2t+1.$$

\item[(2)] 
$k \overline{\mathfrak{a}_1} + (k-1) \overline{\mathfrak{b}_1}$ when $n= 2k-1$ for $k \ge 2$. 
Its  dilatation equals the largest  real root of 
$$f_{(k, k-1, -2k+1)}(t) = t^{4k-2} - t^{3k-1} - t^{3k-2} - t^k - t^{k-1}+1.$$ 

\item[(3)]
$(2k+1) \overline{\mathfrak{a}_1} + (2k-1) \overline{\mathfrak{b}_1}$ when $n= 4k$ for $k \ge 1$. 
Its  dilatation equals the largest  real root of  
$$f_{(2k+1, 2k-1, -4k)}(t) = t^{8k} -t^{6k+1}-t^{6k-1} - t^{2k+1} - t^{2k-1}+1.$$

 \item[(4)]
 $(2k+3) \overline{\mathfrak{a}_1} + (2k-1) \overline{\mathfrak{b}_1}$ when $n= 4k+2$ for $k \ge 1$. 
 Its  dilatation equals the largest  real root of   
 $$f_{(2k+3, 2k-1, -4k-2)}(t) = t^{8k+4} -t^{6k+5}-t^{6k+1} - t^{2k+3} - t^{2k-1}+1. $$
\end{enumerate}
\end{prop}

\begin{proof} 
Proposition~\ref{prop_whitehead} says that 
for the study of monodromies of fibrations on $N(1)$, 
it is enough to deal with fibers whose homology classes are in the cone over an $S$-face. 
From the proof of Lemma~\ref{lem_MNE-S}, 
the center of $\Omega_S$ achieves $\min \mathrm{Ent}(N(1), \Omega_S)$. 
Then the proposition holds from  the strict concavity of the function 
$\tfrac{1}{\mathrm{ent}_1}= \tfrac{1}{\log \lambda_1}: int(C_{\Omega_S}) \rightarrow {\Bbb R}$ together with 
Lemma~\ref{lem_PuncturedTorus}. 
\end{proof}

\begin{lem}
\label{lem_direction_W}
Suppose that $ \tfrac{\ell}{k}$ goes to $1$ as  both $k$ and $\ell$ go to $\infty$. 
Then 
$ \mathrm{Ent}_1(k \overline{\mathfrak{a}_1} + \ell \overline{\mathfrak{b}_1})) $ goes to $\min \mathrm{Ent}(N(1)) =2 \log \delta(D_4)$ 
as both $k$ and $\ell$ go to $\infty$. 
\end{lem}

\begin{proof} 
$\min \mathrm{Ent}(N(1), \Omega_S)$ is achieved by the  center of $\Omega_S$. 
This leads to the lemma. 
\end{proof}

\begin{proof}[Proof of Theorem~\ref{thm_tsai}.] 
See Proposition~\ref{prop_MiniDil_W} and Lemma~\ref{lem_direction_W}. 
\end{proof}

\section{1-cusped manifolds with small volume} 
\label{section_Cusped-manifolds}

The magic manifold plays a central role not only for the minimizing problem on   dilatations but also 
for the minimizing problem on volumes of hyperbolic $3$-manifolds. 
It was proved by Agol that the smallest volume among orientable $2$-cusped hyperbolic  $3$-manifold is achieved by 
either the Whitehead link exterior  $N(1)$ or 
the Whitehead sister link exterior $N(\tfrac{3}{-2})$, see \cite{Agol}. 
Gabai, Meyerhoff and Milley proved that $1$-cusped orientable hyperbolic $3$-manifolds 
with  volume at most $  2.848$ are obtained from $N$ by 
Dehn filling two cusps, and they identified these $1$-cusped manifolds, which we recall in Theorem~\ref{thm_GMM2}. 
In the end of this section, 
we compute the normalized entropy of the monodromy of the fibration on each of them. 

First we recall

\begin{thm}[Theorem 1.1 in \cite{GMM}]
\label{thm_GMM}
Let $M$ be a $1$-cusped orientable hyperbolic $3$-manifold whose volume is at most $2.848$. 
Then $M$ can be obtained from $M'$ by Dehn filling all but one of the cusps, where 
$M'$ is one of the  $21$ manifolds in the SnapPea census: 
$m125$, $m129$, $m202$, $m203$, $m292$, $m295$, $m328$, $m329$, $m359$, $m366$, $m367$, $m391$, 
$m412$, $s596$, $s647$, $s774$, $s776$, $s780$, $s785$, $s898$, $s959$.
\end{thm}

\noindent
$s776$ is homeomorphic to the magic manifold $N$. 
All manifolds listed above other than $s776$   have exactly $2$ cusps.  
The $12$ manifolds $m125$, $m129, \cdots, m391$ are obtained from $s776$ by Dehn filling a cusp. 
We compute the quantities $\min \mathrm{Ent}$ etc. for these manifolds  by using results in this paper, 
see Table~\ref{table_2cusp_census_MNE}. 
(One can check that the first column in Table~\ref{table_2cusp_census_MNE} by using SnapPy~\cite{CDW}.)

\begin{table}[hbtp]
\caption{$\min \mathrm{Ent}$ for some fibered $3$-manifolds in Theorem~\ref{thm_GMM}.} 
\label{table_2cusp_census_MNE}
\begin{center}
$\left|\begin{array}{c|c|c|c}
\hline
\mbox{manifold}\ M  & \min \mathrm{Ent}(M, \Omega_S) &  \min \mathrm{Ent}(M, \Omega_A) & \min \mathrm{Ent}(M) \\
\hline
\hline
m125 \simeq N(\tfrac{3}{-2})   & none & 2 \log( \tfrac{3+ \sqrt{5}}{2}) & 2 \log( \tfrac{3+ \sqrt{5}}{2})   \\
 \hline 
m129 \simeq N(1) & 2 \log \lambda_{(1,1,-2)}  \approx 1.6628 & 2 \log \lambda_{(1,1,-2)} &  2 \log \lambda_{(1,1,-2)}   \\
 \hline 
m202 \simeq N(\tfrac{5}{-2})  & ?   &  4 \log \lambda_{(4,2,1)} \approx   2.5318 &   \\
 \hline 
m203 \simeq  N(\tfrac{1}{-2}) & none   &2 \log( \tfrac{3+ \sqrt{5}}{2})  & 2 \log( \tfrac{3+ \sqrt{5}}{2})\\
 \hline 
m292 \simeq N(-5)          &       6 \log \lambda_{(5,5,2)}  \approx 2.0761      & ?         &   \\
  \hline 
m295 \simeq N(2)                &   2 \log( \tfrac{3+ \sqrt{5}}{2})                                   & ?         &    \\  
 \hline 
 m328 \simeq N(\tfrac{4}{-3})   &  none                                                  &  \clubsuit       &  \clubsuit    \\  
 \hline 
m329 \simeq  N(\tfrac{5}{-3})  &   none                                          & \spadesuit       &  \spadesuit     \\  
     \hline 
m359 \simeq  N(\tfrac{2}{-3})  &  none                                                          &  \clubsuit      &  \clubsuit    \\  
     \hline 
 m366 \simeq   N(\tfrac{7}{-2}) & 8 \log \lambda_{(7,7,4)} \approx 2.4181   &   \diamondsuit  & 8 \log \lambda_{(7,7,4)}     \\  
     \hline 
m367 \simeq   N(\tfrac{1}{2})    &      4 \log \lambda_{(1,1,-4)} \approx 2.9314     & 4 \log \lambda_{(4,2,1)} \approx 2.5318         &4 \log \lambda_{(4,2,1)}    \\  
 \hline 
m391 \simeq  N(\tfrac{1}{-3})   &   none                                              &     \spadesuit      &  \spadesuit     \\
     \hline  
\end{array}\right|$
\end{center} 
\end{table}

\begin{rem}
We make comments on Table~\ref{table_2cusp_census_MNE}. 
\begin{enumerate}
\item[(1)] 
The filling slopes $\tfrac{p}{q}$ of $N(\tfrac{p}{q})$ marked ``?" do not enjoy the assumption of Theorem~\ref{thm_key-SA}.

\item[(2)]
By Theorem~\ref{thm_entropy_equiv}, 
$ N( \tfrac{5}{-3}) \underset{\mathrm{ent}}{\sim}  N( \tfrac{1}{-3} )$ and $N( \tfrac{4}{-3}) \underset{\mathrm{ent}}{\sim}  N( \tfrac{2}{-3} )$. 
This together with Lemma~\ref{lem_monotonicity} implies that 
$$ \spadesuit  =
 \min \mathrm{Ent}(N( \tfrac{5}{-3})) = 
\min \mathrm{Ent}(N( \tfrac{1}{-3} )) > 
\min \mathrm{Ent}(N( \tfrac{2}{-3} )) =  
\min \mathrm{Ent}(N( \tfrac{4}{-3} ))= \clubsuit .$$

\item[(3)] 
We know $ \min \mathrm{Ent}(N(\tfrac{7}{-2}), \Omega_S) = 8 \log \lambda_{(7,7,4)} \approx  2.4181$ by Lemma~\ref{lem_MNE-S}, and 
 $$  \diamondsuit = 
 \min \mathrm{Ent}(N(\tfrac{7}{-2}), \Omega_A)> 
 \min \mathrm{Ent}(N(\tfrac{5}{-2}), \Omega_A) \approx 2.5318$$ 
by  Lemmas~\ref{lem_monotonicity} and \ref{lem_bunbo2}(2). 
 Thus we conclude that $ \min \mathrm{Ent}(N(\tfrac{7}{-2}))=8 \log \lambda_{(7,7,4)}$. 
\end{enumerate}
\end{rem}

\begin{table}[hbtp]
\caption{$s776 \simeq N$ and manifolds obtained from $N$ by Dehn filling in Theorem~\ref{thm_GMM}.} 
\label{table_2cusp_census}
\begin{center}
$\left|\begin{array}{c|c}
\hline
m125 \simeq N(\tfrac{3}{-2}) \simeq S^3 \setminus br(T_{6,2}') & m129 \simeq N(1)\\
\hline
m202 \simeq N(\tfrac{5}{-2}) \simeq S^3 \setminus br(T_{8,3}') & m203 \simeq N(\tfrac{1}{-2}) \simeq S^3 \setminus br(T_{4,1}') \\
 \hline 
 m292 \simeq N(-5)  \simeq S^3 \setminus br(T_{7,5}') & m295 \simeq N(2) \\
 \hline
 m328 \simeq N(\tfrac{4}{-3})  \simeq S^3 \setminus br(T_{8,2}')  & m329 \simeq N(\tfrac{5}{-3})  \simeq S^3 \setminus br(T_{9,5}') \\
 \hline 
 m359 \simeq N(\tfrac{2}{-3})  \simeq S^3 \setminus br(T_{6,3}') & m366 \simeq N(\tfrac{7}{-2})  \simeq S^3 \setminus br(T_{10,4}') \\
 \hline 
 m367 \simeq N(\tfrac{1}{2}) & m391 \simeq N(\tfrac{1}{-3})  \simeq S^3 \setminus br(T_{5,1}') \\
 \hline 
 s776 \simeq N \simeq  S^3 \setminus br(T_{6,3}) & \\
  \hline 
   \end{array}\right|$
   \end{center}
   \end{table}

Now we  would like to point out that many manifolds in Table~\ref{table_2cusp_census_MNE} are braided link exteriors. 
To do this, we first recall the definition of the braided link. 
Let $B_n$ be the $n$-braid group, and 
let $\sigma_1$, $\sigma_2, \cdots, \sigma_{n-1} \in B_n$ be the Artin generators of $B_n$, 
see Figure~\ref{fig_artin}. 
The {\it braided link} $br(b)$ of a braid $b$ is the union of the closed braid of $b$ and its axis, 
see Figure~\ref{fig_braidedlink}. 
For example, the link $6_2^2$  is the braided link $br(\sigma_1^{-1} \sigma_2)$, see Figure~\ref{fig_3chain_etc}. 
Let $T_{m,p}$ be the following $m$-braid for $m \ge 3$ and $p \ge 1$: 
$$T_{m,p} = (\sigma_1^2  \sigma_2 \sigma_3 \cdots \sigma_{m-1})^p \sigma_{m-1}^{-2}.$$
For example $T_{6,2} = (\sigma_1^2 \sigma_2 \sigma_3 \sigma_4 \sigma_5)^2 \sigma_{5}^{-2} = 
\sigma_1^2 \sigma_2 \sigma_3 \sigma_4 \sigma_5 \sigma_1^2 \sigma_2 \sigma_3 \sigma_4 \sigma_5^{-1}$, see Figure~\ref{fig_Tmp}.  
Forgetting the 1st strand of  $T_{m,p}$, one obtains the $(m-1)$-braid, call it $T'_{m,p}$.  
%For example, $T_{6,2}'= \sigma_1 \sigma_2 \sigma_3 \sigma_4 \sigma_1 \sigma_2 \sigma_3 \sigma_4^{-1}$.  
In \cite[Corollary~3.2]{KT} it was shown that if $m-1$ and $p$ are coprime, then 
$N \simeq S^3 \setminus br(T_{m,p})$. 
By \cite[Theorem~3.4]{KT}, one sees that 
if $x, y \in {\Bbb N}$ are coprime such that $\tfrac{y}{-x} \in {\mathcal Hyp}$, 
then there exists $p(x,y) \in {\Bbb N}$ such that 
$N(\tfrac{y}{-x}) \simeq S^3 \setminus br(T'_{x+y+1, p(x,y)})$. 
Some manifolds in Table~\ref{table_2cusp_census_MNE} can be described 
as the exterior  of the braided link of the form $S^3 \setminus br(T'_{m,p})$, see Table~\ref{table_2cusp_census}. 
It is intriguing that some  braids appearing in this table reach the minimal dilatations. 
Table~\ref{table_minimal_dil} shows the minimal dilatation $\delta(D_n)$ and  an $n$-braid 
(equivalently an element of $\mathrm{Mod}(D_n)$) 
realizing $\delta(D_n)$. 
Here $b \sim b'$ means that $b$ is conjugate to $b'$.

\begin{figure}[htbp]
\begin{center}
\includegraphics[width=1in]{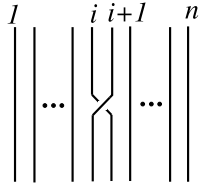}
\caption{braid $\sigma_i \in B_n$.}
\label{fig_artin}
\end{center}
\end{figure}

\begin{figure}[htbp]
\begin{center}
\includegraphics[width=2in]{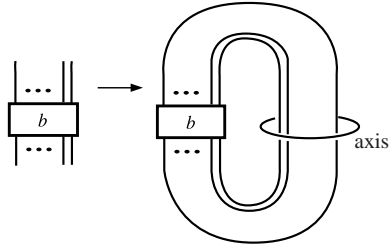}
\caption{braid $b$ $\rightarrow $ braided link $br(b)$.}
\label{fig_braidedlink}
\end{center}
\end{figure}

\begin{figure}[htbp]
\begin{center}
\includegraphics[width=0.7in]{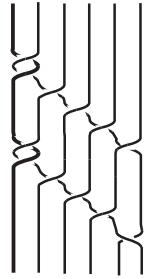}
\caption{braid $T_{6,2}$.}
\label{fig_Tmp}
\end{center}
\end{figure}

\begin{table}[hbtp]
\caption{minimal dilatations of braids.} 
\label{table_minimal_dil}
\begin{center}
$\left|\begin{array}{c|c|c|c}
\hline
n  & \delta(D_n) & n\mbox{-braid realizing}\   \delta(D_n)  & \mbox{reference} \\
\hline
\hline
3  & \tfrac{3+ \sqrt{5}}{2}  \approx 2.6180 & T_{4,1}'  = \sigma_1 \sigma_2^{-1}  & \text{cf. \cite{Matsuoka,Handel}} \\
 \hline 
4  & \lambda_{(3,1,0)} \approx 2.2966  &  T_{5,1}' = \sigma_1 \sigma_2 \sigma_3^{-1} &  \text{\cite{KLS,HS}}  \\
 \hline 
 5  & \lambda_{(2,3,0)} \approx 1.7220 &  T_{6,2}' \sim \sigma_1 \sigma_2 \sigma_3 \sigma_4 \sigma_1 \sigma_2 &  \text{\cite{HS}}   \\
 \hline 
 6  & \lambda_{(3,2,0)} \approx 1.7220  &  T_{6,3} \sim (\sigma_2 \sigma_1 \sigma_2 \sigma_1 (\sigma_1 \sigma_2 \sigma_3 \sigma_4 \sigma_5)^2)^{-1} &  \text{\cite{LT1}}  \\
 \hline 
 7  & \lambda_{(3,4,0)} \approx 1.4655  &  T_{8,2}' \sim \sigma_4^{-2} (\sigma_1 \sigma_2 \sigma_3 \sigma_4 \sigma_5 \sigma_6)^2 &  \text{\cite{LT1}}   \\
  \hline 
 8 & \lambda_{(3,5,0)} \approx 1.4134  & T_{9,5}' \sim \sigma_2^{-1} \sigma_1^{-1} (\sigma_1 \sigma_2 \sigma_3 \sigma_4 \sigma_5 \sigma_6 \sigma_7)^5 & \text{\cite{LT1}}  \\
   \hline
  \end{array}\right|$
\end{center} 
\end{table}

By using Theorem~\ref{thm_GMM}, Gabai, Meyerhoff and Milley proved  

\begin{thm}[Corollary~1.2 in \cite{GMM}] 
\label{thm_GMM2}
Let $M$ be a $1$-cusped orientable hyperbolic $3$-manifold whose volume is at most $ 2.848$. 
Then $M$ is one of $m003$, $m004$, $m006$, $m007$, $m009$, $m010$, $m011$, $m015$, $m016$ and $m017$. 
In particular, every $1$-cusped orientable hyperbolic $3$-manifold whose volume is at most $ 2.848$ 
can be obtained from  the magic manifold by Dehn filling  two cusps. 
\end{thm}

\noindent 
Among the $10$ manifolds in Theorem~\ref{thm_GMM2}, 
$m006$, $m007$, $m015$, $m017$ are non-fibered and the others are fibered \cite{Button,Dunfield}. 
Each of the fibered manifolds in Theorem~\ref{thm_GMM2} has the second Betti number $1$, 
and hence it admits a unique fibration. 
The entropies and normalized entropies of their fibrations are given in Table~\ref{table_1cusp_census}. 
Here are comments on the table. 
\begin{enumerate}

\item[(1)] 
The first column ``manifold" can be checked by using SnapPy \cite{CDW}.

\item[(2)] 
$m003$ and $m004$ are the figure $8$ sister manifold and the figure $8$ knot exterior respectively (\cite[Table A.2]{MP} or \cite{CDW}). 
It is well-known that 
each of $m003$ and $m004$ has the fiber $\varSigma_{1,1}$ 
and the monodromy of its fibration achieves the minimal dilatation $\delta_{1,1}=  \tfrac{3+ \sqrt{5}}{2}$

\item[(3)]
SnapPy tells us that 
the once punctured torus bundles whose monodroies are given by $2 \times 2$ matrices 
$  \left(\begin{array}{cc}3 & 2 \\1 & 1\end{array}\right)$ and  $ \left(\begin{array}{cc}-3 & -2 \\1 & 1\end{array}\right)$ 
are homeomorphic to $m009$ and $m010$ respectively.  
Hence their dilatations equal $2+ \sqrt{3}$ 
which is the largest eigenvalue of $ \left(\begin{array}{cc}| \pm 3| & | \pm 2| \\1 & 1\end{array}\right)$. 

\item[(4)]
The fibered class $a= (13,12,5) \in int(C_{\Delta})$ has the boundary slopes 
$b_{\alpha}(a) = \tfrac{17}{-13}$, $b_{\beta}(a) = \tfrac{3}{-2}$, $b_{\gamma}(a) = \tfrac{5}{-1}$, 
and the genus of $F_a$ is $5$. 
As a consequence of this paper, 
$N(\tfrac{3}{-2}, -5) (\simeq m011)$ has a fiber $\varSigma_{5,1}$ and 
the dilatation of the monodromy of its fibration  equals $\lambda(a)$ 
which is the largest real root $\lambda_{(13,12,5)}   \approx 1.1487 $ of $f_{(13,12,5)}(t)$.

\item[(5)]
The manifold 
$N(\tfrac{3}{-2}, \tfrac{8}{-3})(\simeq m016) $ is homeomorphic to the $(-2,3,7)$-pretzel knot exterior, see \cite[Table~A.4]{MP}. 
Because $a= (18,22,15) \in int(C_{\Delta})$ has the boundary slopes 
$b_{\alpha}(a)= \tfrac{37}{-18}$, $b_{\beta}(a)= \tfrac{3}{-2}$, $b_{\gamma}(a)= \tfrac{8}{-3}$ 
and the genus of $F_a$ equals $5$,  
the fiber of the fibration on $m016$ is $\varSigma_{5,1}$. 
We see that 
the dilatation of the monodromy of its fibration  equals $\lambda(a)$ which is  the largest real root $ \lambda_{(18,22,15)} \approx 1.1762$ of 
$$f_{(18,22,15)}(t)= (t^{11}+1) (t^4 - t^3 + t^2 - t +1)(t^{10} + t^9- t^7-t^6-t^5-t^4-t^3 + t +1).$$ 
Namely it is the largest real root of the last factor. 
Thus the dilatation equals  the so called {\it Lehmer's number}. 
The monodromy of the fibration is described in \cite{Leininger}.

\item[(6)] 
The monodromy of the fibration on $m016 \simeq N(\tfrac{3}{-2}, \tfrac{8}{-3})$ (resp. $m011 \simeq N(\tfrac{3}{-2}, -5)$) 
can extend to the pseudo-Anosov homeomorphism on the closed surface of genus $5$ 
with dilatation $\lambda_{(18,22,15)}$ (resp. $\lambda_{(13,12,5)}$). 
This pseudo-Anosov  is a representative of  $\widehat{\phi}_{(18,22,5)} \in \widehat{\mathcal{M}}$ 
(resp. $\widehat{\phi}_{(13,12,5)} \in \widehat{\mathcal{M}}$). 
On the other hand, 
Lanneau and Thiffeault proved that $\delta_5^+$ equals the Lehmer's number \cite{LT}. 
The pseudo-Anosov representative of $\widehat{\phi}_{(18,22,5)}$ has orientable stable foliation 
(see Lemma~\ref{lem_criterionOri}) and it achieves  $\delta_5^+$. 
The mapping class $\widehat{\phi}_{(13,12,5)}$ was  the example in  \cite{AD,KT1} used to prove $\delta_5< \delta_5^+$. 
\end{enumerate}

\begin{table}[h]
\caption{} 
%\caption{Normalized entropies of fibered $3$-manifolds in Theorem~\ref{thm_GMM2}.} 
\label{table_1cusp_census}
\begin{center}
$\left|\begin{array}{c|c|c|c}
\hline
\mbox{manifold}  & \mbox{fiber}\ \varSigma & \mbox{entropy}\ \log \lambda(\phi_{[\varSigma]}) & \mbox{normalized\  entropy}\ |\chi(\varSigma)| \log \lambda( \phi_{[\varSigma]} ) \\
\hline
\hline
m003 \simeq N(1,-4) & \varSigma_{1,1} &  \log( \tfrac{3+ \sqrt{5}}{2})&  \log( \tfrac{3+ \sqrt{5}}{2}) \approx   0.9624  \\
 \hline 
m004 \simeq N(1,2) & \varSigma_{1,1} &  \log( \tfrac{3+ \sqrt{5}}{2})&  \log( \tfrac{3+ \sqrt{5}}{2}) \approx   0.9624  \\
 \hline 
 m009  \simeq N(1,3) & \varSigma_{1,1} &  \log (2+ \sqrt{3})&  \log (2+ \sqrt{3}) \approx 1.3169   \\
 \hline 
 m010 \simeq N(1, -5)  & \varSigma_{1,1} &  \log (2+ \sqrt{3})&  \log (2+ \sqrt{3}) \approx 1.3169  \\
 \hline 
 m011  \simeq N(\tfrac{3}{-2}, -5)& \varSigma_{5,1} &  \log \lambda_{(13,12,5)}  \approx \log(1.1487) & 9  \log \lambda_{(13,12,5)}  \approx 1.2484   \\
  \hline 
   m016 \simeq N(\tfrac{3}{-2}, \tfrac{8}{-3})& \varSigma_{5,1} &  \log  \lambda_{(18,22,15)} \approx \log(1.1762) & 9  \log \lambda_{(18,22,15)}   \approx 1.4612  \\
   \hline
  \end{array}\right|$
\end{center} 
\end{table}

The number $ \log( \tfrac{3+ \sqrt{5}}{2}) $ is the minimal normalized entropy 
among $1$-cusped hyperbolic fibered $3$-manifolds with volume $\le 2.848$, see Table~\ref{table_1cusp_census}. 
In practice, pseudo-Anosov homeomorphisms with small dilatation occur as monodromies of fibrations 
on  fibered $3$-manifolds with small volume, see \cite{KKT}. 
Thus it is natural to ask

\begin{ques}
\label{ques_m003}
Let $M$ be a $1$-cusped  hyperbolic fibered $3$-manifold with the second Betti number $1$. 
Is it true that the normalized entropy of the monodromy of the fibration on $M$ is greater than or equal to 
$ \log( \tfrac{3+ \sqrt{5}}{2}) $?
Is it  true that the equality is achieved only by either $m003$ or $m004$?
\end{ques}

\section{Remarks}

We find from the next lemma that 
the set $\mathcal{M}$ is very large.

\begin{lem} 
\label{lem_LargeSet}
Let $p \in {\Bbb N}$ and $q \in {\Bbb Z}$ be coprime such that  $\tfrac{p}{q} \in {\mathcal Hyp}$. 
\begin{enumerate}
\item[(1)] 
Suppose that $\tfrac{p}{q} \in (-\infty, -2)$ such that 
 $p+2q \ne 1$ (resp. Suppose that $\tfrac{p}{q} \in (-\infty, -2)$ such that   $|q| \ne 1$).  
 Let $a \in S_{\beta}( \tfrac{p}{q}) $ be a primitive fibered class of $N$ such that 
$\overline{a} \in int(C_{\Omega})$, where $\Omega$ is an $S$-face (resp. $A$-face) of $N( \tfrac{p}{q})$. 
Then $\phi_{a} \in \mathcal{M}$ for any such $a \in S_{\beta}( \tfrac{p}{q})$ but finitely many exceptions. 

\item[(2)]
Suppose that $\tfrac{p}{q} \in (-2, 0)$.  
 Let $a \in S_{\beta}( \tfrac{p}{q}) $ be a primitive fibered class of $N$.  
Then $\phi_{a} \in \mathcal{M}$ for any such $a \in S_{\beta}( \tfrac{p}{q})$ but finitely many exceptions.

\item[(3)] 
Suppose that $\tfrac{p}{q} \in (0, \infty)$ such that  $\tfrac{p}{q} \ne 1$. 
Let $\Omega$ be an $S$-face of $N( \tfrac{p}{q})$ if $|q|=1$ and 
let $\Omega$ be any  face  of $N( \tfrac{p}{q})$ if $|q| \ne 1$.  
Let  $a_n \in S_{\beta}( \tfrac{p}{q}) $ be a primitive fibered class of $N$  such that 
$\overline{a_n} \in int(C_{\Omega})$ for each $n$. 
If  $\overline{a_i} \ne \overline{a_j}$ for $i \ne j$ and 
 $\overline{a_n}$ converges projectively  to a point of $int(\Omega)$ as $n$ goes to $\infty$, then 
$\phi_{a_n} \in \mathcal{M}$ for large $n$. 
 \end{enumerate}
\end{lem}

\begin{proof} 
We first prove the claim (2). 
\medskip
\\
(2) 
In the case $\tfrac{p}{q} \in (-2, 0)$, 
a primitive fibered class $\widehat{a} \in S_{\beta}(\tfrac{p}{q} )$ is in $int(C_{\widehat{\Delta}})$, 
where $\widehat{\Delta}$ is a fibered face $N$ which is either $\Delta$, $\Delta'$, $\Delta_2$ or $\Delta_2'$. 
We note that $\sigma(int(C_{\widehat{\Delta}}) \cap S_{\beta} (\tfrac{p}{q}))= int(C_{\Delta}) \cap S_{\alpha} (\tfrac{p}{q})$ or 
$int(C_{\Delta}) \cap S_{\beta} (\tfrac{p}{q})$. 
By Lemma~\ref{lem_conjugateXY} and Remark~\ref{rem_SymmetrySetXY}, we may assume that 
$\sigma(int(C_{\widehat{\Delta}}) \cap S_{\beta} (\tfrac{p}{q}))= int(C_{\Delta}) \cap S_{\beta} (\tfrac{p}{q})$, and 
it is enough to consider primitive fibered classes $a \in int(C_{\Delta}) \cap S_{\beta} (\tfrac{p}{q})$.

If $\phi_a \not\in \mathcal{M}$, then one of the following two cases  occur. 
\begin{itemize}
\item 
$a \in S_{\alpha} (\tfrac{t}{u}) \cap S_{\beta}(\tfrac{p}{q})$ for some $\tfrac{t}{u}$ such that $|u|=1$. 

\item 
$a \in S_{\beta}(\tfrac{p}{q})  \cap S_{\gamma}(\tfrac{v}{w})$ for some $\tfrac{v}{w}$ such that $v+ 2w=1$. 
\end{itemize}
It suffices to show that  the following two sets are finite: 
Fixing $\tfrac{p}{q}$, 
\begin{eqnarray*}
&\ &\{[x,y] \in int(\Delta)  \cap  S_{\alpha} (\tfrac{t}{u}) \cap S_{\beta}(\tfrac{p}{q}) \ |\  |u|=1\}\ \mbox{and}
\\
&\ & \{[x,y] \in  int(\Delta) \cap S_{\beta}(\tfrac{p}{q})  \cap S_{\gamma}(\tfrac{v}{w})\ |\  v+ 2w = 1\}. 
\end{eqnarray*}
One can prove the first set is finite as follows. 
Recall from Lemma~\ref{lem_boundary} that 
\begin{eqnarray*}
\Delta \cap S_{\alpha}(r)&=&  \{[x,y] \in \Delta\ |\ y=( \tfrac{1+r}{-2})x+ \tfrac{1}{2}\}, \ \mbox{and}
\\
\Delta \cap S_{\beta}(r)&=&  \{[x,y] \in \Delta\ |\ y=( \tfrac{-2}{1+r})x+ \tfrac{1}{1+r}\}. 
\end{eqnarray*}
Suppose that $|u|=1$. 
Then 
$$(\Delta \cap S_{\alpha}(\tfrac{t}{u})) \cap (\Delta \cap S_{\beta}(\tfrac{p}{q})) = \emptyset\ \mbox{for large}\ |\tfrac{t}{u}|.$$
The finiteness of  the second set  can be proved similarly.

The proof of (1) is similar to that of (2). 
\medskip

\noindent
(3) 
Let us consider the case that $\Omega$ is an $S$-face. 
The primitive fibered classes $a_n$ enjoying the assumption of the claim must be 
in the interior of the cone over the fibered face $\Delta_1$ or $\Delta_1'$. 
The images of  $int(C_{\Delta_1}) \cap S_{\beta}(\tfrac{p}{q})$ and $ int(C_{\Delta_1'}) \cap S_{\beta}(\tfrac{p}{q})$ 
under $\sigma$ are the same, and it is the set $int(C_{\Delta}) \cap  S_{\gamma}(\tfrac{p}{q})$. 
Hence it is enough to consider the primitive fibered classes $a_n \in int(C_{\Delta}) \cap S_{\gamma}(\tfrac{p}{q})$ 
which enjoy the assumption of the claim (3).

We consider the following two infinite sets: 
Fixing $\tfrac{p}{q}$, 
\begin{eqnarray*}
&\ & \{[x,y] \in  int(\Delta) \cap S_{\alpha} (\tfrac{t}{u}) \cap S_{\gamma} (\tfrac{p}{q}) \ |\ |u| = 1\}\ \mbox{and}
\\
&\ & \{[x,y] \in int(\Delta) \cap S_{\beta} (\tfrac{t}{u}) \cap S_{\gamma} (\tfrac{p}{q})\ |\ |u| = 1\}.
\end{eqnarray*}
Let us consider the case $\tfrac{p}{q} \in (0,1)$. 
If $\tfrac{t}{u} \in (- \infty, -1]$, then 
$int(\Delta) \cap S_{\alpha} (\tfrac{t}{u}) \cap S_{\gamma} (\tfrac{p}{q}) =  \emptyset$ and 
$int(\Delta) \cap S_{\beta} (\tfrac{t}{u}) \cap S_{\gamma} (\tfrac{p}{q}) = \emptyset$. (See Figure~\ref{fig_appendix}(left).)  
Suppose that $\tfrac{t}{u} \in (-1, \infty)$ such that $|u|=1$. 
(Then $\tfrac{t}{u} = \tfrac{t}{1} \in [0, \infty)$.) 
Fixing $t_0 >0$, the sets
$\{int(\Delta) \cap S_{\alpha} (\tfrac{t}{1}) \cap S_{\gamma} (\tfrac{p}{q}) \ |\ t \in {\Bbb N},\ t < t_0\}$ and 
$\{ int(\Delta) \cap S_{\beta} (\tfrac{t}{1}) \cap S_{\gamma} (\tfrac{p}{q}) \ |\ t \in {\Bbb N},\ t < t_0\}$ are  finite clearly. 
Observe that for large $t \in {\Bbb N}$, 
$int(\Delta) \cap S_{\alpha} (\tfrac{t}{1}) \cap S_{\gamma} (\tfrac{p}{q}) \ne \emptyset$  and 
$ int(\Delta) \cap S_{\beta} (\tfrac{t}{1}) \cap S_{\gamma} (\tfrac{p}{q}) \ne \emptyset$, 
but each point of these nonempty sets goes to a point of $\partial \Delta \cap S_{\beta}(\tfrac{p}{q})$  
as $t$ goes to $\infty$.  (See Figure~\ref{fig_appendix}(right).)    
This leads to the claim (3).  

The proof in the case $\tfrac{p}{q} \in (1,\infty)$ is similar. 

By a similar argument, one can prove the claim (3) when $\Omega$ is an $A$-face. 
\end{proof}

\begin{figure}[htbp]
\begin{center}
\includegraphics[width=3.5in]{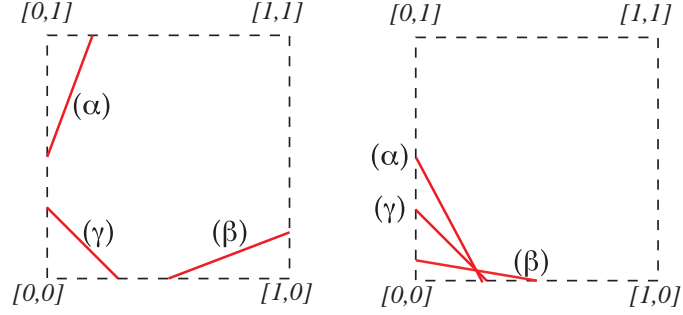}
\caption{(left) $\tfrac{t}{u} \in (- \infty, -1)$. (right) $\tfrac{t}{u} \in (-1 , \infty)$. 
[$(\alpha):= int(\Delta) \cap S_{\alpha}(\tfrac{t}{u})$, 
$(\beta):= int(\Delta) \cap S_{\beta}(\tfrac{t}{u})$, 
$(\gamma):= int(\Delta) \cap S_{\gamma}(\tfrac{p}{q})$.]} 
\label{fig_appendix}
\end{center}
\end{figure}

In contrast with Lemma~\ref{lem_PuncturedTorus}, 
we have

\begin{lem}
\label{lem_num-boundary}
Suppose that $\tfrac{p}{q} \ne 1$. 
 For a primitive fibered class $a \in S_{\alpha}(\tfrac{p}{q})$, 
the number of the boundary components of $F_{\overline{a}}$ 
for $\overline{a} \in  H_2(N_{\alpha}(\tfrac{p}{q}), \partial N_{\alpha}(\tfrac{p}{q}))$ 
is bounded by $2|p|+ 2 |q|$. 
\end{lem}

\begin{proof} 
Suppose that $a= (x,y,z) \in S_{\alpha}(\tfrac{p}{q})$ is a primitive fibered class. 
The number of the boundary components of $F_{\overline{a}} $ equals 
$\gcd(y, z+x) + \gcd(z, x+y)$.

We shall prove that $\gcd(y, z+x) \le |p|+|q|$. 
The  inequality $\gcd(z, x+y) \le |p|+|q|$ can be proved by  the same argument. 
Since $-px= q(y+z)$, there exists an integer $k$ such that $x= -qk$ and $y+z=pk$. 
Hence $z= pk-y$. 

\begin{claim} 
\label{claim_rela-prime}
$\gcd(y,k)= \gcd(z,k)=1$. 
\end{claim} 

\noindent
Proof of Claim~\ref{claim_rela-prime}. 
Suppose that $\gcd(y,k)= \ell >1$. 
Then 
$y= y' \ell$ and $k= k' \ell$ for some integers $y'$ and $k'$. 
One sees that 
$$\gcd(z,k) = \gcd(pk-y, k) = \gcd(pk' \ell - y' \ell, k' \ell) \ge \ell .$$
Thus $\ell$ is a factor of $y,z$ and $k$. 
Recall that $x= -qk$. 
This implies that $\ell$ is also a factor of $x$.  
Thus $\gcd(x,y,z) \ge \ell$. 
Since $(x,y,z)$ is a primitive class, this is a contradiction. 
This completes  the proof of Claim~\ref{claim_rela-prime}. 
\medskip
\\
One has $\gcd(y, z+x) = \gcd(y, (p-q)k-y) = \gcd(y, (p-q)k)$.   
Since $\gcd(y,k)=1$, one obtains 
$ \gcd(y, (p-q)k) = \gcd(y, p-q) \le p-q \le |p|+ |q|$. 
This completes the proof Lemma~\ref{lem_num-boundary}. 
\end{proof}

The following is an application of Lemmas~\ref{lem_LargeSet} and \ref{lem_num-boundary}.

\begin{prop}
\label{prop_small-MNE}
For each $r \in {\mathcal Hyp} \setminus \{1\}$, 
let $\Omega$ be a fibered face of $N(r)$ which enjoys $(*)$  in Theorem~\ref{thm_key}. 
Let $a_n \in S_{\beta}(r)$ be a primitive fibered class such that $\overline{a_n} \in int(C_{\Omega})$ for each $n \in {\Bbb N}$. 
Suppose that $\overline{a_n} $ converges projectively to a unique point of $int(\Omega)$ 
which achieves  $\min \mathrm{Ent}(N(r), \Omega)$. 
Then $\phi_{a_n} \in \mathcal{M}$ for large $n$.  
Moreover 
$$\min \mathrm{Ent}(N(r), \Omega)= \lim_{n \to \infty} \mathrm{Ent}(\widehat{\phi}_{a_n}).$$
\end{prop}

\begin{proof} 
The first assertion is immediate from Lemma~\ref{lem_LargeSet}. 
By the assumption of $\overline{a_n} $, one has 
$$\min \mathrm{Ent}(N(r), \Omega) = \lim_{n \to \infty} \mathrm{Ent}_r(\overline{a_n}) (= \lim_{n \to \infty} \mathrm{Ent}(\phi_{\overline{a_n}})).$$
Since $\phi_{a_n} \in \mathcal{M}$ for large $n$, it follows that 
$\lambda_r(\overline{a_n}) (= \lambda(\phi_{\overline{a_n}})) = \lambda(\widehat{\phi}_{a_n})$ for large $n$. 
By Lemma~\ref{lem_num-boundary}, the number of the boundary components of the minimal representative of $\overline{a_n}$ is bounded. 
Thus the normalized entropy $ \mathrm{Ent}(\widehat{\phi}_{a_n})$ of $\widehat{\phi}_{a_n}$ 
tends to $\min \mathrm{Ent}(N(r), \Omega)$ 
as $n$ tends to $\infty$. 
\end{proof}

\end{document}